\documentclass[10pt]{amsart}

\usepackage{amsmath}
\usepackage{amsthm}
\usepackage{amsopn}
\usepackage{amssymb}
\usepackage[all]{xy}
\usepackage{lscape,xcolor}
\usepackage{graphicx}
\usepackage{hyperref}
\usepackage{mathtools}
\usepackage{stmaryrd}
\usepackage{enumitem}

\setlist{itemsep=1em}

\allowdisplaybreaks[1]

\newcommand{\dotequiv}{\overset{\scriptstyle{\centerdot}}{\equiv}}

\newcommand{\mc}[1]{\mathcal{#1}}
\newcommand{\ul}[1]{\underline{#1}}
\newcommand{\mb}[1]{\mathbb{#1}}
\newcommand{\mr}[1]{\mathrm{#1}}
\newcommand{\mbf}[1]{\mathbf{#1}}
\newcommand{\mit}[1]{\mathit{#1}}

\newcommand{\bra}[1]{\langle #1 \rangle}

\newcommand{\td}[1]{\widetilde{#1}}

\newcommand{\ZZ}{\mathbb{Z}}
\newcommand{\RR}{\mathbb{R}}
\newcommand{\CC}{\mathbb{C}}
\newcommand{\QQ}{\mathbb{Q}}

\newcommand{\FF}{\mathbb{F}}
\newcommand{\GG}{\mathbb{G}}

\def \HF2{\mr{H}\FF_2}

\newcommand{\MU}{\mr{MU}}
\newcommand{\KU}{\mr{KU}}

\newcommand{\BP}{\mr{BP}}

\newcommand{\Sp}{\mr{Sp}}

 \newtheorem{thm}[equation]{Theorem}
 \newtheorem{cor}[equation]{Corollary}
 \newtheorem{lem}[equation]{Lemma}
 \newtheorem{prop}[equation]{Proposition}

 \newtheorem*{thm*}{Theorem}
 \newtheorem*{cor*}{Corollary}
 \newtheorem*{lem*}{Lemma}
 \newtheorem*{prop*}{Proposition}

 \theoremstyle{definition}
 \newtheorem{defn}[equation]{Definition}
 \newtheorem{ex}[equation]{Example}
 \newtheorem{exs}[equation]{Examples}
 \newtheorem{rmk}[equation]{Remark}

 \newtheorem{question}[equation]{Question}
 \newtheorem{conjecture}[equation]{Conjecture}
 \newtheorem{algorithm}[equation]{Algorithm}

\newtheorem*{defn*}{Definition}
\newtheorem*{ex*}{Example}
\newtheorem*{exs*}{Examples}
\newtheorem*{rmk*}{Remark}
\newtheorem*{claim*}{Claim}

\numberwithin{equation}{section}
\numberwithin{figure}{section}
\DeclareMathOperator{\Ext}{Ext}
\DeclareMathOperator{\Hom}{Hom}

\DeclareMathOperator{\im}{im}

\DeclareMathOperator{\Map}{Map}

\DeclareMathOperator{\Spf}{Spf}

\DeclareMathOperator{\Spec}{Spec}

\DeclareMathOperator{\Spc}{Spc}

\DeclareMathOperator{\Sub}{Sub}
\DeclareMathOperator{\Ind}{Ind}

\newcommand{\MFG}[1]{\mc{M}^{#1}_{\mit{fg}}}

\newcommand{\Nbar}{\overline{\mb{N}}}
\newcommand{\vv}{\bar{\mbf{v}}}
\newcommand{\eitem}{\stepcounter{equation}\item}
\newcommand{\eeta}{\pmb{\eta}}
\newcommand{\kappabar}{\bar{\kappa}}
\newcommand{\kkappabar}{\pmb{\bar{\kappa}}}

\title{Periodic phenomena in equivariant stable homotopy theory}
\author{Mark Behrens}
\author{Jack Carlisle}

\date{\today}

\begin{document}

\begin{abstract}
Building off of many recent advances in the subject by many different researchers, we describe a picture of $A$-equivariant  chromatic homotopy theory which mirrors the now classical non-equivariant picture of Morava, Miller-Ravenel-Wilson, and Devinatz-Hopkins-Smith, where $A$ is a finite abelian $p$-group.  Specifically, we review the structure of the Balmer spectrum of the category of $A$-spectra, and the work of Hausmann-Meier connecting this to $\MU_A$ and equivariant formal group laws.  Generalizing work of Bhattacharya-Guillou-Li, we introduce equivariant analogs of $v_n$-self maps, and generalizing work of Carrick and Balderrama, we introduce equivariant analogs of the chromatic tower, and give equivariant analogs of the smash product and chromatic convergence theorems.  The equivariant monochromatic theory is also discussed.  We explore computational examples of this theory in the case of $A = C_2$, where we connect equivariant chromatic theory with redshift phenomena in Mahowald invariants. 
\end{abstract}

\maketitle
\tableofcontents

\section{Introduction}

The nilpotence and periodicity theorems of Devinatz-Hopkins-Smith \cite{DHS}, \cite{HS} famously established that periodic phenomena in the stable homotopy category is detected by complex cobordism.  Through Quillen's Theorem \cite{Quillen}, this periodic structure is reflected in the structure of the moduli stack of formal groups.  The goal of this paper is to discuss how this picture generalizes when $\MU$ is replaced by the equivariant complex cobordism spectrum $\MU_G$.  We restrict attention to the case where $G = A$, a finite abelian group.

Significant work has already been done in this regard.  Balmer introduced the notion of the \emph{spectrum} $\Spc(\mc{C})$ of a tensor-triangulated category $\mc{C}$ \cite{BalmerSpc}, a topological space which in the case of the stable homotopy category of finite spectra, specializes to the content of the Hopkins-Smith Thick Subcategory Theorem \cite{HS}, classifying finite $p$-local spectra by \emph{types} in $\Nbar := \mb{N} \cup \{\infty\}$.

Balmer and Sanders \cite{BalmerSanders} showed that the underlying set of the Balmer spectrum $\Spc(\Sp^A_{\omega})$ of the stable homotopy category of finite genuine $A$-equivariant spectra can be computed in terms of the non-equivariant case (see also the insightful notes of Strickland \cite{Stricklandslides}).  The topology on this Balmer spectrum was determined by Barthel, Hausmann, Naumann, Nikolaus, Noel and Stapleton \cite{sixauthor}.  The \emph{types} of finite $A$-spectra are parameterized by the \emph{admissible closed subsets} of $\Spc(\Sp^A_{\omega})$.

Barthel, Greenlees, and Hausmann \cite{BGH}, in addition to extending these Balmer spectrum computations to the case of compact abelian Lie groups, also show that an equivariant analog of the Devinatz-Hopkins-Smith Nilpotence Theorem \cite{DHS} can be deduced from the non-equivariant counterpart.  They prove that the equivariant complex cobordism spectrum $\MU_A$ detects nilpotence in $\Sp^A$. 

Cole, Greenlees, and Kriz \cite{CGK} introduced the notion of \emph{equivariant formal group laws}, and conjectured that $\pi^A_*\MU_A$ classifies such, generalizing Quillen's Theorem \cite{Quillen}. These equivariant formal group laws were further studied by Strickland \cite{Multicurves}, who introduced equivariant versions of Morava K-theory and Morava E-theory.  The Cole-Greenlees-Kriz conjecture was verified recently by Hausmann \cite{Hausmann}.  Hausmann and Meier \cite{HM} expanded on this picture, giving an explicit homeomorphism between the invariant prime ideal spectrum of $\pi^A_*\MU_A$ and the Balmer spectrum of $\Sp^A_{\omega}$, and relating this to an appropriate height stratification on the moduli stack of equivariant formal group laws.

Examples of periodic self maps in the case where $A = C_2$ were studied by Crabb \cite{Crabb}, Quigley \cite{Quigley} and Bhattacharya-Guillou-Li \cite{BGL}, \cite{BGL2}.  Finite smashing localizations were studied by Hill \cite{Hill}. Periodic smashing localizations (generalizing the $E(n)$-localizations in the non-equivariant context) in the case of $A$ a cyclic $p$-group were introduced by Carrick \cite{Carrick} and in the case of $A$ elementary abelian by Balderrama \cite{BalderramaPower}.  The $RO(C_2)$-graded homotopy groups of the $\KU_{C_2}/2$-local sphere were computed by Balderrama \cite{Balderrama} and computations of the integer graded homotopy groups of the $\KU_{G}$-local sphere ($G$ a finite $p$-group, $p$ odd) were completed by by Bonventre, Carawan, Field, Guillou, Mehrle, and Stapleton \cite{BGS}, \cite{fiveauthor}.

The goal of this paper is to assemble these various threads into a cohesive picture, which relates the geometry of equivariant formal groups to periodic localizations of $\Sp^A$ and connects to explicit periodic phenomena in the equivariant stable stems.

To simplify the discussion, we localize everything at a fixed prime $p$, and assume that $A$ is a finite abelian $p$-group.  The types of finite $p$-local $A$-spectra are parameterized by \emph{admissible type functions}.  These are functions
$$ \ul{n}: \Sub(A) \to \Nbar $$
satisfying
$$ \ul{n}(B) \le \ul{n}(C) + \mr{rk}_p(C/B) $$
for $B \le C \le A$.
Here $\Sub(A)$ denotes the set of subgroups of $A$, and $\mr{rk}_p$ denotes the $p$-rank.  A finite $p$-local spectrum is \emph{type $\ul{n}$} if for all $B \le A$, 
$$ X^{\Phi B} \: \text{is type $\ul{n}(B)$}. $$
Here, $(-)^{\Phi B}$ denotes the geometric fixed points.

A finite $p$-local $A$-spectrum $X$ of type $\ul{n}$ is eligible to admit \emph{$v^S_{\ul{n}}$}-self maps 
$$ v: \Sigma^\gamma X \to X $$
(where $S \subseteq \Sub(A)$ and $\gamma \in RO(A)$).  These are self-maps with the property that 
$$ v^{\Phi B} = 
\begin{cases}
\text{a $v_{\ul{n}(B)}$-self map}, & B \in S, \\
\text{nilpotent}, & B \not\in S.
\end{cases}
$$

We prove many analogs of theorems of Hopkins and Smith \cite{HS} concerning the properties of such $v^S_{\ul{n}}$-self maps, including their asymptotic uniqueness and centrality (Corollary~\ref{cor:vnselfmap}).  We are only able to prove a very limited analog of the Hopkins-Smith Periodicity Theorem (Theorem~\ref{thm:vnB}) which establishes the existence of $v^B_n$-self maps (the case where $S = \{B\}$) from the non-equivariant periodicity theorem.\footnote{Forthcoming work of Burklund, Hausmann, Levy and Meier addresses the existence of more general $v_{\underline{n}}$-self maps.}

Strickland \cite{Multicurves} defined equivariant Morava $K$-theories $K(B,n)$ ($B \le A$). These equivariant Morava K-theories are in bijective correspondence with the primes of the Balmer spectrum.  We extend Strickland's construction to define equivariant Johnson-Wilson theories $E(\ul{n})$ ($\ul{n}$ admissible) which are in bijective correspondence with admissible open subspaces of the Balmer spectrum.  One of our main results is the Equivariant Smash Product Theorem (Theorem~\ref{thm:smashing}) which states that for $\ul{n}$ admissible, $E(\ul{n})$-localization is smashing, and 
$$ (X_{E(\ul{n})})^{\Phi B} \simeq (X^{\Phi B})_{E(\ul{n}(B))}. $$
This generalizes the work of Carrick \cite{Carrick}, who proved this in the case where $A$ is a cyclic $p$-group, and Balderrama \cite{BalderramaPower}, in the case where $A$ is an elementary abelian $p$-group and $\ul{n}(B) = n-\mr{rk}_p(B)$.

As $\ul{n}$ varies over the poset of admissible type functions, the localizations
$$ X_{E(\ul{n})} $$
form an equivariant chromatic tower.  We prove a chromatic fracture theorem (Theorem~\ref{thm:fracture}), which establishes that these $E(\ul{n})$-localizations can be assembled inductively from the $K(B,n)$-localizations.  These ``monochromatic layers'' can be accessed quite directly from the $\mc{G}_n$-homotopy fixed points of the $K(n)$-local Borel $A/B$-spectrum
$$ F(E(A/B)_+, E_n) \in \Sp^{B(A/B)}_{K(n)} $$
(see Remark~\ref{rmk:monochromatic}).
Here $E_n$ is the $n$th Morava $E$-theory spectrum, and $\mc{G}_n$ is the corresponding Morava stabilizer group.  The theories described above are essentially the equivariant Morava E-theories introduced by Strickland in \cite{Multicurves}.

Balderrama proved the Equivariant Chromatic Convergence Theorem in the case where $A$ is an elementary abelian $p$-group.  The proof is quite subtle, and relies on both the Equivariant Smash Product Theorem, and the convergence of the non-equivariant chromatic tower for $\Sigma^\infty_+ BA$, which can be deduced by combining the results of Johnson-Wilson \cite{JohnsonWilson} and Barthel \cite{Barthel}.  Balderrama directed us to an unpublished proof of chromatic convergence for $BA$ for general finite abelian $p$-groups due to J.~Hahn (see \cite{BalderramaHahn}) and explained how our more general Equivariant Smash Product Theorem allows us to prove equivariant chromatic convergence in the case where $A$ is a finite abelian $p$-group.  Carrick independently explained to the authors how one could deduce equivariant chromatic convergence using Hahn's result.  We gratefully include these arguments here. 

We end this paper with some explicit computations and examples of this theory in the case where $A = C_2$.  This is the only group where any significant computations of the equivariant stable stems have been carried out (see \cite{AI}, \cite{BI}, \cite{BGI}, \cite{GI}).  We identify explicit elements of $\pi_\star^{C_2}\MU_{C_2}$ which detect $v_{\ul{n}}$-self maps.  We construct explicit examples of $v_{\ul{n}}$-self maps, and compare with those studied by Crabb \cite{Crabb}, Quigley \cite{Quigley} and Bhattacharya-Guillou-Li \cite{BGL}.\footnote{One of our self-maps was independently also discovered by Balderrama, Hou, and Zhang \cite{BHZ}, who have many more examples in their forthcoming work.}  We explain how the conjectural redshift property of the Mahowald invariant is related to $C_2$-equivariant $v_{\ul{n}}$-periodicity.

\subsection*{Organization of the paper}

In Section~\ref{sec:Balmer} we review the structure of the Balmer spectrum
$$ \Spc(A) := \Spc(\Sp^A_\omega) $$
in the case where $A$ is finite abelian.  The Balmer spectrum describes the shape of the chromatic structure of $\Sp^A$.  We explore how the topology of the Balmer spectrum is related to Greenlees-May isotropy separation, and discuss the resulting decomposition of the equivariant stable homotopy category into Borel strata following Abram-Kriz \cite{AbramKriz} and Ayala-Mazel-Gee-Rozenblyum \cite{AMGR}.

In Section~\ref{sec:FGL} we review the Cole-Greenlees-Kriz \cite{CGK} notion of an equivariant formal group, with an emphasis on the coordinate-free perspective of Strickland \cite{Multicurves} and Hausmann-Meier \cite{HM}.  We review how inverting and modding out euler classes is related to the structure of the associated equivariant formal law.

These equivariant formal groups are related to homotopy theory in Section~\ref{sec:MUA} by Hausmann's equivariant Quillen theorem, which states that $\MU_A$ carries the universal $A$-equivariant formal group law.  We review the Abram-Kriz isotropy decomposition of $\MU_A$, and their associated presentation of $\pi_*^A\MU_A$ \cite{AbramKriz}.  We end this section with a discussion of the $p$-local theory $\BP_A$, following Wisdom \cite{Wisdom}.

\emph{Starting in Section~\ref{sec:HM}, we assume for the rest of the paper that $A$ is a finite abelian $p$-group, and restrict our work to the $p$-local context.}  We flesh out the connection between $\MU_A$ and equivariant chromatic theory by  introducing various notions of \emph{height} one can associate to an equivariant formal group.  Associated to these heights are the irreducible closed substacks
$$ V_{(B,n)} \subseteq (\MFG{A})_{(p)} $$ 
of the moduli stack of $A$-equivariant formal groups studied by Hausmann-Meier \cite{HM}.  Hausmann and Meier showed that the space of irreducible closed substacks of $(\MFG{A})_{(p)}$ is homeomorphic to the Balmer spectrum $\Spc_{(p)}(A)$ of $(\Sp^A_\omega)_{(p)}$.  We introduce the notion of a \emph{$v_{\ul{n}}$-generator} in $\pi_*^A(\MU_A)_{(p)}$ associated to a \emph{height function}
$$ \ul{n}: \Sub(A) \to \Nbar_- := \mb{N} \cup \{ \infty \} \cup \{ -1 \}. $$
We end this section by discussing how Hausmann-Meier construct $v_{\ul{n}}$-generators, and how their theory relates such generators to the topology of the Balmer spectrum.  In particular, we introduce the notion of an \emph{admissible height function}, and explain how Hausmann-Meier's theory implies that $v_{\ul{n}}$-generators only exist for admissible height functions $\ul{n}$.

We move on to a discussion of equivariant non-nilpotent self-maps in Section~\ref{sec:periodicity}.  We review the observation of Barthel-Greenlees-Hausmann \cite{BGH} that the equivariant nilpotence theorem ($\MU_A$ detects nilpotence) follows formally from the non-equivariant nilpotence theorem.  Following the work of Bhattacharya-Guillou-Li in the case of $A = C_2$, we introduce the notion of a $v_{\ul{n}}$-self map for a height function $\ul{n}$.  We explain the Barthel-Greenless-Hausmann notion of a type $\ul{n}$ complex, where $\ul{n}$ is an admissible \emph{type function}, and the constraints type places on the heights of self-maps.  Following Hopkins and Smith \cite{HS}, we deduce the asymptotic uniqueness and centrality of $v_{\ul{n}}$-self maps from the equivariant nilpotence theorem.  We explain how the equivariant complexes which admit $v_{\ul{n}}$-self maps form a thick tensor ideal, and use this to deduce the existence of a basic class of such self-maps.

Chromatic localizations are studied in Section~\ref{sec:loc}.  We define the chromatic tower of $E(\ul{n})$-localizations associated to admissible height functions $\ul{n}$, and prove that these localizations are smashing, generalizing results of Carrick \cite{Carrick} and Balderrama \cite{BalderramaPower}.  The relationship between these $E(\ul{n})$-localizations, and their associated finite variants, is briefly discussed.  We deduce a chromatic fracture theorem, which builds the chromatic tower out of the $K(n,B)$-localizations, and we explain the how 
the equivariant homotopy groups of the $K(n,B)$-local sphere can be computed from the homotopy fixed points of the Morava $E$-theory of Thom spectra on classifying spaces.  As explained to us by Balderrama and Carrick, the chromatic convergence theorem is deduced from a theorem of Hahn \cite{BalderramaHahn}.

The computational examples in the case of $A = C_2$, as well as the connection to Mahowald red-shift, is the subject of Section~\ref{sec:C2}.

\subsection*{Acknowledgments}
The authors would like to thank Bert Guillou and Dan Isaksen for sharing so much of their computational knowledge and providing helpful comments, as well as Guoqi Yan, who has been a constant resource throughout this project.  The authors are very grateful Robert Burklund, Markus Hausmann, Ishan Levy, and Lennart Meier, for communicating their recent work on the Equivariant Periodicity Theorem. They would like to express their special thanks to William Balderrama, for providing corrections, explaining his recent work on equivariant chromatic homotopy theory, especially his joint work with Hou and Zhang on self-maps, and for explaining to the authors how the Smash Product Theorem of this paper combines with Hahn's unpublished work to extend his proof of the Chromatic Convergence Theorem. The authors similarly would like to thank Christian Carrick for his comments, and in particular for his alternative explanation of how the Equivariant Chromatic Convergence Theorem can be deduced from the Smash Product Theorem, the tom Dieck splitting, and Hahn's result.  

The authors are also grateful to Shangjie Zhang, for pointing out to the authors the need of Definition~\ref{defn:dir}, to Tobias Barthel, for explaining how finite smashing localizations naturally arise from the Balmer spectrum, and to Naren Ezhilmaran, for helpful corrections. When this paper was first written, the authors were surprisingly unaware of Strickland's extensive work on equivariant formal group laws and his definition of the associated equivariant Morava $K$- and $E$-theories, which coincided exactly with ours, but were introduced over a decade earlier \cite{Multicurves}.  The authors are grateful to Noah Wisdom and Gabrielle Li for pointing this out to them. The first author benefited greatly from his discussions with J.D. Quigley on equivariant periodicity and Nikolai Konovalov on families and equivariant formal group laws.  The authors benefited greatly from the suggestions and comments of the referee, especially for suggesting Proposition~\ref{prop:completions}.  The first author was supported by NSF grant DMS-2005476 and the second author was supported by an NSF RTG postdoc funded by DMS-2135884.

\subsection*{Notation and conventions}

\begingroup
\begin{align*}
A & = \text{a finite abelian group} \\
e & = \text{the trivial group} \\
A^\vee & = \Hom(A,\CC^\times), \: \text{the group of characters of $A$} \\
K_\alpha & = \ker(\alpha: A \to \CC^\times), \: \text{the kernel of a character $\alpha \in A^\vee$} \\
R(A) & = \ZZ\{A^\vee\}, \: \text{the complex representation ring of $A$} \\
RO(A) & = \text{the real representation ring of $A$} \\
|\gamma| & \in \ZZ, \: \text{the virtual dimension of $\gamma \in RO(A)$} \\
\mc{U_A} & = \text{complete $A$-universe} \\
\Sub(A) & = \text{the set of subgroups of $A$, given a topology where families are closed} \\
\Sp & = \text{the $\infty$-category of spectra} \\
\Sp^A & = \text{the $\infty$-category of genuine $A$-spectra} \\
\Sp^{BA} & = \text{the $\infty$-category of Borel $A$-spectra} \\
\Spc(A) & = \text{the Balmer spectrum of $\Sp^A_{\omega}$} \\
\Spc_{(p)}(A) & = \text{the Balmer spectrum of $\Sp^A_{(p),\omega}$} \\
\pi^{B}_{\star}E & = [S^\star,E]^B, \: \text{for $E \in \Sp^A$, $B \le A$, and $\star \in RO(B)$} \\
\pi^B_*E & = \pi^B_{\star} E \: \text{for $\star$ in the subring $\ZZ \subseteq RO(B)$} \\
a_\gamma & = \text{the inclusion $S^0 \hookrightarrow S^\gamma$, for an $A$-representation $\gamma$,} \\
& \qquad \text{regarded as an element of $\pi^A_{-\gamma}S$} \\
E^{\wedge}_{\mc{F}} & = F(E\mc{F}_+, E), \: \text{the completion of $E \in \Sp^A$ at a family $\mc{F}$} \\
E^h & = E^{\wedge}_{\{e\}} = F(EA_+, E), \: \text{the Borel completion (completion at the trivial family)} \\
E[\mc{F}^{-1}] & = E \wedge \widetilde{E\mc{F}}, \: \text{the localization of $E \in \Sp^A$ away from a family $\mc{F}$} \\
\mc{F}_{\subseteq B} & = \text{the family of subgroups of $A$ which are contained in $B$} \\
\mc{F}_{B \nsubseteq} & = \text{the family of subgroups of $A$ with do not contain $B$} \\
\Phi^B E & = (E[\mc{F}^{-1}_{B \nsubseteq}])^B \in \Sp^{A/B}, \: \text{the $B$-geometric fixed points of $E \in \Sp^A$,} \\
& \qquad \text{regarded as an $A/B$-spectrum} \\
E^{\Phi B} & = (\Phi^B E)^e \in \Sp, \: \text{the $B$-geometric fixed points of $E \in \Sp^A$,} \\
  & \qquad \text{regarded as a spectrum} \\
  \doteq & \qquad \text{equal up to a unit} \\
\dotequiv & \qquad \text{equivalent up to a unit} \\
\Nbar & = \mb{N} \cup \{ \infty \} \\
\Nbar_- & = \mb{N} \cup \{-1\} \cup \{ \infty\}
\end{align*}
\endgroup
For $B \le A$ there are adjoint functors
\begin{align*}
 i_B^*: \Sp^{A} & \leftrightarrows \Sp^B : \Ind^A_B  \\
 q_B^*: \Sp^{A/B} & \leftrightarrows \Sp^A: (-)^B
\end{align*}
where
\begin{align*}
i_B^* & = \text{restriction along the inclusion $i_B: B \hookrightarrow A$}, \\
q_{B}^* & = \text{restriction along the quotient map $q_{B}: A \twoheadrightarrow A/B$ (a.k.a. \emph{inflation})}, \\
\Ind^A_B & = \Map_B(A, -) \simeq A_+\wedge_B -, \: \text{(co)induction along $i_B$}, \\
(-)^B & = \text{$B$-fixed points}.
\end{align*}
In particular $\Ind_B^A$ is also left adjoint to $i_B^*$.\footnote{Inflation does not have a left adjoint (see \cite{Sanders}).}  As a special case, $q_A^*$ is the functor
$$ (-)_{\mr{triv}}: \Sp \to \Sp^A $$
which endows a spectrum with the ``trivial $A$-action.'' For a general group homomorphism $f: A \to A'$, we denote the associated restriction functor as
$$ f^*: \Sp^{A'} \to \Sp^A. $$

If $\mc{C}$ is a presentably symmetric monoidal stable $\infty$-category, then associated to every object $E \in \mc{C}$ is a Bousfield localization
$$ X \to X_E $$
where $X_E$ is $E$-local.  We denote the associated localization of $\mc{C}$ which inverts the $E$-equivalences by $\mc{C}_E$.

We shall use the term \emph{homotopy ring spectrum} to refer to an equivariant or non-equivariant spectrum which is an associative monoid in the homotopy category. 
All stacks are regarded with respect to the fpqc topology.  All formal groups are implicitly $1$-dimensional.

\section{The Balmer spectrum and families}\label{sec:Balmer}

\subsection*{The Balmer spectrum of $\pmb{\Sp^A_\omega}$} 

The $\infty$-category of $A$-spectra has a structure that resembles the derived category of quasi-coherent sheaves on a scheme, which is made precise by the notion of the Balmer spectrum.

Let $K_p(n)$ denote the $p$-primary $n$th Morava K-theory spectrum (where we define $K_p(0) = H\QQ$ and $K_p(\infty) = H\FF_p$).  Recall \cite{BalmerSanders}, \cite{sixauthor}, \cite{BalderramaKuhn} that the Balmer spectrum of $\Sp_\omega^A$ is given by
$$ \mr{Spc}(\Sp_\omega^A) = \{ \mc{P}_{(B,p,n)} \: : \: B \le A, \: \text{$p$ prime}, \: 0 \le n \le \infty \}/(\mc{P}_{(B,p,0)} \sim \mc{P}_{(B,q,0)}) $$
where\footnote{Note that our notation differs slightly from that of \cite{BalmerSanders}, who define $\mc{P}_{(B,p,n)}$ to be the $K_p(\pmb{n-1})$ acyclic finite $A$-spectra for $\pmb{1 \le n} \le \infty$.}
$$ \mc{P}_{(B,p,n)} := \{ X \in \Sp^A_\omega \: : \: K_p(n) \wedge X^{\Phi B} \simeq \ast \}. $$
To streamline notation, we will use the abbreviation
$$ \Spc(A) := \Spc(\Sp^A_\omega). $$
There is also a $p$-local variant, which only involves $\mc{P}_{(B,p,n)}$ for a fixed prime $p$:
$$ \Spc_{(p)}(A) := \Spc(\Sp^A_{(p),\omega}). $$ 
Because $\mc{P}_{(B,p,0)}$ does not depend on $p$, we will sometimes denote 
$$ \mc{P}_{(B,0)} := \mc{P}_{(B,p,0)}. $$
The inclusions of these primes is generated by the inclusions
$$ \mc{P}_{(B,p,n+1)} \subset \mc{P}_{(B,p,n)}, \quad 0 \le n $$
and the inclusions
$$ \mc{P}_{(B,p,n+r)} \subset \mc{P}_{(C,p,n)},  \quad 0 \le n \le \infty, \: B < C, \:\text{$C/B$ is a $p$-group of rank $r$}. $$
Figure~\ref{fig:spcspa} gives a visual depiction of the inclusion structure of $\Spc(A)$.  Note that in this lattice diagram, if $\mc{P}$ lies above $\mc{Q}$, that means that $\mc{P} \subset \mc{Q}$.  We orient things this way both to conform to what has become the standard way to graphically represent the Balmer spectrum, and to mesh with our discussion of the chromatic tower.

\begin{rmk}
Note that if one is comparing the Balmer spectrum with the Zariski spectrum of a ring, it is important to note that the correspondence between the prime ideals of the tt-category of perfect complexes and the prime ideals of a ring reverses inclusions \cite[Sec.~4]{Balmer}.  
\end{rmk}

\begin{figure}
\centering
\includegraphics[width=\linewidth]{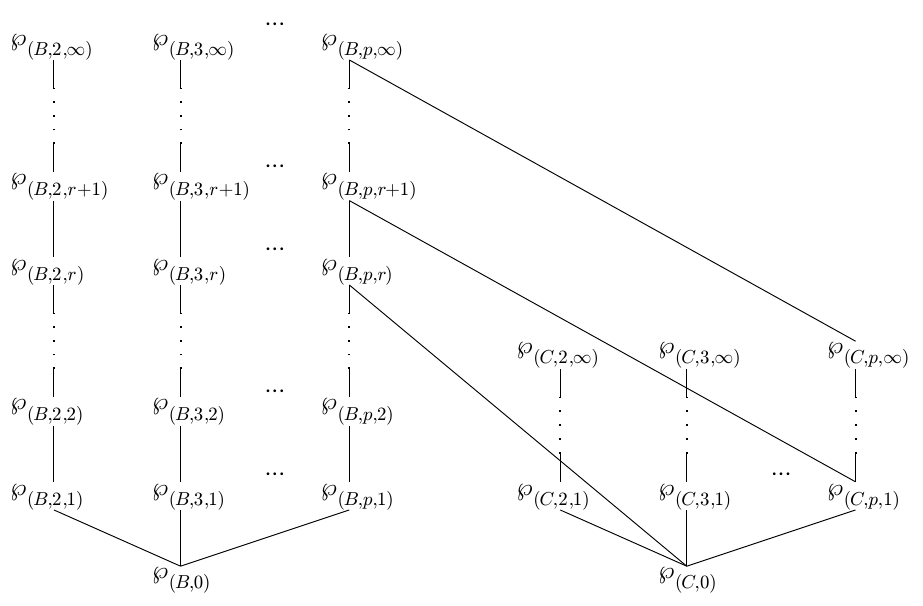}
\caption{Structure of $\Spc(A)$ in the vicinity of $B < C \le A$ with $C/B$ a $p$-group of rank $r$.}
\label{fig:spcspa}
\end{figure}

The topology on $\Spc(A)$ is determined from these inclusions as follows.
The irreducible closed subsets of $\mr{Spc}(A)$ are of the form 
$$ V(\mc{P}) = \{ \mc{Q} \in \Spc(A) \: : \: \mc{Q} \subseteq \mc{P} \}. $$ 
A general closed subset is a finite union of $V(\mc{P})$'s \cite[Cor.~8.19]{BalmerSanders}.  

Conceptually, the topology on $\Spc(A)$ has a basis of closed subsets given by \emph{supports} of objects $X \in \Sp^A_\omega$ \cite[Defn~2.6]{Balmer}:
\begin{equation}\label{eq:supp}
\begin{split}
 \mr{supp}(X) & := \{ \mc{P}_{(B,p,n)} \: : \: X \not\in \mc{P}_{(B,p,n)} \} \\
 & = \{ \mc{P}_{(B,p,n)} \: : \: K_p(n) \wedge X^{\Phi B} \not\simeq \ast \} \subseteq \Spc(A).
 \end{split}  
\end{equation}

The structure and topology of $\Spc_{(p)}(A)$ is much easier to understand than that of $\Spc(A)$.  
For a fixed prime $p$, we define for $0 \le n \le \infty$
$$ \mc{P}_{(B,n)} = \{ X \in \Sp^A_{(p),\omega} \: : \: K_p(n) \wedge X^{\Phi B} = 0\} $$
and we have
$$ \Spc_{(p)}(A) = \{ \mc{P}_{(B,n)} \: : \: B \le A, \: 0 \le n \le \infty \}.
$$
The lattice structure and topology is obtained by simply regarding $\Spc_{(p)}(A)$ as a subspace of $\Spc(A)$ via the map
\begin{align*}
\Spc_{(p)}(A) & \hookrightarrow \Spc(A) \\
(B,n) & \mapsto (B,p,n)
\end{align*} 

With respect to the lattice diagram of $\Spc_{(p)}$, the closed subsets are precisely those subsets which are upwards closed, and the open subsets are precisely those subsets which are downwards closed.  Figure~\ref{fig:SpcCp} depicts the lattice $\Spc_{(p)}(A)$ in the example where $A = C_p$, and shows some typical examples of open and closed subsets therein.

\begin{figure}
\centering
\includegraphics[height=.3\textheight]{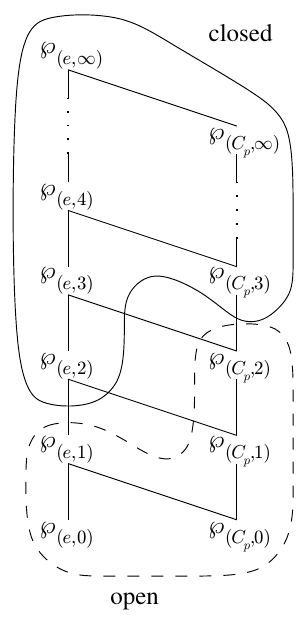}
\caption{The lattice structure of $\Spc_{(p)}(C_p)$, and some examples of open and closed subsets.}
\label{fig:SpcCp}
\end{figure}

\subsection*{Families of subgroups}

The Balmer spectrum of $\Sp^A_\omega$ naturally maps to the subgroup lattice of $A$
\begin{equation}\label{eq:spctosub}
\begin{split}
 \Spc(A) & \to \mr{Sub}(A) \\
 \mc{P}_{(B,p,n)} & \mapsto B.
\end{split}
\end{equation}
If we give $\mr{Sub}(A)$ the topology where the closed subsets are the \emph{families} of subgroups $\mc{F}$, then this map is continuous.\footnote{Note this topology is unrelated to the topology of $\Sub(A)$ used in \cite{BGH}, where $A$ is compact Lie.}
Since the closed subsets of $\mr{Sub}(A)$ are given by families, and therefore characterized by being closed under taking subgroups
$$ B \in \mc{F}, C \le B \Rightarrow C \in \mc{F}, $$
the open subsets of $\mr{Sub}(A)$ are given by complements of families  
$\mc{F}^c$, and are therefore characterized by being closed under enlargement:
$$ B \in \mc{F}^c, B \le C \Rightarrow C \in \mc{F}^c. $$
In order to make the picture compatible with Figure~\ref{fig:spcspa}, we encourage the reader to draw the subgroup lattice ``upside down'' so that the trivial subgroup lies at the top and the full group lies on the bottom.  Figure~\ref{fig:suba} depicts the example of $\mr{Sub}(C_2 \times C_2)$ where
$$ C_2 \times C_2 = \bra{a,b}. $$

\begin{figure}
\centering
\includegraphics[width=0.5\linewidth]{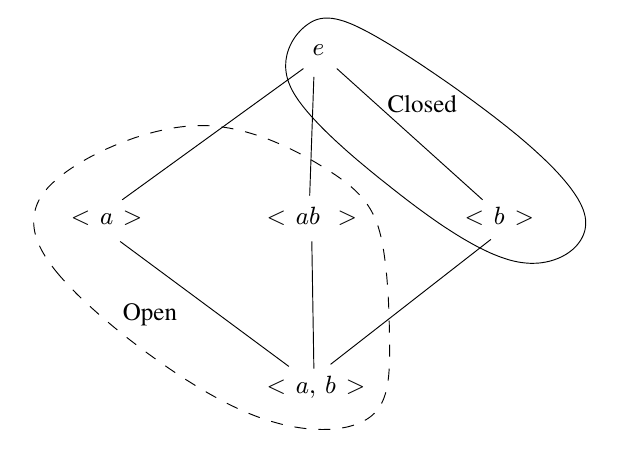}
\caption{Closed and open subsets of $\mr{Sub}(C_2 \times C_2)$.}
\label{fig:suba}
\end{figure}

The most important families to consider are the families which, for each $C \le A$, are given by
\begin{align*}
\mc{F}_{\subseteq C} & := \{ B \le A \: : \: B \le C \}, \\
\mc{F}_{C \nsubseteq} & := \{ B \le A \: : \: C \nsubseteq B \}. 
\end{align*}
Conceptually: 
\begin{enumerate}[label=(\theequation)]
\eitem The closed subsets $\mc{F}_{\subseteq C}$ consist of all of the subgroups lying \emph{above} $C$ (including $C$ itself) in the subgroup lattice of $A$, and every closed subset $\mc{F}$ of $\mr{Sub}(A)$ satisfies
$$ \mc{F} = \bigcup_{C \in \mc{F}} \mc{F}_{\subseteq C}. $$

\eitem The open subsets $\mc{F}^c_{C \nsubseteq}$ consist of all of the subgroups lying \emph{below} $C$ (including $C$ itself) in the subgroup lattice of $A$, and every open subset $\mc{F}^c$ of $\mr{Sub}(A)$ satisfies
$$ \mc{F}^c = \bigcup_{C \in \mc{F}^c} \mc{F}^c_{C \nsubseteq}. $$
\end{enumerate}

In particular, the inverse image of a family $\mc{F}$ under (\ref{eq:spctosub}) gives a closed subset 
\begin{equation}\label{eq:V(F)}
 V(\mc{F}) := \{ \mc{P}_{(B,p,n)} \: : \: B \in \mc{F} \} \subseteq \Spc(A). 
 \end{equation}
We clearly have
\begin{align*}
 V(\mc{F} \cup \mc{F}') & = V(\mc{F}) \cup V(\mc{F}'), \\
 V(\mc{F} \cap \mc{F}') & = V(\mc{F}) \cap V(\mc{F}'). \\
\end{align*}  

The following proposition (c.f. \cite{BalmerSanders}) gives a conceptual interpretation of these closed subsets.

\begin{prop}
For a family $\mc{F}$ of subgroups of $A$, we have
$$
V(\mc{F}) = \bigcup_{C \in \mc{F}} \mr{supp}(\Sigma^{\infty}_+ A/C).  
$$
\end{prop}

\begin{proof}
First note that since for $B,C \le A$ we have
$$
(A/C)^B =
\begin{cases}
\emptyset & B \nsubseteq C, \\
A/C & B \subseteq C
\end{cases}
$$
it follows that
$$ V(\mc{F}_{\subseteq C}) = \mr{supp}(\Sigma^\infty_+ A/C). $$
We therefore have
\begin{align*}
V(\mc{F}) & = V\left(\bigcup_{C \in \mc{F}} \mc{F}_{\subseteq C}\right) \\
& = \bigcup_{C \in \mc{F}} V( \mc{F}_{\subseteq C}) \\
& = \bigcup_{C \in \mc{F}} \mr{supp}(\Sigma^{\infty}_+ A/C).
\end{align*}
\end{proof}

We let
\begin{equation}\label{eq:U(Fc)}
 U(\mc{F}^c) := V(\mc{F})^c
\end{equation}
denote the associated open subset of $\Spc(A)$.  We have
\begin{align*}
 U((\mc{F} \cup \mc{F}')^c) & = U(\mc{F}^c) \cap U((\mc{F}')^c), \\
 U((\mc{F} \cap \mc{F}')^c) & = U(\mc{F}^c) \cup U((\mc{F}')^c). \\
\end{align*}  

\subsection*{Localization and completion functors associated to families}

Given a family $\mc{F}$ of subgroups of $A$, there is a unique (up to equivariant equivalence) $A$-space $E\mc{F}$ such that for $B \in \Sub(A)$ we have
$$ E\mc{F}^B \simeq 
\begin{cases}
\ast, & B \in \mc{F}, \\
\emptyset, & B \not\in \mc{F}. 
\end{cases}
$$
The cofiber sequence
$$ E\mc{F}_+ \to S^0 \to \widetilde{E\mc{F}} $$
and
associated completion and localization functors on $\Sp^A$
\begin{align*}
E & \mapsto E^{\wedge}_\mc{F} := F(E\mc{F}_+, E), \\
E & \mapsto E[\mc{F}^{-1}] := E \wedge \widetilde{E\mc{F}}
\end{align*} 
have been studied extensively, see \cite{GreenleesMay} and \cite{MNN}.

\begin{enumerate}[label = (\theequation)]
\eitem The $\mc{F}$-completion functor
$$ E \mapsto E^{\wedge}_\mc{F} $$
is localization with respect to the $\mc{F}$-\emph{equivalences}, the maps in $\Sp^A$ which are equivalences on $B$-fixed points for all $B \in \mc{F}$.  Equivalently, an $\mc{F}$-equivalence is a map which induces an equivalence on $B$-geometric fixed points for all $B \in \mc{F}$.  One may think of $E^{\wedge}_\mc{F}$ as the completion of $E$ at the closed subset $V(\mc{F}) \subseteq \Spc(A)$:
$$ E^{\wedge}_{\mc{F}} \sim E^{\wedge}_{V(\mc{F})}. $$
For $B \le A$, we have
$$ (E^{\wedge}_{\{e\}})^B \simeq E^{hB} $$
and the restriction $i^*_B$ factors through the $\mc{F}_{\subseteq B}$-completion
$$
\xymatrix{
\Sp^A \ar[rr]^{i^*_B} \ar[dr] && \Sp^B \\
& (\Sp^A)^{\wedge}_{\mc{F}_{\subseteq B}} \ar@{.>}[ur]_{(i^*_B)^{\wedge}}
}
$$
where the factorization $(i^*_B)^{\wedge}$ is conservative.\footnote{In fact, there is an equivalence of $\infty$-categories $(\Sp^A)^{\wedge}_{\mc{F}_{\subseteq B}} \simeq (\Sp^B)^{hA/B}$ \cite[Cor.~6.34]{MNN}}  The induction is $\mc{F}_{\subseteq B}$-complete --- for $E \in \Sp^B$ we have
$$
\Ind_B^A E \in (\Sp^{A})^{\wedge}_{\mc{F}_{\subseteq B}} \subseteq \Sp^A. 
$$
In fact, if we let $E_{A/B_+}$ denote Bousfield localization with respect to the ring spectrum $\Sigma^\infty A/B_+ \simeq D(A/B_+) \in \Sp^A$, then we have \cite[Defn.~6.2, Prop.~6.6]{MNN}
$$ E^{\wedge}_{\mc{F}_{\subseteq B}} \simeq E_{A/B_+}. $$
More generally, completion at $\mc{F}$ is Bousfield localization with respect to $\bigvee_{B \in \mc{F}} A/B_+$.  In the special case where $B = e$, we have
$$ E^{\wedge}_{\{e\}} = F(EA_+, E) = E^h, $$
the Borel completion of $E$, and there is an equivalence of $\infty$-categories
$$ (\Sp^A)^{\wedge}_{\{ e\}} \simeq \Sp^{BA}. $$
\label{eq:Fcompletion}

\eitem
The $\mc{F}^{-1}$-localization functor 
$$ E \mapsto E[\mc{F}^{-1}] $$
is localization with respect to the $\mc{F}^{-1}$-\emph{equivalences}, the maps in $\Sp^A$ which are equivalences on $B$-geometric fixed points for all $B \in \mc{F}^c$.  One may think of $E[\mc{F}^{-1}]$ as the restriction of $E$ to the open subset $U(\mc{F}^c) \subseteq \Spc(A)$
$$ E[\mc{F}^{-1}] \sim E\bigg\vert_{U(\mc{F}^c)}. $$
For $B \le A$, we have 
$$ E[\mc{F}_{B \nsubseteq}^{-1}]^{B} \simeq E^{\Phi B} $$
and the geometric fixed point functor factors through the $\mc{F}_{B \nsubseteq}$-localization
$$
\xymatrix{
\Sp^A \ar[rr]^{\Phi^B} \ar[dr] && \Sp^{A/B} \\
& \Sp^A[\mc{F}_{B \nsubseteq}^{-1}] \ar@{.>}[ur]_{(-)^{B}}^{\simeq}
}
$$ 
where in the above diagram the $B$-fixed point functor $(-)^B$ is an equivalence \cite[Cor.~II.9.6]{LMS}, \cite[Rmk.~2.12]{AMGR}.  The inverse to this equivalence is the localization of the inflation:
\begin{align*}
\Sp^{A/B} & \xrightarrow{\simeq} \Sp^A[\mc{F}^{-1}_{B \nsubseteq}] \\
E & \mapsto (q_B^* E)[\mc{F}^{-1}_{B \nsubseteq}]
\end{align*}
The localization $E \mapsto E[\mc{F}^{-1}]$ is a smashing localization, and is therefore Bousfield localization with respect to the spectrum $S[\mc{F}^{-1}] \in \Sp^A$.
For $B \le A$, we have
\begin{equation}\label{eq:geomfp}
 E[\mc{F}^{-1}]^{\Phi B} \simeq \begin{cases}
E^{\Phi B}, & B \ne \mc{F}, \\
\ast, & B \in \mc{F}. \end{cases}
\end{equation}
\label{eq:geomloc}
\end{enumerate}

\begin{prop}\label{prop:pullbacks}
We have pullbacks
$$
\xymatrix{
E^{\wedge}_{\mc{F} \cup \mc{F}'} \ar[r] \ar[d] & E^{\wedge}_{\mc{F}'} \ar[d] 
\\
E^\wedge_{\mc{F}} \ar[r] & E^\wedge_{\mc{F} \cap \mc{F}'}
}
$$
and 
$$
\xymatrix{
E[(\mc{F} \cap \mc{F}')^{-1}] \ar[r] \ar[d] & E[(\mc{F}')^{-1}] \ar[d]
\\
E[\mc{F}^{-1}] \ar[r] & E[(\mc{F} \cup \mc{F}')^{-1}]
}
$$
\end{prop}

The following isotropy separation theorem is due to Greenlees-May \cite{GreenleesMay}.

\begin{thm}[Greenlees-May]\label{thm:GreenleesMay}
We have pullbacks
$$
\xymatrix{
E \ar[r] \ar[d] & E[\mc{F}^{-1}] \ar[d]
\\
E^{\wedge}_{\mc{F}} \ar[r] & E^{\wedge}_{\mc{F}}[\mc{F}^{-1}]
}
$$
\end{thm}

This last theorem is really a special case of the following generalized Sullivan arithmetic square, using (\ref{eq:geomfp}).

\begin{prop}[Bauer {\cite[Prop.~2.2]{Bauer}}]\label{prop:fracture}
Suppose that $\mc{C}$ is a presentable stable symmetric monoidal $\infty$-category.  For $E,F,X \in \mc{C}$ with $E\otimes X_F \simeq 0$, the canonical square
$$
\xymatrix{X_{E \vee F} \ar[r] \ar[d] & X_E \ar[d] \\
X_F \ar[r] & X_{E,F} }
$$
is a pullback (where $X_{E,F}$ denotes the iterated localization $(X_E)_F$).
\end{prop}

We might regard $\Sub(A)$ as a kind of ``coarse Balmer spectrum,'' which decomposes $\Sp^A$ in terms of isotropy in a manner similar to the way that classical chromatic theory decomposes $\Sp$ into its $K(n)$-local components.  This analogy is made precise in \cite{AbramKriz} (see also \cite{AMGR} and \cite{BalderramaPower}).  We briefly explain this perspective.

We shall require the following, which basically asserts that if $V$ is closed, $U$ is open, and $V \cap U = \emptyset$, then if $E \in \Sp^A$ is $U$-local, then its completion at $V$ (restriction to an infinitesimal neighborhood of $V$) is trivial.

\begin{lem}\label{lem:nullcompletion}
If $\mc{F}$, $\mc{F}' \subseteq \Sub(A)$ are families and $\mc{F}' \cap \mc{F}^c = \emptyset$, then for all $E \in \Sp^A$ we have
$$ E[\mc{F}^{-1}]^{\wedge}_{\mc{F}'} \simeq \ast. $$ 
\end{lem}

\begin{proof}
Observe that the condition $\mc{F}' \cap \mc{F}^c = \emptyset$ is equivalent to $\mc{F}' \subseteq \mc{F}$.  Since for all $B \in \mc{F}' \subseteq \mc{F}$, we have
$$ E[\mc{F}^{-1}]^B \simeq \ast, $$
it follows that the map
$$ E[\mc{F}^{-1}] \to \ast $$
is an $\mc{F}'$-equivalence, and therefore the $\mc{F}'$-completion of $E[\mc{F}^{-1}]$ is trivial.
\end{proof}

%
%

\subsection*{Euler classes}

Just as closed hypersurfaces of a variety are parameterized by line bundles, there are closed subspaces $V_\alpha \subseteq \Spc(A)$ associated to characters $\alpha \in A^\vee$. The associated localization and completion functors admit explicit formulas associated to the Euler classes of these representations.

Given a representation $\gamma$ of $A$, there is a cofiber sequence
$$ S(\gamma)_+ \to S^0 \xrightarrow{a_\gamma} S^\gamma $$
where $a_\gamma$ is the \emph{Euler class}.  We regard $a_\gamma$ as an element of the $RO(A)$-graded stable stems
$$ a_\gamma \in \pi^A_{-\gamma}S. $$

For a character $\alpha \in A^\vee$, there is an associated family
$$ \mc{F}_\alpha := \{ B \le A \: : \: \alpha\big\vert_B = 1 \}. $$
Note that $\mc{F}_\alpha = \mc{F}_{\subseteq K_\alpha}$, where $K_\alpha$ is the kernel of $\alpha$:
$$ 1 \to K_\alpha \to A \xrightarrow{\alpha} \CC^\times. $$
We let $V_\alpha = V(\mc{F}_\alpha)$ denote the associated closed subspace of $\Spc(A)$.

\begin{lem}
We have 
$$ E\mc{F}_\alpha \simeq S(\infty\alpha). $$
\end{lem}

\begin{cor}\label{cor:formula}
We have
$$ E^{\wedge}_{\mc{F}_\alpha} \simeq E^\wedge_{a_\alpha} $$
and 
$$ E[\mc{F}_\alpha^{-1}] \simeq E[a^{-1}_{\alpha}]. $$
\end{cor}

\begin{cor}
$\mc{F}_\alpha$-completion is Bousfield localization with respect to $S/a_\alpha$, and $\mc{F}^{-1}_\alpha$-localization is Bousfield localization with respect to $S[a^{-1}_\alpha]$.
\end{cor}

Corollary~\ref{cor:formula} and Proposition~\ref{prop:pullbacks} allow for an explicit inductive computation of the localization away from and the completion at any family $\mc{F}$.  In particular we have the following.

\begin{prop}\label{prop:eulerloc}
For $B \le A$ we have 
$$ E^{\wedge}_{\mc{F}_{\subseteq B}} \simeq E^{\wedge}_{(a_\alpha\: : \: \alpha|_B = 1)} $$
and we have 
$$ E[\mc{F}^{-1}_{B\nsubseteq}] \simeq E[a_\alpha^{-1} \: : \: \alpha\big|_B \ne 1]. $$
\end{prop}

Figure~\ref{fig:suba2} gives an example of some closed subspaces of $\mr{Sub}(A)$ corresponding to Euler classes of some characters of $C_2 \times C_2 = \langle a, b \rangle$ defined by:
\begin{gather*}
 \alpha(a) = \beta(b) = 1, \\
 \alpha(b) = \beta(a) = -1.
\end{gather*}
I.e., $K_\alpha = \langle a \rangle$ and $K_\beta = \bra{b}$.

\begin{figure}
\centering
\includegraphics[width=0.5\linewidth]{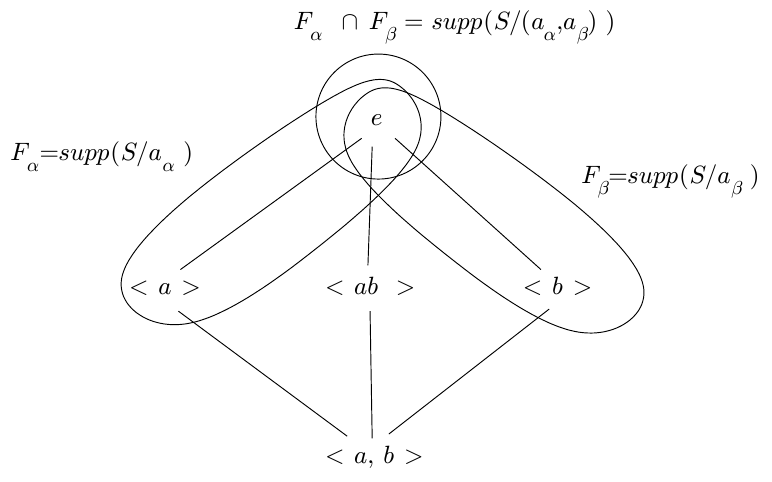}
\caption{Some closed subsets of $\mr{Sub}(C_2 \times C_2)$ determined by characters.}
\label{fig:suba2}
\end{figure}

\subsection*{The Isotropy Tower}
 
Associated to the poset of open subsets $U = \mc{F}^c \subseteq \Sub(A)$ there is a diagram of smashing localizations $E[\mc{F}^{-1}]$, where if $\mc{F}_1^c \subseteq \mc{F}_2^c$, there are maps
\begin{gather*}
E[\mc{F}_2^{-1}] \\
\downarrow \\
E[\mc{F}_1^{-1}]
\end{gather*}
One may think of this as a sort of coarse chromatic tower which we will call the \emph{isotropy tower of $E$}.  Following \cite{AbramKriz}, \cite{AMGR}, and \cite{BalderramaPower}, we will explain how the isotropy tower decomposes the homotopy type of $E$ into its ``monoisotropic layers.''

\begin{defn}
For each $B \le A$, let $T(B) \in \Sp^A$ denote the $A$-spectrum
$$ T(B) := \Sigma^{\infty}A/B_+[\mc{F}^{-1}_{B \not\subseteq}], $$
so a map in $\Sp^A$ is a $T(B)$-equivalence if and only if it is an equivalence on $B$-geometric fixed points.
The \emph{$B$-monoisotropic localization of $E \in \Sp^A$} is the the Bousfield localization of $E$ with respect to $T(B)$.
\end{defn}

The following proposition relates $T(B)$-localization to our localization and completion functors associated to families.  The formulation below was suggested by the referee.

\begin{prop}\label{prop:completions}
If $\mc{F}, \mc{F}' \subseteq \Sub(A)$ are families, then for $E \in \Sp^A$ the map 
\begin{equation}\label{eq:locmap}
 E \to E[\mc{F}^{-1}]^{\wedge}_{\mc{F}'}
\end{equation}
is localization with respect to 
$$ T = \bigvee_{B \in \mc{F}' \cap \mc{F}^c} T(B). $$
Here, the $T$-equivalences are the $\mc{F}'\cap \mc{F}^c$-\emph{equivalences}: the maps in $\Sp^A$ which are equivalences on $B$-geometric fixed points for all $B \in \mc{F}' \cap \mc{F}^c$. 
\end{prop}

\begin{proof}
The map (\ref{eq:locmap}) is the composite
$$ E \to E[\mc{F}^{-1}] \to E[\mc{F}^{-1}]^{\wedge}_{\mc{F}'} $$
of $\mc{F}'\cap \mc{F}^c$-equivalences, hence is a $\mc{F}' \cap \mc{F}^c$-equivalence.  We just need to show that ${E}[\mc{F}^{-1}]^{\wedge}_{\mc{F}'}$ is $T$-local.  Suppose that $X \in \Sp^A$ is $T$-acyclic, so for all $B\in \mc{F}' \cap \mc{F}^c$ we have 
$$ X^{\Phi B} \simeq \ast. $$
We wish to show that
$$ [X, E[\mc{F}^{-1}]^{\wedge}_{\mc{F}'}]^A = 0. $$
By (\ref{eq:geomfp}) we have 
$$ X[\mc{F}^{-1}]^{\Phi B} \simeq \ast $$
for all $B \in \mc{F}'$.  It follows that we have 
$$ [X[\mc{F}^{-1}], E[\mc{F}^{-1}]^{\wedge}_{\mc{F}'}]^A = 0. $$
The result would follow if we could show that $E[\mc{F}^{-1}]^{\wedge}_{\mc{F}'}$ is $\mc{F}^{-1}$-local. By Theorem~\ref{thm:GreenleesMay} there is a pullback
$$
\xymatrix{
E[\mc{F}^{-1}]^{\wedge}_{\mc{F}'} \ar[r] \ar[d] &
E[\mc{F}^{-1}]^{\wedge}_{\mc{F}'\cap \mc{F}} \ar[d] \\
E[\mc{F}^{-1}]^{\wedge}_{\mc{F}'}[\mc{F}^{-1}] \ar[r] &
E[\mc{F}^{-1}]^{\wedge}_{\mc{F}'\cap \mc{F}}[\mc{F}^{-1}]
}
$$
But Lemma~\ref{lem:nullcompletion} implies that
$$ E[\mc{F}^{-1}]^{\wedge}_{\mc{F}'\cap \mc{F}} \simeq \ast. $$
The result follows.
\end{proof}

\begin{cor}\label{cor:T(B)equiv}
We have
$$ E_{T(B)} \simeq E[\mc{F}^{-1}_{B \not\subseteq}]^{\wedge}_{\mc{F}_{\subseteq B}}. $$ 
\end{cor}

One interesting consequence of Proposition~\ref{prop:completions} is that if $U \subseteq \Sub(A)$ is open, and $E \in \Sp^A$ is $U$-local, then its completion a closed $V \subseteq \Sub(A)$ only depends on $V \cap U$.  Said differently, we have the following corollary.

\begin{cor}
If $\mc{F}, \mc{F}', \mc{F}'' \subseteq \Sub(A)$ are families, and 
$$ \mc{F}' \cap \mc{F}^c = \mc{F}'' \cap \mc{F}^c. $$
Then there is a canonical equivalence 
$$ E[\mc{F}^{-1}]^{\wedge}_{\mc{F}'} \simeq E[\mc{F}^{-1}]^{\wedge}_{\mc{F}''}. $$
\end{cor}

Note the following:
\begin{enumerate}[label = (\theequation)]
\eitem $\mc{F}^{-1}$-localization is localization with respect to
$$ \bigvee_{B \in \mc{F}^c} T(B). $$
\eitem By Corollary~\ref{cor:formula}, there is a equivalence
$$ T(B) \simeq \Sigma^{\infty}A/B_+[a_\beta^{-1} \: : \: \beta \big\vert_B \ne 1]. $$
We may think of $T(B)$ as a telescopic localization of a complex with support $\mc{F}_{\subseteq B}$ with respect to all non-trivial Euler classes.  Thus $B$-monoisotropic localization is analogous to $T(n)$-localization in the chromatic setting.
\eitem For $E \in \Sp^A$, Corollary~\ref{cor:formula} implies that we have
$$ E_{T(B)} \simeq E[a^{-1}_\beta \: : \: \beta\big\vert_B \ne 1]^{\wedge}_{(a_\alpha \: : \: \alpha\vert_B = 1)}. $$
\eitem In particular, the spectrum $E_{T(B)}$ is the inverse limit of spectra
$$  E[a^{-1}_\beta \: : \: \alpha\big\vert_B \ne 1]/(a^{i_\alpha}_\alpha \: : \: \alpha \big\vert_B = 1) $$
and the $RO(A)$-graded homotopy groups of these spectra exhibit  $a_\alpha$-periodicity for $\alpha \big\vert_B \ne 1$ (tautologically) and $u_\alpha$-periodicity for $\alpha \big\vert_B = 1$ by the work of Balderrama \cite{Balderramauself}. 
\eitem The fixed points of $E_{T(B)}$ are given by
$$ E^C_{T(B)} \simeq \begin{cases}
(\Phi^B E)^{hC/B}, & B \subseteq C, \\
\ast, & B \not\subseteq C
\end{cases}
$$
for $C \le A$.
\eitem It follows from Corollary~\ref{cor:T(B)equiv}, \ref{eq:geomloc}, and \ref{eq:Fcompletion} that there is an equivalence of $\infty$-categories 
$$
 \Sp^A_{T(B)} \simeq \Sp^{B(A/B)}.
$$
 where the right-hand side is the $\infty$-category of Borel $A/B$-spectra.
 \label{eq:TBlocal}
\eitem Lemma~\ref{lem:nullcompletion} implies that if $B_1 \subsetneq B_2$, then
$$ E_{T(B_2),T(B_1)} \simeq \ast. $$
This mimics the behavior of iterated $K(n)$-localization in chromatic homotopy theory, if one follows our practice of drawing subgroup lattices upside down.
\end{enumerate}

\begin{prop}\label{prop:isotropyfracture}
Suppose that $\mc{F}_1, \mc{F}_2 \subseteq \Sub(A)$ are families with
$$ \mc{F}_2^c = \mc{F}_1^c \cup \{ B \}. $$
Then the square
$$
\xymatrix{
E[\mc{F}_2^{-1}] \ar[d] \ar[r] &
E_{T(B)} \ar[d] \\
E[\mc{F}_1^{-1}] \ar[r] &
E_{T(B)}[\mc{F}_1^{-1}]
}
$$
is a pullback.
\end{prop}   

The following picture illustrates the situation described in Proposition~\ref{prop:isotropyfracture}.

\begin{center}
\includegraphics[width=0.3\linewidth]{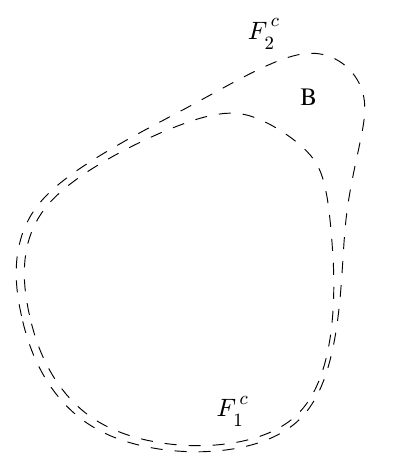}
\end{center}

\begin{proof}[Proof of Proposition~\ref{prop:isotropyfracture}]
The pullback is achieved as a composite of pullbacks
$$
\xymatrix{
E[\mc{F}_2^{-1}]\ar[d] \ar[r] & 
E[\mc{F}^{-1}_{B \not\subseteq}] \ar@{=}[r] \ar[d] &
E[\mc{F}^{-1}_{B \not\subseteq}] \ar[r] \ar[d] &
E[\mc{F}^{-1}_{B \not\subseteq}]^{\wedge}_{\mc{F}_1} \ar[d]
\\
E[\mc{F}_1^{-1}] \ar[r] & 
E[(\mc{F}_{B \not\subseteq} \cup \mc{F}_1)^{-1}] \ar@{=}[r] &
E[\mc{F}^{-1}_{B \not\subseteq}][\mc{F}^{-1}_1] \ar[r] &
E[\mc{F}^{-1}_{B \not\subseteq}]^{\wedge}_{\mc{F}_1}[\mc{F}_1^{-1}]
}
$$
where the left pullback is Proposition~\ref{prop:pullbacks} and the right pullback is Theorem~\ref{thm:GreenleesMay}.  The proposition follows from Proposition~\ref{prop:completions}, which implies that there is an equivalence
$$ E[\mc{F}^{-1}_{B \not\subseteq}]^{\wedge}_{\mc{F}_1}  \simeq E_{T(B)}. $$
\end{proof}

Proposition~\ref{prop:isotropyfracture} implies that the homotopy type of $E \in \Sp^A$ can be inductively built out of the monoisotropic layers
$$ \{ E_{T(B)} \: : \:  \quad B \le A \} $$
by choosing a maximal chain of opens
$$ \{ A \} = \mc{F}_0^c \subsetneq \mc{F}_1^c \subsetneq \mc{F}_2^c \subsetneq \cdots \subsetneq \mc{F}_k^c = \Sub(A) $$
and considering the associated isotropy tower.  An example of such a tower in the case of $A = C_2 \times C_2$ is shown below (compare with G.~Yan's approach in \cite{Yan}):

\begin{center}
\includegraphics[width=\linewidth]{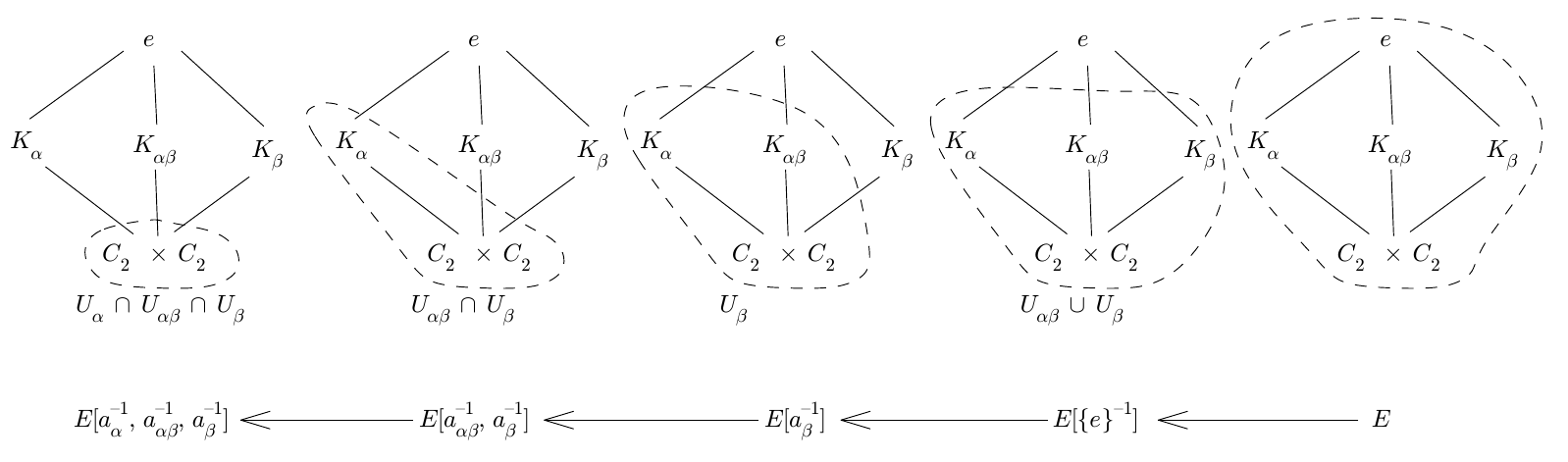}
\end{center}

\section{Equivariant formal group laws}\label{sec:FGL}

\subsection*{Equivariant formal groups}

Let $k$ be a commutative ring, and let 
$A^\vee_k = \Spec(k^{A^\vee}) = A^\vee \times \Spec(k)$
denote the \'etale group-scheme over $k$ associated to the group $A^\vee$.  

\begin{defn}\label{def:EFG}
An \emph{$A$-equivariant formal group} over $k$ is a pair $(\GG, \varphi)$ where
\begin{enumerate}
\item $\GG$ is a formal commutative group scheme over $k$, and 
\item $\varphi: A^\vee_k \to \GG$ is a homomorphism, such that
\item $\varphi(A^\vee_k) \hookrightarrow \GG$ is an infinitesimal thickening\footnote{By this, we mean that the map of the completion $\GG^{\wedge}_{\im \varphi}$ to $\GG$ is an isomorphism of formal schemes.}, and 
\item the completion $\GG_{\varphi_1}^{\wedge}$ at the identity $\varphi_1 = \{1\} \times \Spec(k) \hookrightarrow \GG$ is a $1$-dimensional formal group over $k$. 
\end{enumerate}  
\end{defn}

\begin{rmk}
We will use multiplicative notation for the group $A^\vee$, and additive notation for the formal group scheme $\GG$.  
\end{rmk}

\begin{exs}
We give some examples of equivariant formal groups.
\begin{enumerate}
\item If $\GG$ is a formal group over $k$, and $\varphi: A_k^\vee \to \GG$ is the trivial homomorphism, then $(\GG,\varphi)$ is an $A$-equivariant formal group.
\item If $\GG_1$ is a formal group over $k$, $\GG = A^\vee \times \GG_1$, and 
$$ \varphi: A^\vee \hookrightarrow A^\vee \times \GG_1 $$
is the canonical inclusion, then $(\GG,\varphi)$ is an $A$-equivariant formal group.
\item If $A$ is a $p$-group, $\GG$ is a $1$-dimensional $p$-divisible group over $k$, and 
$$ \varphi: A^\vee_k \to \GG $$ is a homomorphism, then $(\GG^{\wedge}_{\im \varphi}, \varphi )$ is an $A$-equivariant formal group law.  In particular, if $k' = \mc{O}_{\Hom(A_k^\vee, \GG)}$, then $\GG$ pulls back to a $p$-divisible group $\GG'$ over $k'$ with a canonical homomorphism
$$ \varphi': A^\vee_{k'} \to \GG' $$
and $((\GG')^{\wedge}_{\im \varphi'}, \varphi')$ is an $A$-equivariant formal group --- this is the equivariant formal group associated to a tempered cohomology theory (see \cite{Multicurves}, \cite{Lurie}, \cite{HKS}, \cite{GepnerMeier}, \cite{Davies}).
\end{enumerate}
\end{exs}

Let $X$ denote $\Spec(k)$.
For an $A$-equivariant formal group $(\GG, \varphi)$, for each $\alpha \in A^\vee$ we will let
$$ \varphi_\alpha = \{\alpha\} \times X \hookrightarrow \GG$$
denote the corresponding subscheme, and we will define
$$ \GG_\alpha := \GG^\wedge_{\varphi_\alpha}. $$

\subsection*{Equivariant formal group laws}

\begin{defn}
A \emph{(global) coordinate} on an $A$-equivariant formal group $(\GG,\varphi)$ is a regular function $x \in \mc{O}_\GG$ which generates the ideal $I_{\varphi_1} \subset \mc{O}_\GG$.  An \emph{$A$-equivariant formal group law} over $k$ is a triple $(\GG,\varphi, x)$ consisting of an  $A$-equivariant formal group over $k$, and a choice of coordinate.
\end{defn}

\begin{rmk}
For general $(\GG,\varphi)$ over $X$, coordinates only exist Zariski locally in $X$. 
\end{rmk}

For each $\alpha \in A^\vee$ we will let 
$$ t_\alpha: \GG \to \GG $$
denote the translation by $\alpha$ map induced by $\varphi$.  These give isomorphisms of formal schemes
$$ t_{\alpha}: \GG_\beta \xrightarrow{\cong} \GG_{\alpha\beta}. $$

Given a global coordinate $x$ of $(\GG,\varphi)$, we define translation coordinates  
$$ x_\alpha := t^*_{\alpha^{-1}}x \in \mc{O}_\GG. $$
Note that $(x_\alpha) = I_{\varphi_\alpha}$.
We define \emph{Euler classes}
$$ e_\alpha := x_\alpha\bigg|_{\varphi_1} = x_1 \bigg|_{\varphi_{\alpha^{-1}}} \in k. $$

\begin{rmk}
Our conventions are those of \cite{CGK}, and differ from Hausmann-Meier \cite{HM} --- the $y_\alpha$ of \cite{HM} is our $x_{\alpha^{-1}}$ and the $e_\alpha$ of \cite{HM} corresponds to our $e_{\alpha^{-1}}$. 
\end{rmk}

Note that we have
$$ \varphi_1 \cap \varphi_\alpha = X\bigg\vert_{e_\alpha = 0}. $$
\begin{figure}
\centering
\includegraphics[width=.8\linewidth]{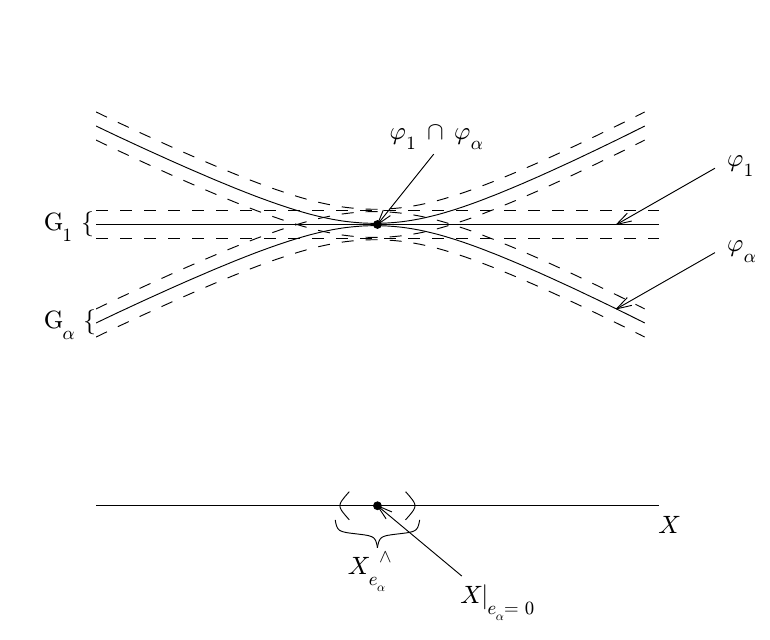}
\caption{Visualization of an equivariant formal group law $(\mathbb{G}, \varphi)$ over $X$.  This particular picture is meant to represent a $C_3$-equivariant formal group. The three curved lines represent the image $\im \varphi$, and $\GG$ is represented by the dotted ``neighborhood'' of $\im \varphi$ in this visualization.  The straight curve and the bottom curve represent the closed subsechemes $\varphi_1$ and $\varphi_\alpha$, respectively, where $\alpha \in C_3^\vee$ is a non-trivial character.  The dotted ``neighborhoods of these two lines represent the formal subschemes $\GG_1$ and $\GG_\alpha$, respectively.}
\label{fig:afgl}
\end{figure}
Figure~\ref{fig:afgl} gives a visualization of the geometric structure of an equivariant formal group law.

Given a global coordinate $x$, we define local coordinates 
\begin{equation}\label{eq:localcoord}
 z_\alpha := x_\alpha\bigg|_{\GG_\alpha} \in \mc{O}_{\GG_\alpha}. \end{equation}
We then have 
$$ \GG_\alpha = \Spf(k[[z_\alpha]]). $$
In particular, for $i \ge 1$ there exist $b^\alpha_i \in k$ such that 
\begin{equation}\label{eq:balpha}
x_\alpha\bigg\vert_{\GG_1} = e_\alpha + b_1^{\alpha} z_1 + b_2^\alpha z_1^2 + \cdots.
\end{equation}
We will find it convenient to denote this power series as $b^\alpha(z)$, where
\begin{equation}\label{eq:balpha(z)}
 b^\alpha(z) := e_\alpha + b_1^{\alpha} z + b_2^\alpha z^2 + \cdots
 \end{equation}
It follows that we have 
\begin{equation}\label{eq:balphabeta}
 x_{\alpha}\bigg\vert_{\GG_\beta} = b^{\alpha\beta^{-1}}(z_\beta).
 \end{equation}
Let $a_{i,j}$ denote the coefficients of the formal group law $(\GG_1, z_1)$:
$$ z_1 +_{\GG_1} z'_1 = \sum_{i,j} a_{i,j} z_1^i (z'_1)^j. $$
Then the multiplication on $\GG$ restricts to give a local multiplication
$$ \GG_\alpha \widehat{\times} \GG_\beta \to \GG_{\alpha\beta} $$
which on local coordinate rings is expressed as 
$$ z_{\alpha\beta} \mapsto z_\alpha +_{\GG_1} z_\beta. $$

\begin{lem}\label{lem:intersection}
We have 
$$ \GG_\alpha \cap \GG_\beta = \GG_\alpha \bigg\vert_{X^{\wedge}_{e_{\alpha\beta^{-1}}}} = \GG_\beta \bigg\vert_{X^{\wedge}_{e_{\alpha\beta^{-1}}}}. $$ 
\end{lem}

\begin{proof}
We have 
$$
\GG_\alpha \cap \GG_\beta = (\GG_\alpha)^{\wedge}_{x_\beta} = \Spf(k[[z_\alpha]]^{\wedge}_{x_\beta})
$$
and we have
\begin{align*}
\GG_\alpha \bigg\vert_{X^{\wedge}_{e_{\alpha\beta^{-1}}}} & = \Spf(k[[z_\alpha]]^{\wedge}_{e_{\alpha\beta^{-1}}}), \\
\GG_\beta \bigg\vert_{X^{\wedge}_{e_{\alpha\beta^{-1}}}} & = \Spf(k[[z_\beta]]^{\wedge}_{e_{\alpha\beta^{-1}}}).
\end{align*}

Thus it suffices to show that
$$ k[[z_\alpha]]^{\wedge}_{x_\beta} = k[[z_\alpha]]^{\wedge}_{e_{\alpha\beta^{-1}}} = k[[z_\beta]]^{\wedge}_{e_{\alpha\beta^{-1}}}. $$
This follows from (\ref{eq:balphabeta}).
\end{proof}

\begin{prop}\label{prop:xalphaloc}
On $\GG_\alpha \cap \GG_\beta = \GG_\beta\bigg\vert_{X^{\wedge}_{e_{\alpha\beta^{-1}}}}$, we have 
$$ x_\alpha = e_{\alpha\beta^{-1}} +_{\GG_1} x_\beta. $$
\end{prop}

\begin{proof}
We first assume $\beta = 1$.  The inclusion
$$ \Spec(k) = \varphi_{\alpha^{-1}} \hookrightarrow \GG $$
on representing rings sends 
$$ x \mapsto e_{\alpha}. $$
The map
$$ t_{\alpha^{-1}}: \GG_{\alpha} \to \GG_1 $$
restricts over $X^{\wedge}_{e_\alpha}$ to a map
$$ t_{\alpha^{-1}}: \GG_1 \cap \GG_\alpha \to \GG_{1} \cap \GG_\alpha. $$
On $\GG_1 \cap \GG_\alpha$, the group law of $\GG$ restricts to the restriction of the formal group law of $\GG_1$, and the result for $\beta = 1$ follows.  The general result follows by writing
$$ x_{\alpha\beta^{-1}} = e_{\alpha\beta^{-1}} +_{\GG_1} x $$
and then precomposing with $t^*_{\beta^{-1}}$.
\end{proof}

Setting $\beta = 1$ in the above proposition immediately gives:

\begin{cor}\label{cor:bsum}
On $\GG_\alpha \cap \GG_1 = \GG\bigg\vert_{X^{\wedge}_{e_\alpha}}$ we have
$$ b^\alpha(z_1) = e_\alpha +_{\GG_1} z_1. $$
\end{cor}

\begin{cor}\label{cor:eb(z)}
On $X^{\wedge}_{e_\alpha}$, we have 
$$ b^{\alpha\beta}(z_1) = b^\beta(e_\alpha +_{\GG_1} z_1). $$
In particular, we have
$$ e_{\alpha\beta} = b^{\beta}(e_\alpha). $$
\end{cor}

\begin{proof}
Equation~(\ref{eq:balphabeta}) implies that 
\begin{equation}\label{eq:eb(z)}
 x_{\alpha\beta}\bigg\vert_{\GG_\alpha} = b^\beta(z_\alpha).
 \end{equation}
Lemma~\ref{lem:intersection} implies that over $X^{\wedge}_{e_\alpha}$ we have $\GG_\alpha = \GG_1$, therefore over $X^{\wedge}_{e_\alpha}$, we have 
$$ z_\alpha = x_\alpha|_{\GG_{\alpha}} = x_{\alpha}|_{\GG_1} = b^\alpha(z_1) = e_\alpha+_{\GG_1} z_1, $$
where the last equality is by Corollary~\ref{cor:bsum}. 
Therefore, over $X^{\wedge}_{e_\alpha}$, we may restrict (\ref{eq:eb(z)}) to $\GG_1$ to get the desired result.
\end{proof}

\begin{cor}
On $X^{\wedge}_{(e_\alpha, e_\beta)}$ we have 
$$ e_{\alpha\beta} = e_\alpha+_{\GG_1} e_\beta $$
\end{cor}

\begin{proof}
The Euler classes $e_\gamma$ only depend on the restrictions of $x_\gamma$ to $\GG_1$.  By Lemma~\ref{lem:intersection} we have
$$ \GG_1\bigg|_{X^{\wedge}_{(e_\alpha, e_\beta)}} = \GG_{\alpha\beta} \cap \GG_\beta \cap \GG_1. $$
By Proposition~\ref{prop:xalphaloc}, on $\GG_1 \bigg|_{X^{\wedge}_{(e_\alpha, e_\beta)}}$, we have 
\begin{align*}
x_{\alpha\beta} & = e_{\alpha} +_{\GG_1} x_\beta \\
& = e_{\alpha} +_{\GG_1} e_{\beta} +_{\GG_1} x.
\end{align*}
The result follows from evaluating this expression at $x = 0$.
\end{proof}

\begin{cor}
Suppose $\alpha \in A^{\vee}$ has order $n$.  Then on $X^{\wedge}_{e_\alpha}$, we have
$$ [n]_{\GG_1}(e_\alpha) = 0. $$
\end{cor}

\subsection*{Moduli of equivariant formal groups}

Let $\mc{M}^A_\mit{fg}$ denote the moduli stack of $A$-equivariant formal groups, and $\mc{M}^A_\mit{fgl}$ the moduli stack of $A$-equivariant formal group laws.  We have
$$ \mc{M}^A_\mit{fgl} = \Spec(L_A) $$
for $L_A$ the equivariant Lazard ring.  There is a third variant, $\mc{M}^{A,1}_\mit{fg}$, which is the moduli stack of $(\GG, \phi, v)$ where $v$ is a nowhere vanishing section of the relative tangent bundle $T_{\varphi_1}\GG$.  Note that two $A$-equivariant formal group laws determine isomorphic points of $\mc{M}^{A,1}_\mit{fg}$ if and only if they are strictly isomorphic.  The multiplicative group $\GG_m$ acts on $\mc{M}^{A,1}_\mit{fg}$ by acting on the vector field $v$, giving the functions on $\mc{M}^{A,1}_\mit{fg}$ the structure of a grading.

Over $\mc{M}^A_\mit{fgl}[e_\alpha^{-1}\: : \: \alpha \in A^\vee - 1]$ we have $\varphi_\alpha \cap \varphi_1 = \emptyset$ for all $\alpha \ne 1$, and therefore
$$ \GG \cong A^\vee \times \GG_1 = \coprod_{\alpha \in A^\vee} \GG_\alpha. $$
The $A$-equivariant formal group is therefore determined by the formal group $\GG_1$.  A coordinate consists of the choice of a coordinate on $\GG_1$ (which is the local coordinate $z_1$, which determines local coordinates $z_\alpha$ on $\GG_\alpha$ by translation) plus an extension $x$ of $z_1$ to all of $\GG$ subject only to the constraint that $x$ is non-vanishing on $\GG_\alpha$ for $\alpha \ne 1$.  This extension is specified by specifying, for $\alpha \ne 1$ 
$$ x\bigg\vert_{\GG_\alpha} = e_\alpha + \sum_{i \ge 1} b_i^\alpha z_\alpha $$
with $e_\alpha$ invertible.  Thus we have
$$ L_A[e_\alpha^{-1}\: : \: \alpha \ne 1] \cong \MU_*[e_\alpha^{\pm}, b_i^\alpha\: : \: \alpha \ne 1, i \ge 1]. $$

Over $(\mc{M}^A_\mit{fgl})^{\wedge}_{(e_\alpha\: : \: \alpha \in A^\vee)}$ we have
$$ \GG_1 \xrightarrow{\cong} \GG. $$
The homomorphism
$$ \varphi: A^\vee \to \GG_1 $$
is then determined by the Euler classes $e_{\alpha}$, and we have
$$ (L_A)^{\wedge}_{(e_\alpha\: : \: \alpha\in A^\vee)} \cong \frac{\MU_*[[e_{\alpha} \: : \: \alpha \in A^\vee]]}{(e_{\alpha}+_{\GG_1} e_\beta = e_{\alpha\beta})}. $$

More generally, for $B \le A$:
\begin{enumerate}[label=(\theequation)]
\eitem $\mc{M}^{A}_\mit{fgl}[e_\alpha^{-1}\: : \: \alpha\big\vert_B \ne 1] = \{(\GG,\varphi) \: : \: \ker \varphi \subseteq (A/B)^\vee\}$
is the locus where
$$ \GG \cong A^\vee \times_{(A/B)^\vee} \GG' $$
for the $A/B$-equivariant formal group $\GG' = \GG^{\wedge}_{\varphi((A/B)^\vee)}$.  Choosing a set theoretic section
$$ s: B^\vee \to A^\vee $$
of the surjection $A^\vee \twoheadrightarrow B^\vee$, 
there is an isomorphism \cite{tomDieck}, \cite[Cor. 10.4]{Greenlees}
$$ L_A[e_\alpha^{-1}\: : \: \alpha\big\vert_B \ne 1] \cong L_{A/B}[e_{s(\beta)}^\pm, b_i^{s(\beta)} \: : \: 1 \ne \beta \in B^\vee, \: i \ge 1]. $$
\label{eq:MFGAloc}

\eitem $\mc{M}^A_{\mit{fgl}}\bigg\vert_{(e_\alpha = 0\: : \: \alpha|_B = 1)} = \{(\GG,\varphi) \: : \: (A/B)^\vee \subseteq \ker \varphi\}$
is the locus where $\phi\big\vert_{(A/B)^\vee} = 1$ ($\phi$ factors through $B^\vee$)
and there is an isomorphism
$$ L_A/(e_\alpha \: : \: \alpha\vert_B = 1) \cong L_B. $$

\eitem $(\mc{M}^A_{\mit{fgl}})^{\wedge}_{(e_\alpha\: : \: \alpha|_B = 1)}$ is the locus where $\GG_\alpha = \GG_1$ for all $\alpha\big|_B = 1$.
\label{eq:MFGAcompl}
\end{enumerate}

\section{$\MU_A$ and the equivariant Lazard ring}\label{sec:MUA}

\subsection*{Hausmann's Theorem}

Given a complex orientable homotopy commutative ring object $E \in \Sp^A$, $\GG_E = \Spf(E^*\CC P^\infty_A)$ has the structure of an equivariant formal group law coming from the tensor product of $A$-equivariant line bundles.  A choice of complex orientation of $E$ is the same thing as a choice of coordinate on $\GG_E$.

The following fundamental theorem was conjectured by \cite{CGK} and is due to Hausmann \cite{Hausmann} (see also \cite{KrizLu} for another approach, and \cite{Hanke}, who first proved this in the case of $A = C_2$).

\begin{thm}[Hausmann \cite{Hausmann}]\label{thm:Quillen}
The canonical complex orientation on $\GG_{MU_A}$ induces an isomorphism
$$ L_A \xrightarrow{\cong} \pi^A_*\MU_A. $$
Moreover, the stackification of the Hopf algebroid $(\pi^A_{2*}\MU_A, \pi^A_{2*}\MU_A \wedge \MU_A)$ is $\mc{M}^{A,1}_{FG}$, and under this correspondence the graded structure of the former coincides with the $\GG_m$-action on the latter.
\end{thm}

Theorem~\ref{thm:Quillen} gives a fundamental connection between $\MFG{A}$ and equivariant homotopy theory, generalizing the connection between $\MFG{}$ and stable homotopy theory established by Quillen \cite{Quillen}.

An essential component of Hausmann's approach is his use of the structure of equivariant complex cobordism a \emph{global spectrum} \cite{Schwede}. In particular, for each $B \le A$
there is an associated equivalence
$$ i_B^* \MU_A \simeq \MU_B $$ 
which gives an isomorphism
\begin{equation}\label{eq:piMUA}
 \pi^{B}_*\MU_A \cong L_B.
\end{equation}
The restriction maps associated to the inclusions $i_B: B \hookrightarrow A$, under the isomorphisms (\ref{eq:piMUA}), are given by the maps
$$ i_B^*: L_A \to L_B $$
which represent the maps
\begin{align*}
 \MFG{B} & \to \MFG{A} \\
 \left( B^\vee \xrightarrow{\varphi} \GG \right) & \mapsto \left(A^\vee \twoheadrightarrow B^\vee \xrightarrow{\varphi} \GG \right). 
\end{align*} 
This global structure also associates to 
the quotient map $q_B: A \to A/B$ 
a map
$$ q_B^*\MU_{A/B} \to \MU_A. $$
The induced map
$$ q_B^*: L_{A/B} \cong \pi^{A/B}_* \MU_{A/B} \to \pi^A_* \MU_A \cong L_{A} $$
represents the map
\begin{align*}
 \MFG{A} & \to \MFG{A/B} \\
 \left( A^\vee \xrightarrow{\varphi} \GG \right) & \mapsto \left((A/B)^\vee \xrightarrow{\varphi\big|_{(A/B)^\vee}} \GG^{\wedge}_{\varphi((A/B)^\vee)} \right). 
\end{align*} 

\subsection*{Elements of $\pmb{\pi^A_\star\MU_A}$}

Because the spectrum $MU_A$ is complex orientable, for every character $\alpha$ there is an invertible orientation class
$$ u_{\alpha} \in \pi^A_{2-\alpha} \MU_A. $$
These orientation classes are determined by the canonical complex orientation of $\MU_A$.  The canonical complex orientation corresponds to a coordinate $x$ of $\GG_{MU_A}$, which gives rise to Euler classes
$$ e_\alpha \in \pi^A_{-2}\MU_A $$
and elements
$$ b_i^\alpha \in \pi^A_{2i-2} \MU_A $$
for each $\alpha \in A^\vee$.
The orientation classes $u_\alpha$ are related to the Euler classes $e_\alpha$ by the formula
$$ a_\alpha = u_\alpha e_{\alpha}. $$
The following observations are well known (see, for example, \cite{Greenlees}).

\begin{lem}
There are equivalences
$$ \MU_A[e^{-1}_\alpha] \simeq \MU_A[\mc{F}_\alpha^{-1}] $$
and 
$$ (\MU_A)^{\wedge}_{e_\alpha} \simeq (\MU_A)^{\wedge}_{\mc{F}_\alpha}. $$
\end{lem}

\begin{cor}
For $\alpha \in A^\vee$ we have
$$ \pi^A_*\MU_A[\mc{F}^{-1}_\alpha] = L_A[e^{-1}_\alpha]. $$
\end{cor}

\begin{cor}
For $B \le A$, letting $s: B^\vee \to A^\vee$ be a set theoretic section of $A^\vee \twoheadrightarrow B^\vee$, we have
\begin{align*}
\pi^{A/B}_* \Phi^B
 \MU_A
& \cong L_A[e^{-1}_\alpha \: : \: \alpha \big\vert_B \ne 1] \\
& \cong L_{A/B}[e^{\pm}_{s(\beta)}, b_i^{s(\beta)} \: : \: 1 \ne \beta \in B^\vee, \: i \ge 1].
\end{align*}
\end{cor}

Since the restriction along $A \to A/B$ makes $\Phi^B \MU_A$ an $\MU_{A/B}$-module, and $\pi_*^{A/B} \Phi^B \MU_A$ is free over $\pi_* \MU_{A/B}$, we have the following.

\begin{cor}
The spectrum $\Phi^B \MU_{A} \in \Sp^{A/B}$ has the homotopy type of a wedge of even suspensions of $\MU_{A/B}$. 
\end{cor}

\begin{cor}\label{cor:MUAPhiB}
We have
$$ \pi_* (\MU_A)^{\Phi B} \cong \MU_*[e_{s(\beta)}^{\pm}, b_i^{s(\beta)} \: : \: 1 \ne \beta \in B^\vee, \: i \ge 1], $$
and $(\MU_A)^{\Phi B} \in \Sp$ has the homotopy type of a wedge of even suspensions of $\MU$.
\end{cor}

Finally, we recall the following well-known computation \cite{Landweber}.

\begin{prop}\label{prop:Borel}
The Borel completion (completion with respect to the family $\mc{F} = \{e\}$) of $\MU_A$ has homotopy groups
$$ \pi^A_*(\MU_A)^h = \MU^{-*}(BA) = \frac{\MU_*[[e_\alpha \: : \: \alpha \in A^\vee]]}{(e_\alpha +_{\GG_1} e_\beta = e_{\alpha\beta})} \cong (L_A)^\wedge_{(e_\alpha\: : \: \alpha \in A^\vee)}. $$
\end{prop}

\subsection*{The Abram-Kriz Theorem}

Abram and Kriz \cite{AbramKriz} observe 
that Corollary~\ref{cor:MUAPhiB} and Proposition~\ref{prop:Borel} generalizes to give a computation of the monoisotropic layers $\pi^A_*(\MU_A)_{T(B)}$.

\begin{prop}[Abram-Kriz]\label{prop:AbramKriz}
For $B \le A$, given a set-theoretic section $s$ of $A^\vee \twoheadrightarrow B^\vee$, there is an isomorphism
\begin{align*}
 \pi^{A}_*((\MU_A)_{T(B)}) & \cong \frac{\MU_*[e_{s(\beta)}^{\pm}, b_i^{s(\beta)} \: : \: 1 \ne \beta \in B^\vee, i \ge 1][[e_\alpha]]_{\alpha \in (A/B)^\vee}}{(e_\alpha +_{\GG_1} e_{\alpha'} = e_{\alpha\alpha'}\: : \: \alpha, \alpha' \in (A/B)^\vee)} \\
 & \cong L_A[e^{-1}_\beta \: : \: \beta\big\vert_B \ne 1]^{\wedge}_{(e_\alpha\: : \: \alpha\vert_B = 1)}.
 \end{align*}
\end{prop} 

\begin{rmk}
Geometrically, the ring in Proposition~\ref{prop:AbramKriz} classifies the universal example of an equivariant formal group
$$ \varphi: A^\vee \to \GG $$
over a ring $k$ complete with respect to an ideal $I$ where
\begin{enumerate}
\item $\ker \varphi  \subseteq (A/B)^\vee$,
\item $\ker \varphi\bigg\vert_{\Spec(k/I)} = (A/B)^\vee$.
\end{enumerate}
\end{rmk}

\begin{rmk}
The algebro-geometric interpretation of Proposition~\ref{prop:AbramKriz} above allows one to calculate the Euler classes of all of the characters in $A^\vee$, and explains the dependence of the isomorphism on the choice of a set theoretic section
$$ s: B^\vee \to A^\vee $$
of the surjection $A^\vee \twoheadrightarrow B^\vee$. Namely given $\gamma \in A^\vee$, write it as 
$$ \gamma = s(\beta)+\alpha, \quad \beta \in B^\vee, \: \alpha \in (A/B)^\vee. $$
Then by Corollary~\ref{cor:eb(z)}, we have 
$$ e_\gamma = b^{s(\beta)}(e_\alpha). $$
\end{rmk}

Abram and Kriz \cite{AbramKriz} observe that the isotropy fracture cube that arises from iteratively applying Proposition~\ref{prop:isotropyfracture} to the isotropy tower of $\MU_A$ is particularly well behaved, and this results in a very concrete general description of $\pi_*^A\MU_A$ as a limit of an explicit diagram of rings.  This diagram is indexed by the category
$$ \mr{Sd}\Sub(A) $$
which is obtained by applying barycentric subdivision to the Hasse diagram of the poset of subgroups of $A$.  Specifically, there are two types of objects of $\mr{Sd}\Sub(A)$:
\begin{enumerate}
\item objects $(B)$ for $B \le A$,
\item objects $(B_1 < B_2)$ for $B_1 \subsetneq B_2$ for which there does not exist a distinct intermediate subgroup $B_1 \subsetneq C \subsetneq B_2$.
\end{enumerate}
The non-identity morphisms are of the form
$$ (B_i) \to (B_1 < B_2), \quad i \in \{0,1\}. $$

We consider the diagram of graded rings
$$ \pi_*^A\Gamma(\MU_A): \mr{Sd}\Sub(A) \to \mr{GrRings} $$
which on objects is given by  
\begin{align*}
(B) & \mapsto \pi_*^A(\MU_A)_{T(B)}, \\
(B_1 < B_2) & \mapsto \pi_*^A (\MU_A)_{T(B_1), T(B_2)}
\end{align*} 
and on morphisms is given by the corresponding localization maps.

\begin{thm}[Abram-Kriz \cite{AbramKriz}]\label{thm:AbramKriz}
The canonical map
$$ \pi_*^A \MU_A \to \varprojlim_{\rm{Sd}\Sub(A)} \pi_*^A\Gamma(\MU_A) $$
is an isomorphism.
\end{thm}

To maximize the effectiveness of this theorem, we explain how Abram and Kriz use Proposition~\ref{prop:AbramKriz} to explicitly compute the diagram $\pi_*^A\Gamma(\MU_A)$.
We deduce from Proposition~\ref{prop:AbramKriz} that for $B_1 \subsetneq B_2$, and $s_1$ a section of $A^\vee \twoheadrightarrow B_1^\vee$, we have 
\begin{equation}\label{eq:MUATB1TB2}
 \pi_*^A (\MU_A)_{T(B_1),T(B_2)} \cong \left( \frac{\MU_*[e_{s_1(\beta)}^{\pm}, b_i^{s_1(\beta)}]_{1 \ne \beta \in B_1}[[e_\alpha]]_{\alpha \in (A/B_1)^\vee}}{(e_\alpha +_{\GG_1} e_{\alpha'} = e_{\alpha\alpha'})}[e^{-1}_{\beta'}]_{\beta'|_{B_2} \ne 1}\right)^\wedge_{(e_{\alpha''}\: : \: \alpha''|_{B_2}=1)}.
 \end{equation}
The map
$$ \pi_*^A(\MU_A)_{T(B_1)} \to \pi_*^A(\MU_A)_{T(B_1),T(B_2)} $$
is the obvious map from $\pi_*^A(\MU_A)_{T(B_1)}$ to the completion of its localization.
The map 
$$ \pi_*^A(\MU_A)_{T(B_2)} \to \pi^A_*(\MU_A)_{T(B_1),T(B_2)}, $$
whose explicit source and target are given as
\begin{multline}\label{eq:AbramKrizmap}
  \frac{\MU_*[e_{s_2(\beta)}^{\pm}, b_i^{s_2(\beta)}]_{1 \ne \beta \in B_2^\vee}[[e_\alpha]]_{\alpha \in (A/B_2)^\vee}}{(e_\alpha +_{\GG_1} e_{\alpha'} = e_{\alpha\alpha'})} \rightarrow \\
   \left( \frac{\MU_*[e_{s_1(\beta)}^{\pm}, b_i^{s_1(\beta)}]_{1 \ne \beta \in B_1^\vee}[[e_\alpha]]_{\alpha \in (A/B_1)^\vee}}{(e_\alpha +_{\GG_1} e_{\alpha'} = e_{\alpha\alpha'})}[e^{-1}_{\beta'}]_{\beta'|_{B_2} \ne 1}\right)^\wedge_{(e_{\alpha''}\: : \: \alpha''|_{B_2}=1)},
\end{multline} 
(where $s_i$ are sections of $A^\vee \twoheadrightarrow B_i^\vee$) is more subtle, and relies on the formulas which govern an equivariant formal group law.  For $\beta \in B_2^\vee$, write
$$ s_2(\beta) = s_1(\beta_1) + \alpha_1 $$
where $\alpha_1 \in (A/B_1)^\vee$.  Then, by Corollary~\ref{cor:eb(z)}, the map (\ref{eq:AbramKrizmap}) is given by
\begin{align*}
e_{s_2(\beta)} & \mapsto b^{s_1(\beta_1)}(e_{\alpha_1}) \\
b_i^{s_2(\beta)} & \mapsto \mr{coef}_{z^i}(b^{s_1(\beta_1)}(e_{\alpha_1}+_{\GG_1} z)) \\
e_\alpha & \mapsto e_\alpha, \quad \alpha \in (A/B_2)^\vee.
\end{align*}
Note that in the codomain of (\ref{eq:AbramKrizmap}), $b^{s_1(\beta_1)}(e_{\alpha_1})$ is invertible since $e_{s_1(\beta_1)}$ is invertible, the elements $e_\alpha$ for $\alpha \in (A/B_2)^\vee$ are topologically nilpotent, and the appropriate relation holds, so the map described above indeed exists.

\begin{ex}
In the case of $A = C_p$ with $p$ prime, Theorem~\ref{thm:AbramKriz} was proven by Kriz \cite{Kriz}.  In this case, Theorem~\ref{thm:AbramKriz} is asserting that the map on homotopy groups
$$
\xymatrix{
\pi_*^{C_p} \MU_{C_p} \ar[r] \ar[d] & 
\pi_* \MU^{\Phi C_p}_{C_p} \ar[d] 
\\
\pi_* \MU^{hC_p}_{C_p} \ar[r] & 
\pi_* \MU^{tC_p}_{C_p}
}
$$
is a pullback.
\end{ex}

\subsection*{Equivariant BP}

There is also an $A$-equivariant analog $BP_A$ of $BP$, introduced by May \cite{May} and studied in detail by Wisdom in \cite{Wisdom}.  May shows that $\MU_A$-modules are tensored over $\MU$-modules, and defines 
$$ \BP_A := \MU_A \wedge_\MU \BP. $$
Wisdom (implementing a proposal of Strickland) defines a \emph{$p$-typical equivariant formal group law} to be a an equivariant formal group law whose underlying non-equivariant formal group law is $p$-typical.  He observes that $p$-typicalization realizes to an idempotent of $(\MU_A)_{(p)}$, and the resulting summand is equivalent to $\BP_A$, and we have
$$ (\BP_A)_*X \simeq (\MU_A)_*X \otimes_{\MU_*} \BP_*. $$
In particular, all of the computations of this section carry over to $\BP_A$ by systematically replacing $\MU$ with $\BP$ everywhere.

\section{The Hausmann-Meier classification of invariant prime ideals}\label{sec:HM}

The classical chromatic picture linking the Balmer spectrum of $\Sp_{(p),\omega}$ to the prime ideals of the moduli stack $(\MFG{})_{(p)}$ of formal groups over $\Spec(\ZZ_{(p)})$ is predicated on two important theorems:
\begin{enumerate}
\item The Landweber Filtration Theorem \cite{Landweberfilt}, and
\item The Hopkins-Smith Thick Subcategory Theorem \cite{HS}. 
\end{enumerate}  
The $A$-equivariant analog of (2) is the computation of the Balmer spectrum $\Spc(A)$ by \cite{BalmerSanders}, \cite{sixauthor} discussed in Section~\ref{sec:Balmer}.  Recently, Hausmann and Meier \cite{HM} generalized the Landweber Filtration Theorem to the $A$-equivariant context, by computing the invariant prime ideals of $L_A$.  

\emph{For simplicity, we focus on the localization $L_{A,(p)}$, and for the remainder of the paper we assume $A$ is a $p$-group.}


\subsection*{The height filtration on $\pmb{(\MFG{A})_{(p)}}$}

Much like in the non-equivariant case, the classification of invariant prime ideals relies on the notion of \emph{height}.  For an abelian group-scheme $\GG$ we let $\GG[p]$ denote the sub-group scheme of $p$-torsion points.

\begin{defn}
If $\GG$ is a formal group or equivariant formal group over a field $k$ such that $p | \mr{char}(k)$, then we say $\GG$ has \emph{height} $n$ ($0 \le n \le \infty$) and write 
$$ \mr{ht}(\GG) = n $$
if the order of the group-scheme $\GG[p]$ is given by
$$ |\GG[p]| = p^{n}. $$
If $(\GG,\varphi)$ is an $A$-equivariant formal group over $k$ as above, then we define its $A$-\emph{height} by
$$ \rm{ht}_A(\GG,\varphi) := \rm{ht}(\GG_1) + \mr{rk}_p(A^\vee) - \mr{rk}_p(\ker \varphi). $$
\end{defn}

We make the following remarks (with $\GG$, $k$, as above).
\begin{enumerate}[label = (\theequation)]
\eitem The notion of height is compatible with the notion of height for a $p$-divisible group.

\eitem If $\GG$ is a formal group and $\mr{char}(k) = 0$ then $\mr{ht}(\GG) = 0$.
\eitem If $\GG$ is a formal group then $\mr{ht}(\GG) \ge 1$ if and only if $\mr{char}(k) = p$, and in this case $\mr{ht}(\GG) = \infty$ if and only if $\GG$ is additive.
\eitem If $(\GG, \varphi)$ is an $A$-equivariant formal group, then 
$$ \mr{ht}(\GG) = \mr{ht}(\GG_1) + \mr{rk}_p(B) $$
where $\ker \varphi = (A/B)^\vee \le A^\vee$.  This is analogous to a well-known formula for height of $p$-divisible groups.
\label{item:heightsum}
\eitem If $(\GG,\varphi)$ is an $A$-equivariant formal group, then
$$ \mr{ht}_A(\GG,\varphi) \le \mr{ht}(\GG) $$
and if
$$ \mr{rk}_p(A^\vee) = \mr{rk}_p(\ker \varphi) + \mr{rk}_p(\im \varphi)$$
then we have
$$ \mr{ht}_A(\GG, \varphi) = \mr{ht}(\GG). $$
In particular, this always holds if $A$ is an elementary abelian $p$-group.
\eitem If $B \le A$ and $(A/B)^\vee \subseteq \ker \varphi$, then $\varphi$ factors through $B^\vee$
$$ 
\xymatrix{
A^\vee \ar[rr]^\varphi \ar@{->>}[dr] && \GG \\
& B^\vee \ar@{.>}[ur]_{\varphi'}
}  
$$
and we may regard $(\GG,\varphi')$ as a $B$-equivariant formal group.  We define the $B$-\emph{height} of such $(\GG,\varphi)$ to be the $B$ height of $(\GG,\varphi')$:
$$ \mr{ht}_B(\GG,\varphi) := \mr{ht}_B(\GG,\varphi'). $$ 
If we have
$$ \mr{rk}_p(B^\vee) = \mr{rk}_p(\ker \varphi') + \mr{rk}_p(\im \varphi')$$
then we have
$$ \mr{ht}_B(\GG, \varphi) = \mr{ht}(\GG). $$
In particular, this always holds if $\ker \varphi = (A/B)^\vee$.
\end{enumerate}

We phrase the classification of invariant prime ideals of $L_{A,(p)}$ in terms of irreducible reduced closed substacks of $(\MFG{A})_{(p)}$.  In the non-equivariant case, Landweber \cite{Landweberfilt} proved that the irreducible reduced closed substacks of $(\MFG{})_{(p)}$ are given by
$$ V_n = \{ \GG \: : \: \mr{ht}(\GG) \ge n \}. $$

\begin{thm}[Hausmann-Meier {\cite[Thm.~4.7]{HM}}]\label{thm:HM}
The non-empty irreducible reduced closed substacks of $(\MFG{A})_{(p)}$ are given by 
$$ V_{(B,n)} := \left\{(\GG,\varphi) \: : \: 
\begin{array}{l}
(A/B)^\vee \subseteq \ker \varphi, \\
\mr{ht}_B(\GG) \ge n + \mr{rk}_p(B)
\end{array} 
\right\} $$
for $B \le A$ and $0 \le n \le \infty$.
\end{thm}

Note that it is elementary to verify from the definition that  
$$ V_{(B,n)} \le V_{(C,m)} $$
if and only if $B \le C$ and $n \ge m + \mr{rk}_p(C/B)$.  Thus Theorem~\ref{thm:HM} gives a homeomorphism 
$$ \Spec^{\rm{inv}}(L_{A,(p)}) \approx \Spc_{(p)}(A) $$
between the spectrum of invariant prime ideals of $L_{A,(p)}$ and the $p$-local Balmer spectrum. Henceforth we will implicitly identify these two spectra.

The notion of \emph{$B$-height} is a somewhat contrived notion introduced to describe the topology of the spectrum of invariant prime ideals of $L_{A,(p)}$.  The notion of height, by contrast, is a completely natural intrinsic aspect of an equivariant formal group law, extending the notion of height for $p$-divisible groups.  We end this subsection by noting that height provides a decreasing filtration of $(\MFG{A})_{(p)}$ by reduced closed substacks.  If $A$ is not elementary abelian, these are not irreducible.

\begin{prop}\label{prop:height}
The notion of height defines for $0 \le n < \infty$ reduced closed substacks $V^A_n \subseteq (\MFG{A})_{(p)}$ defined by
$$ V^A_n := \{ (\GG,\varphi) \: : \: \mr{ht}(\GG) \ge n \}. $$ 
\end{prop}

\begin{proof}
We simply must show that the associated subset of $V^A_n \subseteq \Spc_{p}(A)$ is closed.  
Recall from \ref{item:heightsum} that for $(\GG, \varphi)$ over a field we have
$$ \mr{ht}(\GG) = \mr{ht}(\GG_1) + \mr{rk}_p(B) $$
where $\ker \varphi = (A/B)^\vee$.  It follows that 
$$ V^A_n = \{ (B,m) \: : \: m \ge n - \mr{rk}_p(B) \} \subseteq \Spc_{(p)}(A). $$
To see that this is closed, we have to show that if $(B, m) \in V_n^A$ and $B' \le B$, then 
$$ (B',m+\mr{rk}_p(B/B')) \in V^A_n. $$
That is, we need to show
$$ m+\mr{rk}_p(B/B') \ge n - \mr{rk}_p(B'). $$ 
This follows from the inequalities
\begin{align*}
m & \ge n-\mr{rk}_p(B) \\
\mr{rk}_p(B) & \le \mr{rk}_p(B') + \mr{rk}_p(B/B').
\end{align*}
\end{proof}


\subsection*{$\pmb{v_{\ul{n}}}$-generators}

In the non-equivariant setting, the invariant prime ideal $I_n \unlhd L_{(p)} = \pi_*\MU_{(p)}$ corresponding to $V_n$ is generated by the regular sequence
\begin{equation}\label{eq:Ingen}
 I_n = (v_0, v_1, \ldots, v_{n-1}).
 \end{equation}
Here, the element $v_i \in \pi_{2(p^i-1)}\MU_{(p)}$ is non-canonical, but is determined up to a unit modulo $(v_0,v_1, \ldots, v_{i-1})$ by the formula
$$ [p]_{\GG_{\MU}}(x) \dotequiv v_i x^{p^i} + \cdots \mod (v_0, \cdots, v_{i-1}) $$ 
where $\GG_{\MU}$ is the universal formal group law.  In particular, $v_0 \doteq p$.

\begin{defn}\label{defn:Lvnelt}
Suppose that $R$ is a $L_{(p)}$-algebra, and let $0 \le n < \infty$.  We shall say that $x \in R$ is a \emph{$v_n$-generator} if
$$ x \dotequiv v_n \mod (v_0, \ldots, v_{n-1}). $$
We say that $x$ is a \emph{$v_{-1}$-generator} if $x = 0$, and we will say that $x$ is a \emph{$v_\infty$-generator} if $x$ is a unit. 
\end{defn}

Let $I_{(B,n)} \unlhd L_{A,(p)}$ be the invariant prime ideal associated to the irreducible closed substack $V_{(B,n)}$.
Hausmann and Meier give constructions of the generators of these ideals in \cite{HM}.

The key starting point is their construction, for $n \ge 0$, of elements
$$ \vv_n \in \pi^{C^{n+1}_p}_{2(p^n-1)}(\MU_{C^{n+1}_{p}})_{(p)} $$
with the property that for $B \lneq C_p^{n+1}$ with $\mr{rk}_p(B) = r$, the $B$-geometric fixed points of $\vv_n$ satisfy
\begin{equation}\label{eq:vnPhiB}
 \vv_n^{\Phi B} \dotequiv v_{n-r} \mod (p,v_1, \ldots, v_{n-r-1}) \in \pi_*(\MU^{\Phi B}_{C_p^{n+1}})_{(p)}
 \end{equation}
and 
$$ \vv_n^{\Phi C_p^{n+1}} = 0 \in \pi_*\MU_{C_p^{n+1}}^{\Phi C_p^{n+1}}. $$

\begin{rmk}
When considering (\ref{eq:vnPhiB}), the reader should keep in mind that in $\pi_*(\MU^{\Phi B}_{C_p^{n+1}})_{(p)}$ with $B \ne e$, there are many units given by Euler classes which are not in degree zero.
\end{rmk}

This key property of the elements $\vv_n$ lead us to make the following definitions, which are directly adapted from \cite{HM}.

\begin{defn}\label{defn:htfunc} 
A \emph{height function on $\Sub(A)$} is a function
$$ \ul{n}: \Sub(A) \to \Nbar_-, $$
where
$$ \Nbar_- := \Nbar \cup \{\infty\} \cup \{-1\}. $$
\end{defn}

\begin{defn}
Let $\ul{n}$ be a height function on $\Sub(A)$.  We will say an element $v \in \pi^A_*\MU_{A,(p)}$ is a \emph{$v_{\ul{n}}$-generator} if for every $B \le A$, the $B$-geometric fixed points of $v$ $$ v^{\Phi B} \in \pi_* \MU^{\Phi B}_{A,(p)} $$
is a $v_{\ul{n}(B)}$-generator.
\end{defn}

\begin{exs}\label{ex:vnelts} $\quad$
\begin{enumerate}
\item
We see that from (\ref{eq:vnPhiB}) that the Hausmann-Meier elements $\vv_n \in (L_{C_p^{n+1}})_{(p)}$ are $v_{h[n]}$-generators where $h[n]$ is the height function
$$ h[n](B) = 
n-\mr{rk}_p(B).
$$

\item 
Letting $A$ be arbitrary, and taking $\ul{n} = c[n]$, the constant height $n$ function ($n \ge 0$) with
$$ c[n](B) = n, \quad B \le A, $$
then any $v_n$-generator in $L_{(p)} \subseteq L_{A,(p)}$ is a $v_{c[n]}$-generator. 

\item 
Given a character $\alpha \in A^\vee$, define a height function $\ul{n}_{\alpha}$ by 
$$
\ul{n}_\alpha(B) = 
\begin{cases}
-1,  & \alpha|_B = 1, \\
\infty, & \alpha|_B \ne 1.
\end{cases}
$$
Then $e_\alpha \in L_{A,(p)}$ is a $v_{\ul{n}_\alpha}$-generator.
\end{enumerate}
\end{exs}

In \cite{HM}, Hausmann-Meier construct a variety of $v_{\ul{n}}$-generators from their elements $\bar{\mbf{v}}_{n}$ as follows.  Suppose
$$ f: A \to A' $$ 
is a group homomorphism, and let
$$ f^*: L_{A'} \to L_{A} $$
denote the induced map of rings.  There is also an induced map
\begin{align*}
 f_*: \Sub(A) & \to \Sub(A') \\
 B & \mapsto f(B)
\end{align*}
which allows us to define, for $\ul{n}$ a height function on $\Sub(A')$ a new height function $f^* \ul{n}$ on $\Sub(A)$ by taking the composite
$$ \Sub(A) \xrightarrow{f_*} \Sub(A') \xrightarrow{\ul{n}} \Nbar_-. $$
Then we have the following.

\begin{prop}
Suppose that $f: A \to A'$ is a homomorphism, and that $v \in L_{A'}$ is a $v_{\ul{n}}$-generator.  Then $f^* v$ is a $v_{f^* \ul{n}}$-generator.
\end{prop}

The  $v_{\ul{n}}$-generators are connected to the topology of the Balmer spectrum as follows.  

\begin{defn}\label{defn:admissiblehtfunc}
Let $\ul{n}: \Sub(A) \to \Nbar_-$ be a height function. Define a subset $U_{\ul{n}} \subseteq \Spc_{(p)}(A)$ by 
$$
U_{\ul{n}} := \{ \mc{P}_{(B,m)} \: : \: 0 \le m \le \ul{n}(B) \}.
$$
We will say that $\ul{n}$ is \emph{admissible} if $U_{\ul{n}}$ is open.  Thus $\ul{n}$ is admissible if and only if for $B < C \le A$ we have
$$ \ul{n}(B) \le \ul{n}(C) + \mr{rk}_p(C/B). $$ 
\end{defn}

The topology of the Balmer spectrum puts the following limitation on the existence of $v_{\ul{n}}$-generators.

\begin{prop}[Hausmann-Meier (see proof of {\cite[Thm.~7.5]{HM}})]
Let $v \in L_{A,(p)}$ be a $v_{\underline{n}}$-generator, and let $V(v)$ be the associated closed subset of the Zariski spectrum $\Spec^\mr{inv}(L_{A,(p)})$:
$$ V(v) = \{ \mc{P} \in \Spec^{\mr{inv}}(L_{A,(p)}) \: : \: v \in \mc{P} \} $$
Then under the homeomorphism $\Spc_{(p)}(A) \approx \Spec^{\mr{inv}}(L_{A,(p)})$ we have
$$ V(v) = U_{\ul{n}}^c. $$
In particular, the height function $\ul{n}$ is admissible.   
\end{prop}

\section{Equivariant periodicity}\label{sec:periodicity}

\subsection*{The equivariant Nilpotence Theorem}

Recall the Devinatz-Hopkins-Smith Nilpotence Theorem:

\begin{thm}[Devinatz-Hopkins-Smith, \cite{DHS}]\label{thm:DHS}
Suppose that $R \in \Sp$ is a homotopy ring spectrum.  Then $x \in \pi_*R$ is nilpotent if and only if $\MU \wedge x \in \MU_* R$ is nilpotent.
\end{thm} 

As observed in \cite{BGH}, the following equivariant form of the Nilpotence Theorem follows fairly easily.  We add a proof just to emphasize to the reader how immediately this follows from the classical Nilpotence Theorem and the fact that the geometric fixed points spectra $(\MU_A)^{\Phi B}$ are free $\MU$-modules.

\begin{thm}[Barthel-Greenlees-Hausmann \cite{BGH}]\label{thm:nilp}
Suppose that $R \in \Sp^A$ is a homotopy ring spectrum.  Then the following are equivalent.
\begin{enumerate}
\item $x \in \pi^A_\star R$ is nilpotent. 
\item $\MU_A \wedge x \in (\MU_A)_\star R$ is nilpotent.
\item $x^{\Phi B} \in \pi_* R^{\Phi B}$ is nilpotent for all $B \le A$.
\item $K(n) \wedge x^{\Phi B} \in K(n)_*R^{\Phi B}$ is nilpotent for all $0 \le n \le \infty$ and $B \le A$.
\end{enumerate}
\end{thm}

\begin{proof}
Clearly (1) $\Rightarrow$ (2), (3), (4).  Suppose that $\MU_A \wedge x$ is nilpotent with $x \in \pi^A_\gamma R$.  Then it follows that 
$$ \MU_A \wedge x^{-1}R \simeq \ast. $$
This implies that for all $B \le A$, we have
$$ (\MU_A \wedge x^{-1}R)^{\Phi B} \simeq \ast. $$
Let 
$$ x^{\Phi B} \in \pi_{|\gamma^B|}R^{\Phi B} $$
denote the geometric fixed points of $x$.  Then we have
\begin{align*}
\ast & \simeq (\MU_A \wedge x^{-1}R)^{\Phi B} \\
& \simeq \MU_A^{\Phi B} \wedge (x^{-1}R)^{\Phi B} \\
& \simeq \MU_A^{\Phi B} \wedge \Phi^B(x)^{-1}R^{\Phi B} \\
& \simeq (x^{\Phi B})^{-1}(\MU_A^{\Phi B} \wedge R^{\Phi B}).
\end{align*}
It follows that 
$$ \MU_A^{\Phi B} \wedge x^{\Phi B} \in \pi_{|\gamma^B|}(\MU_A^{\Phi B} \wedge R^{\Phi B}) $$
is nilpotent.  Since by Corollary~\ref{cor:MUAPhiB}, $\MU_A^{\Phi B}$ is a wedge of $\MU$'s, it follows from the Nilpotence Theorem that 
$$ x^{\Phi B} \in \pi_{|\gamma^B|}(R^{\Phi B}) $$
is nilpotent.  Thus (2) $\Rightarrow$ (3).  (3) implies that 
$$
(x^{-1}R)^{\Phi B} \simeq (x^{\Phi B})^{-1} R^{\Phi B} \simeq \ast
$$
for all $B \le A$.  It therefore follows that $x^{-1}R \simeq \ast$, and therefore $x$ is nilpotent.  Thus (3) $\Rightarrow$ (1).  We get (4) equivalent to (3) by \cite[Cor.~5]{HS}.
\end{proof}

Thus it follows that $\MU_A$ detects all non-nilpotent self maps, as expressed in the following Corollary.

\begin{cor}\label{cor:selfmap}
Suppose that $X \in \Sp^A$ is finite and $v : \Sigma^\gamma X \to X$ is a self-map.  Then $v$ is nilpotent if and only if $\MU_A \wedge v$ is nilpotent.  In particular, if $v$ is non-nilpotent, then $\MU_A \wedge v$ is non-zero.  
\end{cor}

\begin{proof}
Take $R$ to be $DX \wedge X$ and $x$ to be the adjoint of $v$.
\end{proof}

\subsection*{Non-equivariant periodicity}

Recall the following fundamental non-equivariant definitions from \cite{HS}. 

\begin{defn}
For  $X \in \Sp_{(p),\omega}$: 
\begin{enumerate}
\item For $0 \le n < \infty$, we say $X$ is \emph{type $n$} if $K(n)\wedge X \ne \ast$ and $K(n-1) \wedge X \simeq \ast$.  If $X \simeq \ast$ we shall say $X$ is \emph{type $\infty$}.

\item For $1 \le n < \infty$, a self-map 
$$ v: \Sigma^k X \to X $$
is called a
\begin{itemize}
\item \emph{$v_n$-self map} if $K(n)\wedge v$ is an equivalence, and $K(i) \wedge v$ is nilpotent for $i \ne n$,  
\item \emph{$v_\infty$-self map} if some power of $v$ is multiplication by an element of $\ZZ_{(p)}^\times$, 
\item \emph{$v_0$-self map} if some power of $v$ is multiplication by an element $\lambda \in \ZZ_{(p)}$ with $p$-adic valuation $1 \le \nu_p(\lambda) < \infty$, 
\item \emph{$v_{-1}$-self map} if $v$ is nilpotent. 
\end{itemize}
\end{enumerate}
\end{defn}

\begin{rmk}
The notion of $v_0$-self map of \cite{HS} differs slightly from our notion, in the sense that it combines our notions of $v_0$ and $v_\infty$-self map.
\end{rmk}

The Hopkins-Smith periodicity theorem links the notion of ``$v_n$-self map'' and  ``type $n$ complex''.

\begin{thm}[Hopkins-Smith \cite{HS}]
Suppose $0 \le n < \infty$.
A complex $X \in \Sp_{(p),\omega}$ admits a non-nilpotent $v_n$-self map if and only if it is type $n$.  
\end{thm} 

\begin{rmk}\label{rmk:vnmaps}
Note that it follows from the Nilpotence Theorem (Theorem~\ref{thm:DHS}) that if $X \in \Sp_{(p),\omega}$ is a type $n$ complex ($0 \le n < \infty$), then for $-1 \le k < n$, a $v_k$-self map $v$ of $X$ is the same thing as a $v_{-1}$-self map.  Therefore every $v_m$-self map of $X$ is exactly one of the following:
\begin{itemize}
\item $v_{-1}$-self map (nilpotent),
\item $v_n$-self map,
\item $v_\infty$-self map (equivalence).
\end{itemize}
Of course if $X$ is type $\infty$, then the unique self map of $X$ is a $v_i$-self map for all $-1 \le i \le \infty$, and in particular every $v_m$-self map is a $v_{-1}$-self map.
\end{rmk}

\subsection*{$\pmb{v_{\ul{n}}}$-self maps}

We now seek to extend this story to $\Sp^A_{(p)}$.
Equivariant periodic self-maps were introduced and studied by Bhattacharya-Guillou-Li \cite{BGL}.  They considered self-maps
$$ v : \Sigma^\gamma X \to X $$
for which 
$$ v^{\Phi B}: \Sigma^{\gamma^B} X^{\Phi B} \to X^{\Phi B} $$
was a $v_{n_B}$-self map as $B$ ranges over the subgroups of $B$, and observed that the topology of the Balmer spectrum put constraints on the numbers $n_B$ (when defined).  To make this idea precise, we will encode the sequences of numbers $n_B$ in a height function
$$ \ul{n}: \Sub(A) \to \Nbar_-. $$
We shall define the \emph{non-negative domain} of $\ul{n}$ to be the places where $\ul{n} \ge 0$:
$$ \mr{dom}_{\ge 0}(\ul{n}) := \{ B \le A \: : \: \ul{n}(B) \ge  0 \} \subseteq \Sub(A). $$
From this perspective, we define, for 
$$ S \subseteq \mr{dom}_{\ge 0}(\ul{n}) $$
the \emph{restriction} of $\ul{n}$ to $S$, which we will denote $\ul{n}|_S$, to be the height function
\begin{equation}\label{eq:restriction}
 \ul{n}|_S(B) := \begin{cases}
\ul{n}(B), & B \in S, \\
-1, & \text{otherwise}, 
\end{cases}
\end{equation}
so that $\mr{dom}_{\ge 0}(\ul{n}|_S) = S$.

\begin{defn}
For a height function $\ul{n}$,
a \emph{$v_{\ul{n}}$-self map} of a $p$-local finite complex $X \in \Sp^A_{(p),\omega}$
is a self-map
$$ v: \Sigma^\gamma X \to X $$
(for $\gamma \in RO(A)$) so that
for all $B \le A$, the geometric fixed points 
$$ v^{\Phi B}: \Sigma^{\gamma^B} X^{\Phi B} \to X^{\Phi B} $$ 
is a $v_{\ul{n}(B)}$-self map.
\end{defn}

\begin{exs}\label{ex:vnmaps}$\quad$
\begin{enumerate}
\item If $\alpha \in A^\vee$ is a character, then the Euler class 
$$ a_\alpha : \Sigma^{-\alpha} X \to X $$ 
gives a $v_{\ul{n}_\alpha}$-self map of any $X \in \Sp^A_{(p),\omega}$, where $\ul{n}_\alpha$ is the height function of Example~\ref{ex:vnelts}(3).

\item If $X \in \Sp_{(p),\omega}$ is a non-equivariant type $n$ complex, and
$$ v: \Sigma^k X \to X $$
is a $v_n$-self map, then 
$$ v_{\mr{triv}}: \Sigma^k X_{\mr{triv}} \to X_{\mr{triv}} $$
is a $v_{c[n]}$-self map, where $c[n]$ is the constant height function at $n$ of Example~\ref{ex:vnelts}(2).  We will simply refer to this as a \emph{$v_n$-self map}.
\end{enumerate}
\end{exs}

For a fixed height function $\ul{n}$, it will be convenient to be able to discuss $v_{\ul{m}}$-self maps as $\ul{m}$ ranges over the restrictions of $\ul{n}$ to the various subsets $S \subseteq \mr{dom}_{\ge 0}(\ul{n})$.

\begin{defn}
For a height function $\ul{n}$, and $S \subseteq \mr{dom}_{\ge 0}(A)$, a \emph{$v^S_{\ul{n}}$-self map} is a $v_{\ul{n}|_S}$-self map.
\end{defn}

As a special case,  
if $S = \{B\}$, then $\ul{n}|_S$ is determined by the single value $\ul{n}(B) = n$.
For such $\ul{n}$ we make the following more specialized definition:

\begin{defn}
Suppose that $B \le A$.  A \emph{$v^B_{n}$-self map} of a $p$-local finite complex $X \in \Sp^A_{(p),\omega}$ is a self-map
$$ v: \Sigma^\gamma X \to X $$
so that $v^{\Phi B}$ is a $v_n$-self map, and for all $B \ne C \le A$, $v^{\Phi C}$ is nilpotent.
\end{defn}

\begin{ex}
In the case of $A = C_p$ we may represent a height function as a tuple of values
$$ \ul{n} = (n,m) $$
where $n = \ul{n}(e)$ and $m = \ul{n}(C_p)$.
In this situation we may talk about $v_{(n,m)}$-self maps, and for $0 \le n,m < \infty$, this agrees with the notation of \cite{BGL}.  However, in the case where one of these values is $-1$, the notions of $v_{(-1,m)}$-self map or $v_{(n,-1)}$-self map where $0 \le m,n < \infty$, correspond to the notions of  $v_{(\mr{nil},m)}$-self map and $v_{(n,\mr{nil})}$-self map of \cite{BGL}. In our notation, a $v_{(-1,m)}$-self map is the same thing as a $v^{C_p}_m$-self map, and a $v_{(n,-1)}$-self map is a $v_n^e$-self map. 
\end{ex}

\subsection*{Type $\pmb{\ul{n}}$ complexes}

As discussed in \cite{BGH}, certain closed subsets of $\Spc_{(p)}(A)$ correspond to the ``types'' of equivariant finite complexes through the notion of support (\ref{eq:supp}). These types are parameterized by \emph{type functions}:

\begin{defn}
A \emph{type function} is a function
$$ \ul{n} : \Sub(A) \to \Nbar. $$
A finite complex $X \in \Sp^A_{(p),\omega}$ is said to be \emph{type $\ul{n}$} if for all $B \le A$, the geometric fixed points $X^{\Phi B}$ is type $\ul{n}(B)$.  The \emph{finite domain} of $\ul{n}$ is the subset of $\Sub(A)$ where $\ul{n}$ is finite:
$$ \mr{dom}_{< \infty}(\ul{n}) := \{ B \le A \: : \: \ul{n}(B) < \infty \} \subseteq \Sub(B).  $$ 
The infinite domain is the complement:
$$ \mr{dom}_{= \infty}(\ul{n}) = \{ B \le A \: : \: \ul{n}(B) = \infty \} \subseteq \Sub(B). $$
\end{defn} 

\begin{rmk}
Just as in the non-equivariant case, where heights of $v_n$-self maps range over $\Nbar_-$ and types of $p$-local finite complexes range over $\Nbar$, height functions take values in $\Nbar_-$ while type functions take values in $\Nbar$.  Note that if $\ul{n}$ is a type function, then $\ul{n}-1$ is a height function.
\end{rmk}

Only certain type functions $\ul{n}$ are realized as type functions of actual equivariant complexes.  In \cite{BGH}, these type functions are shown to be precisely the \emph{admissible} type functions.

\begin{defn}\label{defn:adtype} $\quad$
\begin{enumerate}
\item
A type function
$$\ul{n}: \Sub(A) \to \Nbar $$
is \emph{admissible} if the subset $V_{\ul{n}} \subseteq \Spc_{(p)}(A)$ given by
$$ V_{\ul{n}} = \left\{ \mc{P}_{(B,i)} \: : \: \begin{array}{l} B \in \mr{dom}_{< \infty}(\ul{n}), \\ i \ge \ul{n}(B) \end{array} \right\} $$
is closed.  Said differently, the type function $\ul{n}$ is admissible if and only if the associated height function $\ul{n}-1$ is admissible in the sense of Definition~\ref{defn:admissiblehtfunc}.

\item
A closed subset $F \subseteq \Spc_{(p)}(A)$ will be called \emph{admissible closed} if $F = V_{\ul{n}}$ for $\ul{n}$ an admissible type function.

\item An open subset $U \subseteq \Spc_{(p)}(A)$ will be called \emph{admissible open} if $U = V^c_{\ul{n}}$ for $\ul{n}$ admissible.  Note that this is the same as asserting that $U = U_{\ul{n}-1}$ in the sense of Definition~\ref{defn:admissiblehtfunc}, so $U$ is admissible open if and only if $U = U_{\ul{m}}$ for an admissible height function $\ul{m}$.
\end{enumerate}
\end{defn}

\begin{rmk}
It may initially seem confusing that there is a notion of admissibility for type functions and height functions, when type functions are really just height functions that never take on the value of $-1$.  However, the two notions of admissibility are compatible: a type function $\ul{n}$ is admissible (as a type function) if an only if it is admissible when regarded as a height function, or equivalently, if for all $B < C \le A$ we have
$$  \ul{n}(B) \le \ul{n}(C) + \mr{rk}_p(C/B). $$
The distinction between the two is meant more as a convenience --- height functions parameterize $v_{\ul{n}}$-generators and admissible open subsets of $\Spc_{(p)}(A)$, whereas type functions parameterize types of complexes and admissible closed subsets of $\Spc_{(p)}(A)$.
\end{rmk}

We take note of the following things:
\begin{enumerate}[label=(\theequation)]
\eitem The admissible closed subsets of $\Spc_{(p)}(A)$ are precisely those subsets which are of the form $\mr{supp}(X)$ for some $X \in \Sp^A_{(p),\omega}$, and form a basis of the topology on $\Spc_{(p)}(A)$ \cite{BGH}.

\eitem A complex $X \in \Sp^A_{(p),\omega}$ is type $\ul{n}$ if and only if 
$$ \mr{supp}(X) = V_{\ul{n}} $$
(and in this case, by the previous remark, $\ul{n}$ must be admissible).

\eitem The tt-ideals $I \trianglelefteq \Sp^A_{(p),\omega}$ are precisely those of the form
$$ I_{\ul{n}} := \{X \in \Sp^A_{(p),\omega} \: : \: \mr{supp}(X) \subseteq V_{\ul{n}} \} $$
for $\ul{n}$ an admissible type function \cite{BGH}.

\eitem For a family $\mc{F} \subseteq \Sub(A)$, the closed subset $V(\mc{F}) \subseteq \Spc_{(p)}(A)$ and the open subset $U(\mc{F}^c) \subseteq \Spc_{(p)}(A)$ of (\ref{eq:V(F)}), (\ref{eq:U(Fc)}) are admissible.  If we define
$$ \ul{n}_\mc{F} = 
\begin{cases}
0, & B \in \mc{F}, \\
\infty, & B \not\in \mc{F} 
\end{cases} 
$$
then
$$ V(\mc{F}) = V_{\ul{n}_\mc{F}}. $$

\eitem Because we are assuming $A$ is a $p$-group, if $\ul{n}$ is an admissible type function, then the subset $\mr{dom}_{<\infty}(\ul{n}) \subseteq \Sub(A)$ is a family.
\end{enumerate}

\subsection*{Existence and properties of $\pmb{v_{\ul{n}}}$-self maps}

We say that a height function $\ul{m}$ is \emph{finite} and write
$$ \ul{m} < \infty $$
if $\ul{m}(B) < \infty$ for all $B \le A$.

The question is:
\begin{question}\label{ques:v_m}
For $X \in \Sp^A_{(p),\omega}$ of type $\ul{n}$, for which height functions $\ul{m} < \infty$ does $X$ possess a $v_{\ul{m}}$-self map?
\end{question}

\begin{rmk}
The case where $\ul{m}$ has infinite values is somewhat muddled by the fact that these infinite values can be manufactured by Euler classes (see Example~\ref{ex:vnmaps}(1)), so we sidestep it here, as it just complicates the statements.
\end{rmk}

The following follows immediately from Remark~\ref{rmk:vnmaps}.

\begin{lem}
Suppose $X \in \Sp^A_{(p),\omega}$ is type $\ul{n}$, and that
$$ v : \Sigma^\gamma X \to X $$
is a $v_{\ul{m}}$-self map with $\ul{m} < \infty$.  Then $\ul{m} \le \ul{n}$, and for 
$$ S = \{ B \le A \: : \: \ul{n}(B) = \ul{m}(B) \} $$
$v$ is a $v^S_{\ul{n}}$-self map.
\end{lem}

Thus Question~\ref{ques:v_m} reduces to the following more refined question.

\begin{question}\label{ques:v_n|S}
For $X$ of type $\ul{n}$, for which $S \subseteq \mr{dom}_{< \infty}(\ul{n})$ does $X$ admit a $v^S_{\ul{n}}$-self map?
\end{question}

Let $\chi_S$ denote the characteristic function for $S \subseteq \Sub(B)$:
$$ \chi_S(B) = \begin{cases}
1, & B \in S, \\
0, & B \not\in S. 
\end{cases}
$$

\begin{lem}
If $X \in \Sp^{A}_{(p),\omega}$ is a type $\ul{n}$ complex, and $v$ is a $v^S_{\ul{n}}$-self map for $S \subseteq \mr{dom}_{<\infty}(\ul{n})$, then its cofiber $X/v$ is of type $\ul{n}+\chi_S$.
\end{lem}

We therefore have the following generalization of \cite[Prop.~1.9]{BGL}. 

\begin{prop}\label{prop:obstruction}
Suppose $X \in \Sp^A_{(p),\omega}$ is of type $\ul{n}$.  Then if $S \subseteq \mr{dom}_{< \infty}(\ul{n})$ and $\ul{n}+\chi_S$ is not admissible, then $X$ does not possess a $v^S_{\ul{n}}$-self map.
\end{prop}

One may optimistically speculate that Proposition~\ref{prop:obstruction} represents the only obstruction for the existence of a $v^S_{\ul{n}}$-self map on a type $\ul{n}$ complex\footnote{Forthcoming work of Burklund, Hausmann, Levy and Meier addresses the existence of more general $v_{\underline{n}}$-self maps.} --- the authors know of no counterexamples to such a conjecture.  We can only record some relatively straightforward special cases which derive from the non-equivariant periodicity theorem.

In \cite{HS}, Hopkins and Smith deduce many excellent properties of $v_n$-self maps for $1 \le n < \infty$ from the Nilpotence Theorem (\ref{thm:DHS}).  Many of their deductions involve the following lemma.

\begin{lem}[Hopkins-Smith {\cite[Lem.~3.4]{HS}}]\label{lem:HSlem}
Suppose that $x$ and $y$ are commuting elements of a $\ZZ_{(p)}$-algebra.  If $x - y$ is both torsion and nilpotent, then for $N \gg 0$, 
$$ x^{p^N} = y^{p^N}. $$
\end{lem}   

Lemma~\ref{lem:HSlem} does not apply to help us understand self-maps which are not torsion.  In the non-equivariant context, this does not have a great effect, because the only self-maps which are not necessarily torsion are $v_0$ and $v_\infty$-self maps, and these are asymptotically understandable by definition.  This is no longer holds in the equivariant context: if $v$ is a $v_{\ul{n}}$-self map, where $\ul{n}(B) = 0$ for some $B$, then $v_{\ul{n}}$ is necessarily not torsion, yet in the case where $\ul{n} \not\equiv 0$, the notion of $v_{\ul{n}}$-self map is highly non-trivial.
 
We will therefore have to restrict our observations to $v_{\ul{n}}$ with $\ul{n}(B) \not\in \{ 0, \infty\}$ (see Lemma~\ref{lem:torsion} below).  We shall call such height functions \emph{finite, non-zero} height functions.

Recall from \cite{HS} that for $R \in \Sp_{(p)}$ a homotopy ring spectrum, an element $x \in \pi_*R$ is said to be a \emph{$v_n$-element} $1\le m < \infty$ if
$$ K(m) \wedge x \in  K(m)_* R $$
is a unit if $m = n$, and nilpotent if $m \ne n$, $0 \le m \le \infty$.  We shall say $x$ is a \emph{$v_0$-element} if some power of $x$ is equal to the image of 
$$ \lambda \in \ZZ_{(p)} \to \pi_0 R $$ 
with $1 \le \nu_p(\lambda) < \infty$, $x$ is a \emph{$v_\infty$-element} if some power of $x$ is in the image of 
$$ \ZZ^\times_{(p)} \to (\pi_0 R)^\times $$
and $x$ is a  
\emph{$v_{-1}$-element} if it is nilpotent, which, by \cite[Thm.~3]{HS}, is equivalent to $K(m) \wedge x$ being nilpotent for all $0 \le m \le \infty$.

\begin{defn}\label{defn:Rvnelt}
Suppose that $R \in \Sp^A_{(p)}$ is a homotopy ring spectrum, and $\ul{n}$ is a height function on $\Sub(A)$.  We will say that $x \in \pi^A_\star R$ is a \emph{$v_{\ul{n}}$-element} if $x^{\Phi B} \in \pi_*R^{\Phi B}$ is a $v_{\ul{n}(B)}$-element for all $B \le A$. 
\end{defn}

Essentially following \cite{HS}, \cite[Lec.~27]{Lurienilp} verbatim, we may now deduce equivariant analogs of Hopkins and Smith's observations concerning $v_n$-self maps.

\begin{prop}\label{prop:vnmult}
Suppose that $R \in \Sp^A_{(p)}$ is a finite homotopy ring spectrum, $\ul{n}$ is a height function on $\Sub(A)$, and $x \in \pi^A_\star R$ is a $v_{\ul{n}}$-element.  Then there exist integers $i$ and $\{j_B\}_{B \le A}$ so that
$$
K(m) \wedge (x^i)^{\Phi B} = 
\begin{cases}
v_{\ul{n}(B)}^{j_B}, & 1 \le \ul{n}(B) < \infty, m = \ul{n}(B), \\
sp^{j_B},  & m = \ul{n}(B) = 0, s \in \ZZ^\times_{(p)}, j_B \ge 1, \\
s \in \ZZ^\times_{p}, & \ul{n}(B) = \infty, \\
0, & m \ne \ul{n}(B) < \infty. 
\end{cases}
$$ 
\end{prop}

The following elementary observation will be of crucial importance.

\begin{lem}\label{lem:torsion}
Suppose that $R \in \Sp^A_{(p)}$ is a homotopy ring spectrum, $\ul{n}$ is a finite non-zero height function on $\Sub(A)$, and $x \in \pi^A_\star R$ is a $v_{\ul{n}}$-element. Then some power of $x$ is torsion.
\end{lem}

\begin{proof}
Suppose that $x^N$ is not torsion for all $N$.  This implies that $x^N \ne 0 \in \pi^A_\star R_\QQ$ for all $N$, which is equivalent to saying that 
$$ x^{-1}R_\QQ \not\simeq \ast. $$
This is equivalent to there existing a $B \le A$ with
$$ (x^{\Phi B})^{-1}R^{\Phi B}_\QQ \not\simeq \ast. $$
But this would mean that $K(0) \wedge x^{\Phi B}$ was not nilpotent, which contradicts our hypotheses.
\end{proof}

\begin{prop}\label{prop:vncen}
Suppose that $R \in \Sp^A_{(p)}$ is a finite homotopy ring spectrum, $\ul{n}$ is a finite, non-zero height function on $\Sub(A)$, and $x \in \pi^A_\star R$ is a $v_{\ul{n}}$-element.  
Then some power of $x$ is central.
\end{prop}

\begin{proof}
Using Theorem~\ref{thm:nilp} and Lemma~\ref{lem:torsion}, the same proof as that in \cite[Lem.~3.5]{HS} applies.
\end{proof}

To establish asymptotic uniqueness of $v_{\ul{n}}$-elements, we must deal with a subtlety which is not present in the non-equivariant case.  We are thankful to Shangjie Zhang for pointing out this subtlety to the authors.

\begin{defn}\label{defn:dir}
Let $\gamma$ be an element of $RO(A)$.  We say that a $v_{\ul{n}}$-element $x \in \pi^A_{\gamma'} R$ (respectively a $v_{\ul{n}}$-self map $v: \Sigma^{\gamma'}X \to X$) has \emph{direction $\gamma$} if there exist $i,j \in \mathbb{N}_{> 0}$ so that $i\gamma = j\gamma'$.
\end{defn}

Following \cite[Prop.3.6]{HS}, we have the following asymptotic uniqueness result.

\begin{prop}\label{prop:asymu}
Suppose that $\ul{n}$ is a finite, non-zero height function and $x \in \pi^A_{\gamma'} R$ and $y \in \pi^A_{\gamma''} R$ are $v_{\ul{n}}$-elements with a common direction $\gamma$.  Then there exists $i,j$ so that $x^i = y^j$.
\end{prop}

\begin{proof}
The assumption that $x$ and $y$ have the same direction, combined with Proposition~\ref{prop:vnmult}, implies that there exists $i$ and $j$ so that for each $B$ and $m$ we have 
$$ K(m) \wedge (x^i-y^j)^{\Phi B} = 0. $$
The result then follows from Theorem~\ref{thm:nilp}, Lemma~\ref{lem:torsion}, and Lemma~\ref{lem:HSlem}.
\end{proof}

Letting $R = X \wedge DX$ for $X \in \Sp^A_{(p),\omega}$, following \cite[Cor.~3.3, 3.5, 3.7-8]{HS} we have:

\begin{cor}\label{cor:vnselfmap}
Analogous results to Propositions~\ref{prop:vnmult} and \ref{prop:vncen}, and \ref{prop:asymu} apply with ``$v_{\ul{n}}$-element'' replaced by ``$v_{\ul{n}}$-self map.''
\end{cor}

We deduce the following theorem.

\begin{thm}
Suppose that $\ul{n}$ is a finite, non-zero height function on $\Sub(A)$ and fix $\gamma \in RO(A)$.  The full subcategory $\mc{V}^{\gamma}_{\ul{n}}$ of $\Sp^A_{(p),\omega}$ consisting of $X$ which admit a $v_{\ul{n}}$-self map in the direction $\gamma$ is a tt-ideal.
\end{thm}

\begin{proof}
The argument of \cite[Cor.~3.9]{HS} applies to show that $\mc{V}^{\gamma}_{\ul{n}}$ is thick.  Suppose that $X \in \mc{V}^\gamma_{\ul{n}}$, and let
$$ v : \Sigma^{\gamma'} X \to X $$
be a $v_{\ul{n}}$-self map.  Then if $Y \in \Sp^A_{(p),\omega}$, then
$$ v \wedge Y: \Sigma^{\gamma'} X \wedge Y \to X \wedge Y $$
is a $v_{\ul{n}}$-self map.
\end{proof}

The following is an immediate consequence of the classification of tt-ideals of $\Sp^A_{(p),\omega}$ of \cite{BGH}. 

\begin{cor}\label{cor:single}
Suppose $\ul{n}$ is an admissible height function, fix $\gamma \in RO(A)$, and suppose that $S \subseteq \Sub(A)$ satisfies 
$$ \ul{n}|_S \: \text{is finite and non-zero}. $$
If there exists a single example of a finite type $\ul{n}$ complex with a $v^S_{\ul{n}}$-self map with direction $\gamma$, then every finite type $\ul{n}$ complex admits a $v^S_{\ul{n}}$-self map with direction $\gamma$.
\end{cor}

As a sample application we have the following limited version of the Hopkins-Smith periodicity theorem.

\begin{thm}\label{thm:vnB}
Suppose $X \in \Sp^A_{(p),\omega}$ is of type $\ul{n}$ with   
$$ \mr{dom}_{< \infty}(\ul{n}) = \mc{F}_{\subseteq B} $$
and $n := \ul{n}(B) \ne 0$. 
Then $X$ possesses a $v^B_{n}$-self map.
\end{thm}

This theorem follows from the following lemma, which provides the single example we need to apply Corollary~\ref{cor:single}.

\begin{lem}
Suppose that $X \in \Sp^A_{(p),\omega}$ has $X^{\Phi B}$ of type $n < \infty$.  Then there exists a $v^B_n$-self map on $A/B_+ \wedge X$.
\end{lem}

\begin{proof}
By hypothesis, $X^{\Phi B}$ is a type $n$ complex, and so therefore there exists a non-equivariant $v_n$-self-map
$$ v: \Sigma^k X^{\Phi B} \to X^{\Phi B}. $$
This gives rise to a map of $A/B$-spectra
$$ v' := A/B_+ \wedge v_{\mr{triv}}: \Sigma^k A/B_+ \wedge (X^{\Phi B})_{\mr{triv}} \to A/B_+ \wedge (X^{\Phi B})_{\mr{triv}}.
$$
Under the equivalence of $\infty$-categories
$$ \Phi^B: \Sp^{A}[\mc{F}^{-1}_{B \nsubseteq}] \simeq \Sp^{A/B} $$
this gives a map of $A$-spectra
$$ v'' : \Sigma^k A/B_+ \wedge X \to A/B_+ \wedge X[\mc{F}^{-1}_{B \nsubseteq}]. $$
Let $\{\alpha_j\}^\ell_{j = 1}$ be the set of characters $\alpha$ such that $\alpha\big\vert_B \ne 1$.  By the second equivalence of Corollary~\ref{cor:formula}, it follows that for $i_j \gg 0$, the map $v''$ gives rise to a map
$$ a_1^{i_1} \cdots a_\ell^{i_\ell}v'': \Sigma^{k-i_1\alpha_1 - \cdots - i_\ell\alpha_\ell} A/B_+ \wedge X \to A/B_+ \wedge X.
$$
Take
$$ v''' : =  a_1^{i_1+1} \cdots a_\ell^{i_\ell+1}v''. $$
Note that by construction
$$ (v''')^{\Phi B} \simeq A/B_+ \wedge v $$
and is therefore a $v_n$-self map.
If $C \nsubseteq B$, then 
$$ (A/B_+)^{\Phi C} \simeq \ast $$
and therefore $(v''')^{\Phi C} \simeq 0$.  If $C \subsetneq B$, then there exists a character $\alpha_{j_0}$ which is trivial on $C$, but not on $B$.  Since $v'''$ is divisible by $a_{j_0}$, it follows that $(v''')^{\Phi C} \simeq 0$.
\end{proof}

\section{Chromatic localizations}\label{sec:loc}

\subsection*{Equivariant analogs of $\pmb{K(n)}$ and $\pmb{E(n)}$}

In \cite{Multicurves}, Strickland defined \emph{equivariant Morava $K$-theories}
$$ K(B,n) := A/B_+ \wedge K(n)_{\mr{triv}}[\mc{F}^{-1}_{B \nsubseteq}] \in \Sp^A_{(p)}. $$
We define \emph{equivariant Johnson-Wilson theories} similarly by
$$ E(B,n) := A/B_+ \wedge E(n)_{\mr{triv}}[\mc{F}^{-1}_{B \nsubseteq}] \in \Sp^A_{(p)}. $$
Since the spectra $K(n)$ and $E(n)$ are associative ring spectra, and the spectra
$$ A/B_+ \simeq D(A/B_+) $$
are commutative ring spectra, the spectra $K(B,n)$ and $E(B,n)$ are associative ring spectra.

\begin{prop}
For $C \le A$ we have
$$ K(B,n)^{\Phi C} \simeq \begin{cases}
\text{$A/B_+\wedge K(n)$},  & C = B, \\
\ast, & \text{otherwise}.
\end{cases} $$
and
$$ E(B,n)^{\Phi C} \simeq \begin{cases}
\text{$A/B_+\wedge E(n)$},  & C = B, \\
\ast, & \text{otherwise}.
\end{cases} $$
\end{prop}

\begin{proof}
This follows from the following observations:
\begin{align*}
(\Sigma^\infty A/B_+)^{\Phi C} & \simeq \begin{cases}
\Sigma^\infty A/B_+, & C \subseteq B, \\
\ast, & C \nsubseteq B,
\end{cases}\\
E[\mc{F}^{-1}_{B \nsubseteq}]^{\Phi C} & \simeq 
\begin{cases}
E^{\Phi C}, & B \subseteq C, \\
\ast, & B \nsubseteq C,
\end{cases}\\
E_{\mr{triv}}^{\Phi C} &\simeq E.
\end{align*}
\end{proof}

\begin{cor}\label{cor:null}
For $X \in \Sp^A$, we have 
$$ K(B,n) \wedge X \simeq \ast $$
if and only if
$$ K(n) \wedge X^{\Phi B} \simeq \ast, $$
and
$$ E(B,n) \wedge X \simeq \ast $$
if and only if
$$ E(n) \wedge X^{\Phi B} \simeq \ast, $$
\end{cor}

The following is an immediate consequence of \ref{eq:TBlocal}.

\begin{cor}\label{cor:monochromatic}
There are equivalences of $\infty$-categories:
\begin{align*} 
\Sp_{K(B,n)} & \xrightarrow{\simeq} \Sp^{B(A/B)}_{K(n)} \\
X & \mapsto (X^{\Phi B})_{K(n)} \\
\\
\Sp_{E(B,n)} & \xrightarrow{\simeq} \Sp^{B(A/B)}_{E(n)} \\
X & \mapsto (X^{\Phi B})_{E(n)}
\end{align*} 
\end{cor}


\subsection*{$\pmb{E(\ul{n})}$-localization}

Carrick introduced chromatic towers of smashing localizations for $\Sp^{C_{p^i}}_{(p)}$ \cite{Carrick}, and Balderrama did the same in the elementary abelian $p$-group case in \cite{BalderramaPower}.  Here we discuss a generalization to $\Sp^A_{(p)}$ for any finite abelian $p$-group $A$.  Undoubtedly, the interested reader could adapt these arguments to apply to any finite abelian group.

\begin{defn}
Suppose 
$$ \ul{n}: \Sub(A) \to \Nbar_- $$
is a height function.
Define a spectrum $E(\ul{n}) \in \Sp^A_{(p)}$ by\footnote{Here we take $E(B,-1) := \ast$.} 
$$ E(\ul{n}) = S[\mr{dom}_{< \infty}(\ul{n})^{-1}] \vee \bigvee_{B \in \mr{dom}_{<\infty}(\ul{n})} E(B,\ul{n}(B)).
$$
\end{defn}

\begin{rmk}
We will primarily be interested in the spectra $E(\ul{n})$ in the cases where the height function $\ul{n}$ is admissible.  Note that since we are assuming that $A$ is a $p$-group, the finite domain $\mr{dom}_{< \infty}(\ul{n}) \subseteq \Sub(A)$ of an admissible height function $\ul{n}$ is necessarily a family, and the infinite domaine $\mr{dom}_{=\infty}(\ul{n}) \subseteq \Sub(A)$ is open. 
\end{rmk}

Our main interest in $E(\ul{n})$ is the associated localization functor
\begin{align*}
\Sp^A_{(p)} & \to \Sp^A_{(p)} \\
X & \mapsto X_{E(\ul{n})}. 
\end{align*}
Note that since geometric fixed points is a monoidal functor, the Bousfield class of $E \in \Sp^A$ is determined by the collection of Bousfield classes $\bra{E^{\Phi B}}$ as $B$ ranges over the subgroups of $A$.  In particular, the Bousfield class $\bra{E(\ul{n})}$ is determined by
$$
\bra{E(\ul{n})^{\Phi B}} = 
\begin{cases}
\bra{E(\ul{n}(B))}, & 0 \le \ul{n}(B) < \infty, \\
\bra{S}, & \ul{n}(B) = \infty, \\
\bra{\ast}, & \ul{n}(B) = -1.
\end{cases}
$$

\begin{rmk}\label{rmk:E(n)}$\quad$
\begin{enumerate}
\item In the case of $A = C_p$, and $0 \le \ul{n} < \infty$, our $E(\ul{n})$ were introduced by Carrick in \cite{Carrick}.  Carrick also introduced \emph{different} spectra $E(\ul{n})$ which are Bousfield equivalent to our $E(\ul{n})$ for $A = C_{p^i}$, $i \ge 2$.

\item Balderrama studied Bousfield localizations of $\Sp^A$ with respect to $E^h_n$ (Morava $E_n$-theory regarded as a Borel $A$-spectrum with trivial action).  In this case we have an equality of Bousfield classes \cite[Thm~3.5]{sixauthor} (see also \cite{Torii})
$$ \bra{(E_n^h)^{\Phi B}} = \bra{E(n-\mr{rk}_p(B))} $$
(where $\bra{E(m)}$ is taken to be $\bra{\ast}$ if for $m < 0$).  Therefore we have an equality of Bousfield classes
$$ \bra{E^h_n} = \bra{E(h[n])} $$
where $h[n]$ is the height function given by
\begin{equation}\label{eq:h[n]}
h[n](B) = \mr{max}(n - \mr{rk}_p{B}, -1). 
\end{equation}
Note that the associated admissible closed subset $V_{h[n]+1} \subseteq \Spc_{(p)}(A)$ is the subset $V_{n+1}^A$ of Proposition~\ref{prop:height} corresponding to those $A$-equivariant formal group laws of height greater than or equal to $n+1$.

\item 
For $B \le A$, tom Dieck computes the geometric fixed points  $(\KU_A)^{\Phi B}$ \cite[Prop.~7.7.7]{tomDieckbook}.  Bonventre, Guillou, and Stapleton observed that this results in an equality of Bousfield classes \cite[Prop.~3.5,3.10]{BGS}
$$ \bra{(\KU_A)_{(p)}^{\Phi B}} =
\begin{cases}
\bra{E(1)},  & B = e, \\
\bra{H\QQ}, & B \: \text{cyclic}, \\
\ast, & \text{otherwise}.   
\end{cases}
$$
We therefore deduce an equality of Bousfield classes
$$ \bra{(\KU_A)_{(p)}} = \bra{E(h[1])}. $$

\item 
Let $c[n]$ be the constant height function at $n$.  Then the spectrum $E(c[n])$ is Bousfield equivalent to $E(n)_{\mr{triv}}$.

\item A map $f: X \to Y$ in $\Sp^A_{(p)}$ is an $E(\ul{n})$-equivalence if and only if 
$$ \text{$f^{\Phi B}$ is an equivalence, for $B$ with $\ul{n}(B) = \infty$} $$
and
$$ \text{$f^{\Phi B}$ is a $E(\ul{n}(B))$-equivalence, for $B$ with $0 \le \ul{n}(B) < \infty$}. $$

\item There is an equality of Bousfield classes
$$ \langle E(\ul{n}) \rangle = \left\langle \left( \bigvee_{B \in \mr{dom}_{=\infty}(\ul{n})} T(B) \right) \vee \left( \bigvee_{0 \le m \le \ul{n}(B) < \infty} K(B,m) \right) \right\rangle. $$ 
For $\ul{n}$ admissible, the reader should think of $X_{E(\ul{n})}$ as ``$X$ restricted to $U_{\ul{n}}$''.

\item 
If $\ul{n}' \ge \ul{n}$,  we have natural maps
$$ X_{E(\ul{n}')} \to X_{E(\ul{n})}. $$
\end{enumerate}
\end{rmk}

The \emph{chromatic tower} of $X \in \Sp^A_{(p)}$ is the system
$$ \{ X_{E(\ul{n})} \: : \: \text{$\ul{n}$ is an admissible height function} \}. $$

\begin{rmk}
The reader may prefer to regard the chromatic tower as the tower of localizations $X_{E(\ul{n})}$ where $\ul{n}$ ranges over the \emph{finite} admissible height functions, as these will eventually be shown to be built out of finitely many chromatic layers. Balderrama \cite{BalderramaPower} instead considers (by Remark~\ref{rmk:E(n)}(2)) the chromatic tower 
$$ \{ X_{E_n^h} \} = \{ X_{E(h[n])} \}. $$
Which of these notions is the more appropriate notion probably depends on the intended application.
\end{rmk}

\subsection*{The equivariant Smash Product Theorem}

Recall the Hopkins-Ravenel Smash Product Theorem \cite{Ravenel}.

\begin{thm}[Hopkins-Ravenel (\cite{Ravenel})]
Localization with respect to $E(n)$ is smashing.
\end{thm}

Smashing localizations enjoy many exceptional properties, but in the equivariant case Carrick shows that smashing localizations are especially advantageous, in that they can be understood on the level of geometric fixed points.

\begin{prop}[Carrick {\cite[Prop.~3.13, Cor.~3.15(1)]{Carrick}}]\label{prop:Carrick}
Suppose that localization with respect to $E \in \Sp^A$ is smashing.  Then $X \in \Sp^A$ is $E$-local if $X^{\Phi B}$ is $E^{\Phi B}$-local for all $B \le A$, and we have
$$ (X_E)^{\Phi B} \simeq (X^{\Phi B})_{E^{\Phi B}}. $$
\end{prop}

Carrick showed that in the case of $A = C_{p^i}$, for $\ul{n}$ admissible, $E({\ul{n}})$ is smashing \cite{Carrick}.  Balderrama showed that in the case where $A$ is an elementary abelian $p$-group, $E(h[n])$ is smashing \cite[Prop.~A.4.8]{BalderramaPower}. The following generalizes these results to all finite abelian $p$-groups and all $\ul{n}$.

\begin{thm}[Smash Product Theorem]\label{thm:smashing}
Suppose that $\ul{n}$ is an admissible height function on $\Sub(A)$.  Then $E(\ul{n})$ is smashing.
\end{thm}

Hovey-Sadofsky's Tate blue-shift theorem \cite[Thm.~1.1]{HoveySadofsky} plays a fundamental role in Carrick's proof of Theorem~\ref{thm:smashing} for $A = C_{p^i}$.  To prove the result for general $A$, we will need a variant, which we will derive from the results of \cite{sixauthor}.

Suppose that $X$ is a $G$-spectrum, and let $X^h$ denote its Borel completion.  We define
$$ X^{\tau G} := (X^h)^{\Phi G}. $$

\begin{thm}\label{thm:tauA}
Suppose that the underlying spectrum of $X \in \Sp^{A}_{(p)}$ is  $E(n)$-local.  Then $X^{\tau A}$ is $E(n-r)$-local, where $r = \mr{rk}_p(A)$.
\end{thm}

\begin{proof}
By \cite[Prop.~2.7(5)]{Carrick}, it suffices to show that $S_{E(n)}^{\tau A}$ is $E(n-r)$-local.  This result is proven from following the proof of \cite{HoveySadofsky}, but replacing their functor $P_G$ (which is their notation for the Tate spectrum) with $(-)^{\tau A}$, and replacing their use of \cite[Thm.~1.2]{HoveySadofsky} with \cite[Thm.~3.5]{sixauthor}.
\end{proof}

\begin{proof}[Proof of Theorem~\ref{thm:smashing}]
Choose a maximal chain of opens between $\mr{dom}_{= \infty}(\ul{n})$ and $\mr{dom}_{\ge 0}(\ul{n})$:
\begin{equation}\label{eq:chain}
\mr{dom}_{= \infty}(\ul{n}) = \mc{F}^c_0 \subsetneq \mc{F}_1^c \subsetneq \cdots \subsetneq \mc{F}^c_k = \mr{dom}_{\ge 0}(\ul{n}).
\end{equation}
Define height functions $\ul{n}_i$ by
$$
\ul{n}_i := \ul{n}\vert_{\mc{F}_i^c}
$$
(c.f. Equation~(\ref{eq:restriction})) so $\ul{n}_k = \ul{n}$.
We will prove the theorem by proving inductively on $i$ that $E(\ul{n}_i)$ is smashing.  Since
$$ X_{E(\ul{n}_0)} = X[\mc{F}_0^{-1}], $$
$E(\ul{n}_0)$ is smashing. Suppose that we know $E(\ul{n}_{i-1})$ is smashing.  Then, by Proposition~\ref{prop:Carrick}, we have
\begin{equation}\label{eq:PhiCi-1}
 (X_{E(\ul{n}_{i-1})})^{\Phi C} \simeq 
\begin{cases}
X^{\Phi C}, & \ul{n}(C) = \infty \\
(X^{\Phi C})_{E(\ul{n}(C))}, & C \in \mc{F}^c_{i-1}\cap \mr{dom}_{< \infty}(\ul{n}), \\
\ast, & C \not\in \mc{F}^c_{i-1}.
\end{cases}
\end{equation}
By maximality of the chain (\ref{eq:chain}), we have
$$ \mc{F}^c_{i} = \mc{F}^c_{i-1} \cup \{ B \}. $$
Let 
$$ n := \ul{n}(B). $$
Since $B \not\in \mc{F}^c_{i-1}$, Corollary~\ref{cor:null} implies that we have 
\begin{equation}\label{eq:PhiB_inull}
 E(B,n) \wedge X_{E(\ul{n}_{i-1})} \simeq \ast.
 \end{equation}
Since 
$$ E(\ul{n}_i) = E(\ul{n}_{i-1}) \vee E(B, n) $$
it follows from Proposition~\ref{prop:fracture} that the canonical square
\begin{equation}\label{eq:indsquare}
\xymatrix{
X_{E(\ul{n}_i)} \ar[r] \ar[d] & X_{E(B,n)} \ar[d] \\
X_{E(\ul{n}_{i-1})} \ar[r] & X_{E(B, n), E(\ul{n}_{i-1})}
}
\end{equation}
is a pullback.  Under the equivalence of Corollary~\ref{cor:monochromatic}, $E(B,n)$-localization is computed to be 
\begin{align*} 
\Sp^A_{(p)} & \to \Sp^A_{E(B, n)}  \simeq \Sp^{B(A/B)}_{E(n)} \\
X & \mapsto  X_{E(B,n)}  \mapsto (\Phi^{B} X )^h_{E(n)}.
\end{align*} 
and therefore we have
\begin{equation}\label{eq:EBnPhiC}
(X_{E(B, n)})^{\Phi C} \simeq
\begin{cases}
((\Phi^B X)^h_{E(n)})^{\tau C/B}, & B \subseteq C \\ 
\ast,  & B \not\subseteq C.
\end{cases}
\end{equation}
Suppose that $C \not\in \mc{F}^c_{i-1}$.
Then applying $(-)^{\Phi C}$ to (\ref{eq:indsquare}) and using (\ref{eq:PhiCi-1}), there is an equivalence
\begin{equation}\label{eq:Enigeomfp}
 (X_{E(\ul{n}_i)})^{\Phi C} \simeq (X_{E(B,n)})^{\Phi C}.
 \end{equation}
If $C = B$, then (\ref{eq:Enigeomfp}) gives
$$ (X_{E(\ul{n}_i)})^{\Phi B} \simeq (X^{\Phi B})_{E(n)}. $$
If $C \ne B$, then (\ref{eq:Enigeomfp}) implies that $C \not\in \mc{F}^c_{i}$.  Note that $B \in \mc{F}^c_i$, which implies (since $\mc{F}_i$ is a family) that $B \not\subseteq C$, and so  
$C \in \mc{F}_{B\not\subseteq}$.  It follows from (\ref{eq:EBnPhiC}) that we have
$$ (X_{E(\ul{n}_i)})^{\Phi C} \simeq \ast. $$
Suppose now that $C \in \mc{F}^c_{i-1}$.  We claim that the map
\begin{equation}\label{eq:claim}
(X_{E(B,n)})^{\Phi C} \xrightarrow{\simeq} (X_{E(B,n),E(\ul{n}_{i-1})})^{\Phi C}
\end{equation} 
obtained by applying 
 $(-)^{\Phi C}$ to the right vertical arrow of (\ref{eq:indsquare}) is an equivalence.
This would then imply that for $C \in \mc{F}^c_{i-1}$, there is an equivalence
$$ (X_{E(\ul{n}_i)})^{\Phi C} \simeq (X_{E(\ul{n}_{i-1})})^{\Phi C} \simeq (X^{\Phi C})_{E(\ul{n}(C))} $$
(where the last equivalence is by (\ref{eq:PhiCi-1})).
To verify this claim, we use (\ref{eq:PhiCi-1}) to identify the map (\ref{eq:claim}) with the map
\begin{equation}\label{eq:claim2}
(X_{E(B,n)})^{\Phi C} \to ((X_{E(B,n)})^{\Phi C})_{E(\ul{n}(C))}.
\end{equation}
In the case where $B \not \subseteq C$, (\ref{eq:EBnPhiC}) implies both the source and target of (\ref{eq:claim2}) are trivial, and therefore (\ref{eq:claim2}) is an equivalence.  If $B \subseteq C$, then (\ref{eq:EBnPhiC}) identifies the map (\ref{eq:claim2}) with the map
$$ ((\Phi^B X)^h_{E(n)})^{\tau C/B} \to (((\Phi^B X)^h_{E(n)})^{\tau C/B})_{E(\ul{n}(C))}. $$
To show this map is an equivalence, we simply need to know that $((\Phi^B X)^h_{E(n)})^{\tau C/B}$ is $E(\ul{n}(C))$-local.  By Theorem~\ref{thm:tauA}, the spectrum $((\Phi^B X)^h_{E(n)})^{\tau C/B}$ is $E(n-\mr{rk}_p(C/B))$-local.  Since $\ul{n}$ is admissible, we have
$$ n \le \ul{n}(C) + \mr{rk}_p(C/B) $$
and therefore
$$ n - \mr{rk}_p(C/B) \le \ul{n}(C). $$
It follows that $((\Phi^B X)^h_{E(n)})^{\tau C/B}$ is $E(\ul{n}(C))$-local.  This completes the argument that (\ref{eq:claim}) is an equivalence for $C \in \mc{F}^c_{i-1}$.

In summary, we have proven that for $C \le A$, we have
$$ (X_{E(\ul{n}_i)})^{\Phi C} \simeq
\begin{cases}
X^{\Phi C}, & \ul{n}(C) = \infty, \\
(X^{\Phi C})_{E(\ul{n}(C))}, & C \in \mc{F}^c_{i}\cap \mr{dom}_{< \infty}(\ul{n}), \\
\ast, & C \not\in \mc{F}^c_{i}.
\end{cases}
$$
It follows that $E(\ul{n}_i)$-localization commutes with arbitrary wedge products, and therefore $E(\ul{n}_i)$ is smashing \cite[3.3.2]{HPS}.
\end{proof}

The following is an immediate consequence of Proposition~\ref{prop:Carrick}.

\begin{cor}\label{cor:EPhi}
Suppose that $\ul{n}$ is an admissible height function on $\Sub(A)$.  Then we have
$$ (X_{E(\ul{n})})^{\Phi C} \simeq
\begin{cases}
X^{\Phi C}, & \ul{n}(C) = \infty, \\
(X^{\Phi C})_{E(\ul{n}(C))}, & 0 \le \ul{n}(C) < \infty, \\
\ast, & \ul{n}(C) = -1.
\end{cases}
$$
\end{cor}

\begin{rmk}
We note that Hill \cite{Hill} considered \emph{finite} localizations with respect to $E(\ul{n})$ for admissible height functions $\ul{n}$, and these localizations are smashing for formal reasons.   Specifically, recall from  \cite[Thm.~4.1]{BalmerFavi} (see also \cite{Miller}, \cite[Thm.~3.3.3]{HPS}) that associated to any thick tensor ideal of $\Sp^A_{(p),\omega}$, there is an associated smashing localization which kills the objects in the ideal.  Taking the ideal to be the finite $E(\ul{n})$-acyclics, there is a smashing localization 
\begin{align*}
 \Sp^A_{(p)} & \to \Sp^A_{(p)} \\
 X &\mapsto X^f_{E(\ul{n})}.
\end{align*}
\begin{enumerate}
\item It follows from \cite[Defn~(3)]{Miller} that if $X \in \Sp^{A}_{(p),\omega}$ is type $\ul{n}$, and $v$ is $v_{\ul{n}}$-self map, we have
$$ X^f_{E(\ul{n})} =  v^{-1} X. $$
\item As a consequence of the resolution of the non-equivariant telescope conjecture \cite{Mahowald}, \cite{MillerASS}, \cite{BHLS}, the only finite $\ul{n}$ for which the natural transformation 
$$ X^f_{E(\ul{n})} \to X_{E(\ul{n})} $$
is an equivalence for all $X$ are the $\ul{n}$ which satisfy $\ul{n} \le 1$.
\end{enumerate}  
\end{rmk}

\subsection*{Chromatic fracture}

We can now prove a general chromatic fracture theorem.

\begin{thm}\label{thm:fracture}
Suppose that $\ul{n}_1 \le \ul{n}_2$ are a pair of height functions on $\Sub(A)$ such that $\ul{n}_1$ is admissible and 
$$ \mr{dom}_{= \infty}(\ul{n}_1) = \mr{dom}_{=\infty}(\ul{n}_2). $$
Let
$$ T = \bigvee_{\{(B,m) \: : \: \ul{n}_1(B) < m \le \ul{n}_2(B) < \infty\}} K(B,m). $$ 
Then for each $X \in \Sp^A_{(p)}$, the canonical square
$$
 \xymatrix{
 X_{E(\ul{n}_2)} \ar[r] \ar[d] & X_{T} \ar[d] \\
 X_{E(\ul{n}_{1})} \ar[r] & X_{T, E(\ul{n}_{1})}
 }
$$
is a pullback.
\end{thm}

\begin{proof}
This is a consequence of Proposition~\ref{prop:fracture}, Corollary~\ref{cor:EPhi}, and Corollary~\ref{cor:null}.
\end{proof}

In particular, we have the following corollary, which illustrates how $E(\ul{n})$-localization is inductively built out of monochromatic layers.

\begin{cor}\label{cor:chromfrac}
Suppose that $\ul{n}$ is an admissible height function on $\mr{Sub}(A)$, $B \le A$ is a subgroup  with $\ul{n}(B) < \infty$, and $X \in \Sp^A_{(p)}$.  Then the canonical square
$$
\xymatrix{
X_{E(\ul{n}+\chi_B)} \ar[r] \ar[d] & X_{K(B,\ul{n}(B)+1)} \ar[d] \\
 X_{E(\ul{n})} \ar[r] & X_{K(B,\ul{n}(B)+1),E(\ul{n})} 
}
$$
is a pullback.
\end{cor}

\begin{rmk}\label{rmk:monochromatic}
The monochromatic layers of Corollary~\ref{cor:chromfrac} may be understood as follows.
\begin{enumerate}
\item The monochromatic layers in Corollary~\ref{cor:chromfrac} may be accessed in a straightforward way by means of Corollary~\ref{cor:monochromatic}.  For example, using \cite{DevinatzHopkins}, we have
\begin{equation}\label{eq:monochromatic}
 \pi^A_{*+\gamma} (S_{K(B,n)}) \cong \pi_* S^{B(A/B)^{\gamma^B}}_{K(n)} \cong \pi_*(E^{B(A/B)^{\gamma^B}}_n)^{h\mc{G}_n}
 \end{equation}
where
$\gamma$ is a real representation of $A$, $\gamma^B$ its $B$-fixed points, regarded as a representation of $A/B$, $B(A/B)^{\gamma^B}$ is the associated Thom spectrum, $E_n$ is the $n$th Morava $E$-theory spectrum, and $\mc{G}_n$ is the extended Morava stabilizer group.  Since the homotopy fixed points in (\ref{eq:monochromatic}) are continuous (in the sense of \cite{BehrensDavis}), the last isomorphism requires the $K(n)$-local dualizablity of $B(A/B)^{\gamma^B}$ (see \cite[Cor.~8.7]{666}).

\item There is therefore a homotopy fixed point spectral sequence
\begin{equation}\label{eq:HFPSS}
 E_2^{s,t}(\gamma) = H^s_c(\mc{G}_n, E^t_n(B(A/B)^{\gamma^B})) \Rightarrow \pi^A_{t-s+\gamma}(S_{K(B,n)}) 
 \end{equation}
for each $\gamma \in RO(A)$.

\item In the case where $\gamma = 0$, it follows from \ref{eq:MFGAloc} and \ref{eq:MFGAcompl} that the $E_2$-term of (\ref{eq:HFPSS}) is intimately connected to the moduli stack $\MFG{A}$.  Indeed, letting $\GG_1$ denote the formal group of $E_n$, 
for each homomorphism
$$ \varphi': (A/B)^\vee \to \GG_1 $$
there is an associated equivariant formal group
$$ \varphi: A^\vee \to A^\vee \times_{(A/B)^\vee} \GG_1. $$
This gives rise to a canonical equivariant formal group over 
$$ \Spf(E_n^0(B(A/B))) \cong \Hom((A/B)^\vee, \GG_1) $$
(compare with \cite[Ch.~11]{Multicurves}, where Strickland shows that this equivariant formal group is a universal deformation).
The action of $\mc{G}_n$ on $\GG_1 = \Spf(E_n^0(\CC P^\infty))$ induces the action of $\mc{G}_n$ on $E_n^0(B(A/B)^\vee)$ through its action on the second variable of $\Hom((A/B)^\vee, \GG_1)$.
\end{enumerate}
\end{rmk}

\begin{rmk}
The homotopy groups $\pi^{A}_\star S_{K(e,1)}$ were computed by Balderrama \cite{Balderrama} in the case of $A = C_2$, and by Carawan, Field, Guillou, Mehrle and Stapleton in the integer graded case where $A$ is a finite $p$-group with $p$ odd \cite{fiveauthor}.
\end{rmk}

\subsection*{Chromatic convergence}

The classical Chromatic Convergence Theorem of Hopkins-Ravenel \cite{Ravenel} states that if $X \in \Sp_{(p),\omega}$ is a finite $p$-local complex, then 
$$ X \simeq \varprojlim_n X_{E(n)}. $$
This theorem was generalized to all connective $p$-local spectra $X$ of finite projective $\BP$-dimension by Barthel \cite{Barthel}.

Equivariant chromatic convergence is surprisingly subtle because geometric fixed points doesn't commute with inverse limits.  
Nevertheless, in the case where $A$ is an elementary abelian $p$-group and we consider the tower 
$$ \{ X_{E(h[n])} \}_n, $$
Balderrama \cite{BalderramaPower} was able to combine isotropy separation techniques with his proof that $E(h[n])$ is smashing to reduce the question to the convergence of the chromatic tower for $\Sigma^\infty_+ BA$.  Johnson-Wilson \cite{JohnsonWilson} showed that $\Sigma^\infty_+ BA$ satisfies Barthel's criterion described above.

Balderrama explained to the authors that since we proved the Smash Product Theorem for $A$ an abelian $p$-group, his techniques imply that the chromatic convergence theorem is reduced to the convergence of the chromatic tower for $\Sigma^\infty_+ BA$.  J.~Hahn communicated to Balderrama a proof that the chromatic tower indeed does converge for $BA$ in the finite abelian $p$-group case, which Balderrama wrote up \cite{BalderramaHahn}.  Carrick independently sent the authors a different approach using tom Dieck splitting techniques to reduce the problem to the chromatic convergence of $\Sigma^{\infty}_+ BA$.  Thanks to these astute observations, we can deduce the following equivariant chromatic convergence theorem.

\begin{thm}[Balderrama-Carrick-Hahn]
Suppose that $(\ul{n}_i)_i$ is a non-decreasing sequence of admissible height functions on $\Sub(A)$ such that for each $B$ we have
$$ \lim_{i \to \infty} \ul{n}(B) = \infty. $$
Then for  $X \in \Sp^A_{(p),\omega}$, the canonical map gives an equivalence
$$ X \xrightarrow{\simeq} \varprojlim_i X_{E(\ul{n}_i)}. $$
\end{thm}

 \begin{proof}
It suffices to prove this for $X \simeq S$, the equivariant sphere. By \cite[Cor.~A.2.7]{BalderramaPower}, it suffices to show that for all $B \le C \le A$, the map
\begin{equation}\label{eq:convmap}
 \Sigma^\infty_+ B(C/B) \to \varprojlim_i ((S_{E(\ul{n}_i)})^{\Phi B})_{hC/B} 
 \end{equation}
is an equivalence.  By Corollary~\ref{cor:EPhi}, we have
\begin{align*}
((S_{E(\ul{n}_i)})^{\Phi B})_{hC/B} 
& \simeq (S_{E(\ul{n}_i(B))})_{hC/B} \\
& \simeq S_{E(\ul{n}_i(B))}\wedge {B(C/B)}_+ \\
& \simeq (\Sigma^\infty_+ {B(C/B)})_{E(\ul{n}_i(B))}. 
\end{align*}
It follows that (\ref{eq:convmap}) is an equivalence if chromatic convergence holds for $\Sigma^\infty_+ B(C/B)$, and this was proven by Hahn (as recorded in \cite{BalderramaHahn}).
 \end{proof}

\section{The case of $A = C_2$}\label{sec:C2}

We now give explicit examples of the theoretical framework of this paper in the case of $A = C_2$.  For the purposes of this section:
\begin{align*}
\sigma & \in RO(C_2), \quad \text{the sign representation with $|\sigma| = 1$} \\
a & := a_\sigma \in \pi^{C_2}_{-\sigma}(S)  \\
\eeta & \in \pi_\sigma^{C_2}(S), \quad \text{the equivariant Hopf map with $\eeta^{\Phi C_2} = 2$} \\
h & := 2 - a\eeta \in \pi^{C_2}_0(S) \quad \text{so $2 = h + a\eeta$} \\
u & :=  u_{2\sigma} \in \pi^{C_2}_{2-2\sigma}(\MU_{C_2}) \\ 
e  & :=  e_{\sigma} = u^{-1}a^2 \in \pi^{C_2}_{-2}(\MU_{C_2}) \quad \text{(regarding $\sigma$ as a complex representation)}\\
b_i & := b_i^{\sigma} \in \pi^{C_2}_{2i-2} (\MU_{C_2}) \\
x +_\MU y & = \text{the universal formal group law over $\MU_*$} \\
x +_\BP y & = \text{the universal $p$-typical formal group law over $\BP_*$} \\
[2]_{\MU}(x) & = \text{the 2-series of $\MU$} \\
[2]_{\BP}(x) & = \text{the 2-series of $\BP$}
\end{align*}

We will implicitly always be working in the 2-local setting, and we will simply denote the 2-local sphere as $S$.
Height functions
$$ \ul{n} : \Sub(C_2) \to \Nbar_- $$
will be represented as $\ul{n} = (\ul{n}(e), \ul{n}(C_2))$, so a $v_{(n, m)}$-self map is a $v_{\ul{n}}$-self map with $\ul{n}(e) = n$ and $\ul{n}(C_2) = m$.  The height function $\ul{n}$ is admissible if and only if $n \le m+1$.

\subsection*{Homotopy of $\pmb{\MU_{C_2}}$}

Theorem~\ref{thm:AbramKriz} specializes in this case to give a pullback
\begin{equation}\label{eq:MUC2}
\xymatrix{
\pi_*^{C_2}\MU_{C_2} \ar[d] \ar[r] & \frac{\MU_*[[e]]}{([2]_\MU(e))}  \ar[d] \\
\MU_* [e^{\pm}, b_i]_{i \ge 1} \ar[r] &   \frac{\MU_*[[e]]}{([2]_\MU(e))}[e^{-1}]
}
\end{equation}
In the right-hand terms of the pullback (\ref{eq:MUC2}), we have
$$ z_\sigma = e +_{\MU} z = e + \sum_{i \ge 1} b_i z^i. $$
where $z_\sigma$ is the local coordinate (\ref{eq:localcoord}).
Strickland \cite{Strickland} used this to obtain the following explicit presentation of $\pi_*^{C_2}\MU_{C_2}$ (see also \cite{Sinha}).  Write the formal group law for $\MU$ as
$$ z +_{\MU} w = \sum_{i,j}a_{i,j} z^i w^j $$
with 2-series
$$ [2]_\MU(z) = \sum_i p_i z^i. $$
In particular $p_1 = 2$ and $p_{2^n}$ is a $v_n$-generator.
Define elements $q_i$ ($i\ge 1$) and $b_{i,j}$ ($i \ge 1, j\ge 0$) in the right-hand terms of the pullback (\ref{eq:MUC2}) by
\begin{align}
 q_i & := p_i + p_{i+1}e + p_{i+2}e^2 + \cdots, \label{eq:qie} \\ 
 b_{i,j} & := a_{i,j}+a_{i,j+1}e + a_{i,j+2}e^2 + \cdots 
\end{align}
so we have
\begin{align*}
 eq_1 & = [2]_\MU(e) = 0, \\ 
 b_{i,0} & = b_i.
\end{align*}
Then we have relations
\begin{align}
q_i & = p_i + e q_{i+1}, \quad & i \ge 1,  \label{eq:rel1} \\
b_{i,j} & = a_{i,j} + e b_{i,j+1}, \quad & i \ge 1, j \ge 0 \label{eq:rel2} 
\end{align}
which allow us to inductively define elements of the lower left hand term of (\ref{eq:MUC2}) by
\begin{align*}
q_1 & = 0, \\
q_{i+1} & = e^{-1}(q_i-p_i), \\
b_{i,0} & = b_i, \\
b_{i,j+1} & = e^{-1}(b_{i,j}-a_{i,j}).
\end{align*}
Thus since (\ref{eq:MUC2}) is a pullback, there are elements
\begin{align*}
 q_{i} & \in \pi^{C_2}_{2i-2}(\MU_{C_2}), \\
 b_{i,j} & \in \pi^{C_2}_{2i + 2j -2}(\MU_{C_2}).
\end{align*}

\begin{thm}[Strickland \cite{Strickland}]
There is an isomorphism
$$
\pi^{C_2}_*\MU_{C_2} \cong \frac{\MU_*[e, q_i, b_{i,j}]_{i\ge 1,j \ge 0}}{(eq_1 = 0, (\ref{eq:rel1}), (\ref{eq:rel2}))}.
$$
\end{thm}

\begin{rmk}
Relations (\ref{eq:rel1}) and (\ref{eq:rel2}) imply that we don't need all of the $q_i$'s and $b_{i,j}$'s to generate $\pi_*^{C_2}\MU_{C_2}$; we only need $q_{i_k}$ and $b_{i,j_k}$ for a cofinal collection of $i_k$'s and $j_k$'s.
\end{rmk}

\begin{cor}
There is an isomorphism
$$
\pi^{C_2}_*\BP_{C_2} \cong \frac{\BP_*[e, q_i, b_{i,j}]_{i\ge 1,j \ge 0}}{(eq_1 = 0, (\ref{eq:rel1}), (\ref{eq:rel2}))}
$$
where the $p_i$'s and $a_{i,j}$'s in (\ref{eq:rel1}) and (\ref{eq:rel2}) are taken to be coefficients of $[2]_\BP(x)$ and $x+_\BP y$, respectively.
\end{cor}

We extend this computation $RO(C_2)$-degrees.

\begin{prop}\label{prop:ROC2MU}
There is an isomorphism 
$$ \pi^{C_2}_\star \MU_{C_2} \cong \frac{\MU_*[a, u^{\pm}, q_i, b_{i,j}]_{i \ge 1, j \ge 0}}{(aq_1 = 0, eu = a^2, (\ref{eq:rel1}), (\ref{eq:rel2}))}
$$
and an analogous isomorphism for $\pi_\star^{C_2}\BP_{C_2}$.
\end{prop}

\begin{proof}
Because $\MU_{C_2}$ is complex oriented, it has an invertible orientation class $u = u_{2\sigma} \in \pi^{C_2}_{2-2\sigma}\MU_{C_2}$.  We therefore only have to determine $\pi^{C_2}_{*-\sigma}$.
The cofiber sequence 
$$ S^{-\sigma} \xrightarrow{\cdot a} S \to \Sigma^{\infty}_+ C_2 $$
induces an
exact sequence 
$$ 
\pi^{C_2}_{2k+1}\MU_{C_2} \xrightarrow{\cdot a} \pi^{C_2}_{2k+1-\sigma}\MU_{C_2} \to \MU_{2k} \xrightarrow{\mr{Tr}_e^{C_2}} \pi^{C_2}_{2k} \MU_{C_2} \xrightarrow{\cdot a} \pi^{C_2}_{2k-\sigma}\MU_{C_2} \to \pi_{2k-1} \MU. $$
Since $\mr{Tr}_e^{C_2}(1) = q_1$ we may compute for $x \in \MU_*$
$$ \mr{Tr}_e^{C_2}(x) = x\mr{Tr}_e^{C_2}(1) = xq_1. $$
Since this map is injective, and $\MU_*$, $\pi^{C_2}_*\MU_{C_2}$ are concentrated in even degrees, we deduce that $\pi^{C_2}_{2k+1+\sigma}\MU_{C_2} = 0$ and there is a short exact sequence
$$ 0 \to \MU_{2k} \xrightarrow{\cdot q_1} \pi^{C_2}_{2k}\MU_{C_2} \xrightarrow{\cdot a} \pi^{C_2}_{2k-\sigma} \MU_{C_2} \to 0. $$
The result follows.
\end{proof}

\subsection*{$\pmb{v_{(n,m)}}$-generators}

Clearly the elements $v_n \in \pi^{C_2}_* \BP_{C_2}$ are $v_{(n,n)}$-generators, and the Euler class $e \in \pi^{C_2}_*\BP_{C_2}$ is a $v_{(-1,\infty)}$-generator.

\begin{prop}\label{prop:vnm}
The elements $q_{2^n} \in \pi^{C_2}_*\BP_{C_2}$ are $v_{(n,n-1)}$-generators.
\end{prop}

\begin{proof}
Letting $v_i \in \BP_*$ be the Araki generators, the $2$-series is given by
$$ [2]_\BP(x) = 2x +_\BP v_1x^{2} +_\BP v_2 x^4 +_\BP \cdots. $$
It follows that we have
\begin{equation}\label{eq:2seriescong} 
[2]_\BP(x) \equiv v_{n-1} x^{2^{n-1}} + v_{n}x^{2^{n}} + \cdots \mod (2, \ldots, v_{n-2}). 
\end{equation}
We need to show that
\begin{equation}\label{eq:phieq}
 q^{\Phi e}_{2^n} \in \pi_*\BP_A^{\Phi e} = \BP_* \: \text{is a $v_n$-generator}
 \end{equation}
and
\begin{equation}\label{eq:phiC2q} 
q^{\Phi C_2}_{2^n} \in \pi_*\BP_A^{\Phi C_2} = \BP_*[e^\pm, b_i]_{i \ge 1} \: \text{is a $v_{n-1}$-generator}.
\end{equation}
For (\ref{eq:phieq}) we note that $(-)^{\Phi e}$ is the same thing as restriction, and is given by the composite
$$ \pi^{C_2}_{*}\BP_A \to \pi^{C_2}_{*}\BP^h_A = \BP^{-*}(BA) \to \BP_*. $$
The last map in the composite above is the map
$$ \frac{\BP_*[[e]]}{([2]_\BP(e))} \to \BP_* $$
of $\BP_*$-modules given by sending $e$ to $0$. We therefore deduce (\ref{eq:phieq}) from (\ref{eq:qie}) and (\ref{eq:2seriescong}).

Turning our attention to (\ref{eq:phiC2q}), by iterated application of (\ref{eq:rel1}) we have
$$ 0 = 2e + p_2e^2 + p_3 e^3 + \cdots + p_{2^n-1} e^{2^n-1} + q_{2^n}e^{2^n}. $$
In $\pi_* \BP_A^{\Phi C_2} = \pi_*^{C_2}\BP_{C_2}[e^{-1}]$, by (\ref{eq:2seriescong}), the equation above gives
$$ 0 \equiv v_{n-1} e^{2^{n-1}} + q_{2^n} e^{2^n} \mod (2, \ldots, v_{n-2}) $$
and (\ref{eq:phiC2q}) follows.
\end{proof}

\subsection*{Geometric fixed points and Mahowald invariants}

\begin{lem}\label{lem:Phimap}
Suppose $X, Y \in \Sp^A$.
\begin{enumerate}
\item The geometric fixed points $(-)^{\Phi e}$ of maps (aka underlying nonequivariant map) is given by
$$
\xymatrix{ [X,Y]^{C_2} \ar[r] \ar[dr]_{(-)^{\Phi e}} &
[X, Y/a]^{C_2} \ar@{=}[d] \\
& [X^{\Phi e},Y^{\Phi e}]
}
$$

\item Suppose that $X$ is finite.  Then geometric fixed points $(-)^{\Phi C_2}$ of maps is given by
$$
\xymatrix{ [X,Y]^{C_2} \ar[r] \ar[dr]_{(-)^{\Phi C_2}} &
[X, Y]^{C_2}[a^{-1}] \ar@{=}[d] \\
& [X^{\Phi C_2},Y^{\Phi C_2}]
}
$$

\end{enumerate}
\end{lem}

\begin{proof}
(1) follows immediately from the fact that $(-)^{\Phi e}$ is restriction, with right adjoint (co)induction $\Ind^{C_2}_e$. The shearing isomorphism gives an equivalence
$$ \Ind_e^{C_2} Y^{e} \simeq Ca \wedge Y. $$
(2) follows from the equivalences 
$$ \Sp^{C_2}[a^{-1}] \simeq \Sp^{C_2}[\{e\}^{-1}] \xrightarrow[\simeq]{(-)^{\Phi C_2}} \Sp
$$
induced from Proposition~\ref{prop:eulerloc} and \ref{eq:geomloc}.
\end{proof}

We recall how Lemma~\ref{lem:Phimap} relates Mahowald invariants to $a$-divisibility \cite{BrunerGreenlees},\cite{GIMI}. Given 
$$ x \in \pi_i S = \pi^{C_2}_i S[a^{-1}], $$ 
we search for the maximum $a$-divisibility of $x$: we want $k$ maximal such that there exists $\td{x}$ which maps to $x$ in the localization
$$ \td{x} \in \pi^{C_2}_{i+k\sigma}S \xrightarrow{a} \pi^{C_2}_{i+(k-1)\sigma}S \xrightarrow{a} \pi^{C_2}_{i+(k-2)\sigma}S \xrightarrow{a} \cdots \to \pi^{C_2}_i S[a^{-1}] \ni x. $$
Since $k$ is maximal, $\td{x}$ is not $a$-divisible, and therefore the image $y$ under the composite
$$ \td{x} \in \pi^{C_2}_{i+k\sigma}S \to \pi^{C_2}_{i+k\sigma}Ca \cong \pi_{i+k}S \ni y $$
is non-trivial.  We say that \emph{$y$ is a Mahowald invariant of $x$} and write
$$ y \in M(x). $$

This discussion shows that Mahowald invariants have an obvious generalization to arbitrary elements
$$ f \in [X^{\Phi C_2},Y^{\Phi C_2}] \cong [X,Y[a^{-1}]]^{C_2} $$
with $X$ finite.  
Let $k$ be maximal such that there exists $\td{f}$ which maps to $f$ in the localization
$$ \td{f} \in [\Sigma^{k\sigma} X,Y]^{C_2} \xrightarrow{a}  [\Sigma^{(k-1)\sigma} X,Y]^{C_2} \xrightarrow{a} [\Sigma^{(k-2)\sigma} X,Y]^{C_2} \xrightarrow{a} \cdots \to [X,Y[a^{-1}]]^{C_2} \ni f. $$
Since $k$ is maximal, $\td{f}$ is not $a$-divisible, and therefore the image $g$ of $\td{f}$ under the composite
$$ \td{f} \in [\Sigma^{k\sigma}X,Y]^{C_2} \to [\Sigma^{k\sigma}X,Y \wedge Ca]^{C_2} \cong [\Sigma^k X^e,Y^e] \ni g $$
is non-trivial.

\begin{defn}
We say that $g$ as above is a \emph{Mahowald invariant of $f$ (with respect to $DX \wedge Y$)} and write
$$ g \in M(f). $$
We will call a lift $\td{f}$ as above a \emph{Mahowald lift} of $f$.
\end{defn}

\begin{rmk}
Since geometric fixed points commutes with the Spanier-Whitehead duals for finite equivariant spectra \cite[III.1.9]{LMS}, 
the concept of Mahowald invariants and Mahowald lifts of maps $ [X^{\Phi C_2},Y^{\Phi C_2}]$ and maps $[S,(DX \wedge Y)^{\Phi C_2}]$ coincide.  This is our reason for the terminology ``with respect to $DX \wedge Y$.''  If $\td{f}$ is a Mahowald lift of $f$, then we have
$$ \td{f}^{\Phi e} \in M(\td{f}^{\Phi C_2}). $$
The collection of all Mahowald invariants of $f$ forms a coset
$$ M(f) \subseteq [\Sigma^k X^e, Y^e] $$
and the collection of Mahowald lifts of $f$ form a coset of
$$ [\Sigma^{k\sigma} X,Y]^{C_2}. $$
The indeterminacy of the Mahowald invariant is a subgroup of $[\Sigma^k X^e,Y^e]$. The indeterminacy of the Mahowald lift is a subgroup of $[\Sigma^{k\sigma}X,Y]^{C_2}$, and the restriction map
$$ [\Sigma^{k\sigma}X,Y]^{C_2} \to [\Sigma^k X^e,Y^e] $$
and the image of the indeterminacy of the Mahowald lift under the restriction is the indeterminacy of the Mahowald invariant. 
\end{rmk}

Proposition~\ref{prop:vnm} has the following reinterpretation in terms of Mahowald invariants.

\begin{prop}\label{prop:Mvn}
Every Mahowald invariant of 
$$ v_{n-1} \in \pi_{2(2^{n-1}-1)}(\BP^{\Phi C_2}_{C_2}) $$
(with respect to $\BP_{C_2}$) is a $v_n$-generator in
$$ \pi_{2(2^n-1)}(\BP) = \pi_{2(2^n-1)}(\BP^e_{C_2}). $$
The element 
$$ u^{-2^{n-1}} q_{2^n} \in \pi^{C_2}_{2^n-2+2^n\sigma}\BP_{C_2} $$  
is a Mahowald lift of $v_{n-1}$, and for every Mahowald lift $\td{v}_{n-1}$ of $v_{n-1}$, the element 
$$ u^{2^{n-1}} \td{v}_{n-1} \in \pi_{2^n-2}\BP_{C_2} $$
is a $v_{(n,n-1)}$-generator.  More generally, every Mahowald invariant of $v^i_{n-1}$ satisfies 
$$ M(v^i_{n-1}) \dotequiv v_n^i \mod (2,\ldots, v_{n-1}). $$
\end{prop}

\begin{proof}
The existence of a $v_{(n,n-1)}$-generator implies that there is a $v_n$-generator $v \in M(v_{n-1})$.
Proposition~\ref{prop:ROC2MU} implies that the kernel of $a$-multiplication is the ideal generated by $q_1$.  Therefore the indeterminacy of the Mahowald lift is the ideal $(q_1)$, and therefore the indeterminacy of the Mahowald invariant is the ideal $(2)$. This implies that if $w$ is another Mahowald invariant of $v_{n-1}$, then
$$ v \equiv w \mod 2. $$
Thus $w$ is also a $v_n$-generator, and the corresponding Mahowald lift is also a $v_{(n,n-1)}$-generator (after multiplication by $u^{2^{n-1}}$).  The last part follows by taking $w^i$ to be a Mahowald lift of $v^i_{n-1}$, and using the same indeterminacy argument.
\end{proof}

\subsection*{A strategy to inductively produce $\pmb{v_{(n,n-1)}}$-self maps}

Classically, the canonical procedure to construct non-equivariant $v_n$-self maps is to start with $S/p$ (a type $1$ complex), and to find $i_1$ minimal such that $S/p$ has a $v_1^{i_1}$-self map 
$$ v_1^{i_1}: \Sigma^{2i_1(p-1)} S/p \to S/p. $$
Then the cofiber $S/(p,v_1^{i_1})$ of this map is a type $2$ complex.  Find $i_2$ minimal so that $S/(p,v_1^{i_1})$ has a $v_2^{i_2}$-self map, and repeat.

We propose a modification to this algorithm to produce  $C_2$-equivariant complexes of type $(n,n-1)$, and $v_{(n,n-1)}$-self maps on these.  

\begin{algorithm}\label{alg:vnm}
Suppose that $X \in \Sp^{C_2}_{(2),\omega}$ is a type $(n,n-1)$ complex.
\begin{enumerate}[label=(Step~\arabic*)]
\item Since $X^{\Phi C_2}$ is a type $n-1$, it admits a $v_{n-1}$-self map
$$ v : \Sigma^i X^{\Phi C_2} \to X^{\Phi C_2}. $$

\item Compute the Mahowald invariants of iterates $v^j$ (with respect to $X$) until you have 
$$ w: \Sigma^{ij+k} X^e \to X^e \in M(v^j) $$
with $w$ a $v_n$-self map.
\label{item:hypstep}

\item The corresponding Mahowald lift
$$ \td{v} : \Sigma^{ij + k\sigma} X \to X $$
is then a $v_{(n,n-1)}$-self map.

\item The cofiber $X/\td{v}$ has type $(n+1,n)$.
\end{enumerate}
\end{algorithm}

\begin{conjecture}
A power $j$ satisfying the criteria of \ref{item:hypstep} above always exists.
\end{conjecture}

\begin{rmk}\label{rmk:v-1n}
Theorem~\ref{thm:vnB} implies that $v_{(-1,n)}$ and $v_{(n,-1)}$-self maps always exist on suitable complexes, but we can make this totally explicit in the case of $A = C_2$.
\begin{enumerate}
\item Suppose that $X \in \Sp_{(2),\omega}$ is type $n$, and 
$$ v: \Sigma^k X \to X $$
is a $v_n$-self map.  Then $X_{\mr{triv}}/a \in \Sp^{C_2}_{(2),\omega}$ is type $(n,\infty)$, and  
$$ v \wedge 1: X_{\mr{triv}}/a \to X_{\mr{triv}}/a $$
is a $v_{(n,-1)}$-self map.
In particular, if $S/(2^{i_0}, v_1^{i_1}, \ldots, v_n^{i_n})$ exists non-equivariantly, then 
$$ S/(a,v_{(0,-1)}^{i_0}, v_{(1,-1)}^{i_1}, \ldots, v_{(n,-1)}^{i_n}) $$
exists $C_2$-equivariantly. 

\item If $X \in \Sp^{C_2}_{(2),\omega}$ is type $(m,n)$, then $X^{\Phi C_2}$ is type $n$.  Suppose that 
$$ v : \Sigma^i X^{\Phi C_2} \to X^{\Phi C_2} $$
is a non-equivariant $v_n$-self map.  Suppose that
$$ w \in [\Sigma^{i+k\sigma} X, X]^{C_2} $$
is a lift of $v$ in the diagram
$$ [\Sigma^{i+k\sigma}X, X]^{C_2} \xrightarrow{a} [\Sigma^{i+(k-1)\sigma}X, X]^{C_2}  \xrightarrow{a} \cdots \to [\Sigma^i X, X[a^{-1}]]^{C_2} = [\Sigma^i X^{\Phi C_2}, X^{\Phi C_2}].  $$
Then either
\vspace{1em}

\begin{enumerate}[label=(\alph*)]
\item $w^{\Phi e}$ is nilpotent and $w$ is a $v_{(-1,n)}$-self map, or
\item $w^{\Phi e}$ is non-nilpotent, and $w$ is potentially a $v_{(m,n)}$-self map,\footnote{The only reason $w$ wouldn't be a $v_{(m,n)}$-self map is if $K(i) \wedge w^{\Phi C_2}$ was non-nilpotent for some $i > m$} and $a\cdot w$ is a $v_{(-1,n)}$-self map.
\end{enumerate}

\item 
The situation (2)(a) above is most interesting the case where $w$ is a Mahowald lift --- in this case 
$$ x = w^{\Phi e} \in [\Sigma^{i+k}X^e, X^e] $$ 
is a non-trivial nilpotent self-map with a \emph{non-nilpotent equivariant lift $w$}.  In this case $x$ is the Mahowald invariant of $v$ with respect to $DX \wedge X$.  Thus the nilpotent non-equivariant element $x$ is connected to $v_{(-1,n)}$-periodicity.  We shall see this in our examples later in this section in the case where $x = \eta$ and $n = 0$, and $x = \bar{\kappa}$ and $n = 1$.  In these cases these periodicities has been studied by Quigley \cite{Quigley}, generalizing work of Andrews in the $\CC$-motivic setting \cite{Andrews}. Quigley, adopting Andrews' terminology to the $C_2$-equivariant context, refers to such periodicities as \emph{$w_n$-periodicity}.

\end{enumerate}
\end{rmk}

\subsection*{$\pmb{v_{(0,-1)}}$-periodicity and $\pmb{h}$}

The (2-local) sphere $S \in \Sp^{C_2}_{(2),\omega}$ is a type $(0,0)$ spectrum.  The following is observed in \cite{BGL}.

\begin{prop}\label{prop:h}
The map
$$ h \in  [S,S]^{C_2} $$
is a $v_{(0,-1)}$-self map.
\end{prop}

\begin{proof}
Since $2 = h + a\eeta$, it follows from Lemma~\ref{lem:Phimap}(1) that
$$ h^{\Phi e} = 2 $$
Since $h$ is $a$-torsion, it follows from Lemma~\ref{lem:Phimap}(2)
$$ h^{\Phi C_2} = 0. $$
\end{proof}

It follows from the fact that $ah = 0$ that we have
$$ \Sp_{(2)}^{C_2}[h^{-1}] \simeq \Sp^{BC_2}_\QQ. $$
and $h$-localization is $E(0,-1)$-localization.
We can compute the $h$-periodic $C_2$-stable stems as
\begin{align*}
\pi^{C_2}_{i+j\sigma}S[h^{-1}]  
& \cong H\QQ^0(S^{i+j\sigma})^{C_2} \\
& \cong 
\begin{cases}
\QQ, & i = -j \: \text{even}, \\
0, & \text{otherwise}. 
\end{cases}
\end{align*}

\subsection*{$\pmb{v_{(-1,0)}}$-periodicity: $\pmb{w_0}$-periodicity and $\pmb{\eta}$}

The following is also observed by \cite{BGL} (this is an instance of Remark~\ref{rmk:v-1n}(3)). 

\begin{prop}
The map
$$ \eeta \in [S^\sigma,S]^{C_2} $$
is a $v_{(-1,0)}$-self map.
\end{prop}

\begin{proof}
We have
$$ \eeta^{\Phi C_2} = 2 \in \pi_0(S). $$
However, the underlying map is
$$ \eeta^{\Phi e} = \eta \in \pi_1(S), $$
which is nilpotent.
\end{proof}

Since
\begin{align*}
 S[\eeta^{-1}]^{\Phi e} & \simeq \ast, \\
 S[\eeta^{-1}]^{\Phi C_2} & \simeq S_\QQ, 
\end{align*}
we deduce that $\eeta$-localization is $E(-1,0)$-localization, and we have (c.f. \cite[Prop.~8.1]{GI})
$$ \pi^{C_2}_\star S[\eeta^{-1}] = \QQ[a^{\pm}]. $$

\subsection*{$\pmb{v_{(1,0)}}$-periodicity}

We now come to our first truly interesting example (and the only example we have of non-trivial ``mixed periodicity'').  In order to admit a $v_{(1,0)}$-self map, our complex must at least be of type $(1,0)$.  An example of such a complex is the cofiber $S/h$.
Since the underlying spectrum $S/2$ only admits a $v_1^4$-self map, we know that we cannot arrange for better on $S/h$.

\begin{prop}\label{prop:v10^4}
The spectrum $S/h$ admits a $v^4_{(1,0)}$-self map
$$ v^{4}_{(1,0)}: \Sigma^{8\sigma}S/h \to S/h $$
satisfying
$$ a^4 v^{4}_{(1,0)} = \eeta^4. $$
\end{prop}

\begin{proof}
We follow Algorithm~\ref{alg:vnm}.
$$ \eeta^4: S^{4\sigma} \to S. $$
Since the Mahowald invariant of $2^4 = (\eeta^4)^{\Phi C_2}$ is $8\sigma \in \pi_7S$ \cite{MR}, \cite{GIMI}, we deduce that 
there exists a Mahowald lift
$$ a^{-3} \eeta^4 \in \pi^{C_2}_{7\sigma}S $$
of $2^4 \in \pi_0 S$ with
$$ (a^{-3} \eeta^4)^{\Phi e} = 8\sigma. $$
To get one additional power of $a$-divisibility, we must work mod $h$ (which would kill $8\sigma$ since it is $2$-divisible).

Specifically, the $S$-module structure gives a map  
$$  a^{-3} \eeta^4 \in [S^{7\sigma}/h, S/h]^{C_2}. $$
The image of this self-map in 
$$ [S^{7\sigma}/h, S/h \wedge Ca]^{C_2} \cong [S^7/2, S/2] $$
is $8\sigma$, which is trivial, since the identity on $S/2$ is $4$-torsion.  It follows that $a^{-3}\eeta^4$ is $a$-divisible as a self-map of $S/h$, giving
$$ a^{-4}\eeta^4 \in [S^{8\sigma}/h,S/h]^{C_2}.  $$
We wish to show that this is a Mahowald lift by comparing with $\BP_{C_2}$.
A self-map 
$$ f: S^{i+j\sigma}/h \to S/h $$
gives rise to a ``Hurewicz image'' homotopy element 
$$ H(f): S^{i+j\sigma} \to S^{i+j\sigma}/h \to S/h \to \BP_{C_2}/h. $$
This gives the following diagram (using $(S/h)^{\Phi C_2} = S \vee S^1$)
$$
\xymatrix{
[S^8/2,S/2]\ar[d]_H  &
[S^{8\sigma}/h,S/h]^{C_2} \ar[r]^{\Phi C_2} \ar[d]_{H} \ar[l]_{\Phi e} & [S^0\vee S^1, S^0 \vee S^1] \ar[d]_{H} 
\\
\pi_8(\BP/2) &
\pi^{C_2}_{8\sigma}(\BP_{C_2}/h) \ar[r]^-{\Phi C_2} \ar[l]_{\Phi e} &
\pi_0(\BP^{\Phi C_2}_{C_2}\vee \Sigma^1 \BP^{\Phi C_2}_{C_2})
\\
\pi_8\BP \ar[u] &
\pi^{C_2}_{8\sigma}(\BP_{C_2}) \ar[r]^{\Phi C_2} \ar[u]^{(\ast)} \ar[l]_{\Phi e} &
\pi_0(\BP^{\Phi C_2}_{C_2})
\ar@{^{(}->}[u]
}$$

Since $u$ is invertible in $\pi^{C_2}_\star\BP_{C_2}$, we have
$$ \pi^{C_2}_{2i+2j\sigma}(\BP_{C_2}/h) \cong \pi_{2(i+j)}(\BP_{C_2}/h). $$
Since the Hurewicz image of $h$ is $q_1$, the cofiber sequence
$$ \BP_{C_2} \xrightarrow{h} \BP_{C_2} \to \BP_{C_2}/h $$
gives a long exact sequence
$$ \pi^{C_2}_{2k}\BP_{C_2} \xrightarrow{\cdot q_1} \pi^{C_2}_{2k}\BP_{C_2} \xrightarrow{(\ast)} \pi_{2k}(\BP_{C_2}/h) \to \pi^{C_2}_{2k-1}\BP_{C_2} = 0 $$
and in particular $(\ast)$ is surjective.

Let 
$$ v \in \pi^{C_2}_{8\sigma}\BP_{C_2} $$ 
be an element which maps under $(\ast)$ to $H(a^{-4}\eeta^4)$.  Then 
$$ v^{\Phi C_2} = 2^4. $$ 
By Proposition~\ref{prop:Mvn}, we have 
$$ v^{\Phi e} \dotequiv v_1^4 \mod (2). $$
It follows that 
$$ H((a^{-4}\eeta^4)^{\Phi e}) \doteq v_1^4 $$
and therefore 
$$ a^{-4}\eeta^4 : S^{8\sigma}/h \to S/h $$
is a $v_{(1,0)}$-self map.
\end{proof}

\begin{rmk}
The existence of the $v^4_{(1,0)}$-self map above was conjectured by Bhattacharya-Guillou-Li \cite[Rmk.~5.6]{BGL}.  Crabb discussed $v_{(1,0)}$ periodicity in \cite{Crabb}.  Balderrama, Hou, and Zhang independently also proved Proposition~\ref{prop:v10^4}.  This, and many additional examples of $v_{\ul{n}}$-self maps, will appear in their forthcoming paper \cite{BHZ}.  
\end{rmk}

\begin{rmk}
The existence of the $v^4_{(1,0)}$-self map above implies that $h$-torsion elements which have non-trivial image under the map
$$ \pi^{C_2}_{\star}S \to \pi^{C_2}_\star S_{K(e,1)} $$
should exhibit $v^4_{(1,0)}$-periodicity (the target of which was computed by Balderrama \cite{Balderrama}).
This explains the apparent $v_1^4$-periodic patterns in the fixed coweight charts of Guillou-Isaksen \cite{GI}.  
\end{rmk}

\subsection*{$\pmb{v_{(-1,1)}}$-periodicity: $\pmb{w_1}$-periodicity and $\pmb{\bar{\kappa}}$}

This is an instance of Remark~\ref{rmk:v-1n}(3).  The $C_2$-equivariant complex $S/\eeta$ is a type $(0,1)$-complex, with 
$$ (S/\eeta)^{\Phi C_2} \simeq S/2. $$
We seek to find a Mahowald lift of the non-equivariant self-map
$$ v_1^4 : \Sigma^8 S/2 \to S/2. $$
Note that the composite
$$ S^8 \hookrightarrow \Sigma^8 S/2 \xrightarrow{v_1^4} S/2 \to S^1 $$
(where the last map is projection onto the top cell) is $8\sigma$. 
The first author proved \cite{Behrens} that
$$ \eta^2\eta_4 \in M(8\sigma) $$
This implies that $8\sigma$ has a Mahowald lift
$$ \pmb{\eta^2 \eta_4} \in \pi^{C_2}_{7+11\sigma} \to \pi_7(S) \ne 1 $$
whose underlying class is $\eta^2\eta_4 \in \pi_{18}(S)$.

One notable feature of $\CC$-motivic, $\RR$-motivic, and $C_2$-equivariant homotopy theory is the failure of the existence of a motivic or equivariant analog of $\kappabar \in \pi_{20}S$ (detected by $g$ in the Adams spectral sequence) \cite{Isaksen}, \cite{BI}, \cite{BGI}.  This deficiency is rectified by replacing $S$ with $S/\eeta$, as we will now explain.

The computations of Guillou-Isaksen \cite{GI} imply that
$$ \eeta \cdot \pmb{\eta^2\eta_4} = 0 $$
and therefore $\pmb{\eta^2\eta_4}$ lifts over the top cell of $S/\eeta$ to give a class which we shall call
$$ \pmb{\bar{\kappa}} \in \pi^{C_2}_{8+12\sigma}(S/\eeta). $$
Our reason for giving it this name is as follows.  Combining \cite{BI} with \cite{BGI}, the element $h_0g$ in the $C_2$-equivariant spectral sequence is a permanent cycle, giving rise to an element
$$ \{ h_0g \} \in \pi^{C_2}_{8+12\sigma}S. $$

\begin{lem}\label{lem:kappabar}
Under the map induced by the inclusion of the bottom cell
$$ i_*: \pi_{8+12\sigma}^{C_2}(S) \to \pi^{C_2}_{8+12\sigma}(S/\eeta) $$
we have
$$ i_* \{h_0g\} \doteq h \cdot \kkappabar. $$  
\end{lem}

\begin{proof}
Observe that by the definition of the Toda bracket we have
$$ i_* \bra{h, \pmb{\eta^2\eta_4}, \eeta} \ni h \cdot \kkappabar $$
so we just must show that 
\begin{equation}\label{eq:toda}
 \{ h_0 g \} \in \bra{h, \pmb{\eta^2\eta_4}, \eeta} 
 \end{equation}
and properly bound the indeterminacy.  In \cite[Lem.~2.3.3, Table~16]{Isaksen}, Isaksen shows that in $\CC$-motivic cohomology of the Steenrod algebra $\Ext_{A^\CC}$ we have 
\begin{equation}\label{eq:massey}
 h_0g = \bra{h_0, h_1^3h_4, h_1} \quad \text{(no indeterminacy)}. \end{equation}
Using the map
$$ \Ext_{A^\RR} \to \Ext_{A^\CC} $$
(consulting \cite{BI} and checking indeterminacy) we see that
(\ref{eq:massey}) also holds in $\Ext_{A^\RR}$.  The computations of \cite{BI}, together with \cite{BGI}, then imply (\ref{eq:toda}) with indeterminacy $h \{h_0g\} = 2 \{h_0 g\}$.
\end{proof}

The following is a $C_2$-equivariant analog of a $\CC$-motivic theorem of Andrews~\cite{Andrews}. 

\begin{lem}
The element $\kkappabar \in \pi^{C_2}_{8+12\sigma}(S/\eeta)$ lifts to a self-map
$$ w_1^4 = \td{\kkappabar}: S^{8+12\sigma}/\eeta \to S/\eeta. $$
\end{lem}

\begin{proof}
The cofiber sequence
$$ S^\sigma \xrightarrow{\eeta} S \xrightarrow{i} S/\eeta \xrightarrow{p} S^{1+\sigma} $$
gives rise to a diagram
\[\xymatrix{
&  &
[S^{8+13\sigma},S]^{C_2} \ar[d]^{i_*} \\
[S^{8+12\sigma}/\eeta, S/\eeta]^{C_2} \ar[r]^{i^*} & 
[S^{8+12\sigma}, S/\eeta]^{C_2} \ar[r]^{\eeta^*} \ar[d]^{p_*} & 
[S^{8+13\sigma}, S/\eeta]^{C_2} \ar[d]^{p_*} \\
& [S^{7+11\sigma}, S]^{C_2} \ar[r]^{\eeta^*} &
[S^{7+12\sigma}, S]^{C_2}
}\]
where the middle row and right column are exact. To produce $\td{\kkappabar}$ it suffices to show that 
$$ \eeta^*\kkappabar = 0. $$
We have
$$ p_* \eeta^* \kkappabar = \eeta^* p_* \kkappabar = \eeta \cdot \pmb{\eta^2\eta_4} = 0 $$
so there exists a 
$$ y \in [S^{8+13\sigma},S]^{C_2} $$
so that 
$$ i_*y = \eeta^* \kkappabar. $$
However, the computations of Guillou-Isaksen \cite{GI} show that $\pi^{C_2}_{8+13\sigma}S = 0$.\footnote{The Guillou-Isaksen computations only go up to coweight $7$, but they compute the $E_3$-page of the equivariant Adams spectral sequence in coweight $8$, and this is already trivial in total degree $21$.}  
\end{proof}

\begin{lem}
The self-map 
$$ \td{\kkappabar}: S^{8+12\sigma}/\eeta \to S/\eeta $$
is a $v^4_{(-1,1)}$ self map.
\end{lem}

\begin{proof}
The fact that the geometric fixed points 
$$ \td{\kkappabar}^{\Phi C_2} : S^{8}/2 \to S/2 $$
is a $v_1^4$-self map follows from the fact that the composite
$$ S^8 \hookrightarrow S^8/2 \xrightarrow{\td{\kkappabar}^{\Phi C_2}} S/2 \twoheadrightarrow S^1 $$
is $(\pmb{\eta^2\eta_4})^{\Phi C_2} = 8\sigma$.  Since $(S/\eeta)^e = S/\eta$ is a type $0$ complex, and $\td{\kkappabar}^{\Phi e}$ is torsion, it has to be nilpotent.
\end{proof}

Lemma~\ref{lem:kappabar} then implies that $\td{\kkappabar}^{\Phi e}$ is non-trivial, which, together with an analysis of $a$-divisibility using the computations of \cite{BI}, \cite{BGI}, \cite{GI}, shows that $\td{\kkappabar}$ is a Mahowald lift.

\begin{rmk}
Andrews introduced $\kkappabar$-periodicity in the $\CC$-motivic context in \cite{Andrews}, where he calls it \emph{$w_1^4$-periodicity}, and Quigley studies $w_1^4 = \td{\kkappabar}$-periodicity in the $\RR$-motivic and $C_2$-equivariant context in \cite{Quigley}.  In particular, he constructs examples of infinite $w_1$-periodic families in the $C_2$-equivariant stable stems.
\end{rmk}

\begin{rmk}
D.~Isaksen has also independently suggested a connection between $\kkappabar$-periodicity and $v_1$-periodicity on geometric fixed points.
\end{rmk}

\subsection*{$\pmb{v_{(2,1)}}$-periodicity?}

The next step would be to construct a $v_{(2,1)}$-self map on
$$ S/(h, v^4_{(1,0)}) $$
Since the minimal $v_2$-self map that $S/(2,v_1^4)$ possesses is a $v_2^{32}$-self map \cite{BHHM}, we propose the following optimistic conjecture.

\begin{conjecture}\label{conj:v232}
The complex $S/(h, v^4_{(1,0)})$ has a $v^{32}_{(2,1)}$-self map
$$ v^{32}_{(2,1)}: S^{64+128\sigma}/(h, v^4_{(1,0)}) \to S/(h, v^4_{(1,0)}) $$
such that 
$$ a^{32}v_{(2,1)}^{32} = \td{\kkappabar}^{8}. $$
\end{conjecture}

\begin{rmk}
P.~Bhattacharya suggested to the second author that there is an interesting connection of Proposition~\ref{prop:v10^4} and Conjecture~\ref{conj:v232} to nilpotence orders.  Indeed, Proposition~\ref{prop:v10^4} is related to the fact that $\eeta^4$ is $a$-divisible, and this implies
$$ \eta^4 = (\eeta^4)^{\Phi C_2} = 0 \in \pi_4 S. $$
Note that this is precisely the nilpotence order of $\eta \in \pi_1 S$.
The authors do not know the nilpotence order of $\kappabar \in \pi_{20}S$, but as Conjecture~\ref{conj:v232} is related to the $a$-divisibility of $\td{\kkappabar}^8$, it would imply
$$ \kappabar^8 = 0 \in \pi_{160}(S/\eta). $$
Note that $\kappabar^6 \in \pi_{120}(S)$ was shown to be non-trivial in \cite{BHHM2}, and we don't know whether $\kappabar^7$ is non-trivial.\footnote{Since the writing of the first version of this paper, the remarkable computer computations of \cite{LWX} have shown that $\kappabar^7 = 0$.}
\end{rmk}

\subsection*{$\pmb{v_{(n,n-1)}}$-periodicity and Mahowald redshift}

Mahowald and Ravenel have conjectured that the classical classical Mahowald invariant takes $v_{n-1}$-periodic families to $v_n$-periodic families \cite[Conj.~12]{MRglobal}, and this has been verified in numerous specific cases \cite{MS}, \cite{MR}, \cite{Sadofsky}, \cite{Behrensroot}, \cite{Behrens}.  We describe how this conjectural phenomenon is related to $v_{(n,n-1)}$-periodicity.

Suppose that $x \in \pi_{i}(S)$, and let $\td{x}$ be a Mahowald lift with corresponding Mahowald invariant $\bar{x} \in M(x)$.
\[
\xymatrix@R-2em{
\pi_{i+j}(S) & \pi_{i+j\sigma}^{C_2}(S) \ar[r]^{\Phi C_2} \ar[l]_{\Phi e} & \pi_i(S) \\
\bar{x} & \quad \td{x} \quad \ar@{|->}[r] \ar@{|->}[l] & {x} 
}
\]
Suppose that $\td{x}$ is $v_{(m,m-1)}$-torsion for $m < n$, in the sense that there inductively exist $v_{(m,m-1)}$-self maps and cofiber sequences
$$ S^{b_m+c_m\sigma}/I_{m} \xrightarrow{v^{k_m}_{(m,m-1)}} S/I_{m} \to S/I_{m+1} $$
and a lift
$$ \td{y} \in \pi_{i+d+(j+e)\sigma}^{C_2}(S/I_n) \xrightarrow{p_*} \pi_{i+d+(j+e)\sigma}^{C_2}(S^{d+e\sigma}) $$
(where $p: S/I_n \to S^{d+e\sigma}$ is projection to the top cell).
Suppose that $S/I_n$ has a $v_{(n,n-1)}$-self map
$$ v^k_{(n,n-1)}: S^{b+c\sigma}/I_n \to S/I_n $$
Then we have
\begin{align*}
 (S/I_n)^{\Phi e} & \simeq S/(2^{j_0}, \ldots, v^{j_{n-1}}_{n-1}), \\
 (S/I_{n})^{\Phi C_2} & \simeq S/(2^{j_1}, \ldots, v^{j_{n-1}}_{n-2}) \vee S^1/(2^{j_1}, \ldots, v^{j_{n-1}}_{n-2})
\end{align*}
and
\begin{align*}
(v^k_{(n,n-1)})^{\Phi e} & \simeq v_n^k, \\
(v^k_{(n,n-1)})^{\Phi C_2} & \simeq v_{n-1}^k.
\end{align*}
Consider the following diagram (for $s \ge 0$).
\[
\xymatrix{
\pi_{i+j}S &
\pi^{C_2}_{i+j\sigma}S \ar[l]_{\Phi e} \ar[r]^{\Phi C_2} &
\pi_i S 
\\
\pi_{i+j+d+e}(S/I_n)^{e} \ar[u]_{p^{\Phi e}_*} \ar[d]^{v^{sk}_{n}} &
\pi^{C_2}_{i+d+(j+e)\sigma}(S/I_n) \ar[l]_{\Phi e} \ar[r]^{\Phi C_2} \ar[u]_{p_*} \ar[d]^{v^{sk}_{(n,n-1)}}&
\pi_{i+d} (S/I_n)^{\Phi C_2} \ar[u]_{p^{\Phi C_2}_*} \ar[d]^{v^{sk}_{n-1}}
\\
\pi_{i+j+d+e+s(b+c))}(S/I_n)^{e} \ar[d]^{p^{\Phi e}_*}  &
\pi^{C_2}_{i+d+sb+(j+e+sc)\sigma}(S/I_n) \ar[l]_-{\Phi e} \ar[r]^-{\Phi C_2} \ar[d]^{p_*} &
\pi_{i+d+sb} (S/I_n)^{\Phi C_2} \ar[d]^{p^{\Phi C_2}_*} 
\\
\pi_{i+j+s(b+c))}S   &
\pi^{C_2}_{i+sb+(j+sc)\sigma}S \ar[l]_{\Phi e} \ar[r]^{\Phi C_2}  &
\pi_{i+sb} S  
}
\]
Suppose: 
\begin{enumerate}[label=(\theequation)]
\eitem The elements
$$ v^{sk}_{n-1} x := p_*^{\Phi C_2} v^{sk}_{n-1} \td{y}^{\Phi C_2} \in \pi_*S $$
are all non-trivial.  Then $x$ is $v_{n-1}$-periodic, and $v^{sk}_{n-1}x$ is the $v_{n-1}$-periodic family generated by $x$.

\eitem
The elements 
$$ p_* v^{sk}_{(n,n-1)} \td{y} \in \pi^{C_2}_\star S$$ 
are Mahowald lifts of the $v_{n-1}$-periodic family $v^{sk}_{n-1}x$.
\label{item:MLhyp}
\end{enumerate}
Then:
\begin{enumerate}[label=(\theequation)]
\eitem The elements
$$ v^{sk}_{n} \bar{x} := p_*^{\Phi e} v^{sk}_{n} \td{y}^{\Phi e} \in \pi_*S $$
are non-trivial and form a $v_n$-periodic family.  In particular, $\bar{x} \in M(x)$ is 
$v_n$-periodic.
\eitem The Mahowald invariant takes the $v_{n-1}$-periodic family generated by $x$ to the $v_n$-periodic family generated by $\bar{x}$:  
$$ v^{sk}_n \bar{x} \in M(v^{sk}_{n-1}x). $$

\eitem The elements
$$ v^{sk}_{(n,n-1)}\td{x} := p_* v^{sk}_{(n,n-1)} \td{y} \in \pi^{C_2}_\star S $$
are non-trivial, so form the \emph{$v_{(n,n-1)}$-periodic family generated by $\td{x}$}, and in particular $\td{x}$ is \emph{$v_{(n,n-1)}$-periodic}.

\end{enumerate} 

\begin{ex}
Mahowald lifts of
$$ 2,2^2, 2^3, 2^4 \in \pi_0(S) $$
are given by $h$-torsion elements
$$ \eeta, \eeta^2, \eeta^3, a^{-3}\eeta^4 $$
with corresponding Mahowald invariants
$$ \eta, \eta^2, \eta^3, 8\sigma \in \pi_*S. $$
The $v_{(1,0)}$-self map
$$ v^4_{(1,0)}: S^{8\sigma}/h \to S/h $$ 
of Proposition~\ref{prop:v10^4} generates $v^4_{(1,0)}$-periodic families 
$$ v^{4k}_{(1,0)}\eeta, \: v^{4k}_{(1,0)}\eeta^2, \: v^{4k}_{(1,0)}\eeta^3, \: v^{4k}_{(1,0)}a^{-3}\eeta^4 \in \pi^{C_2}_\star S $$  
which witness Mahowald invariants
\begin{align*}
v_1^{4k}\eta & \in M(2^{4k+1}), \\
v_1^{4k}\eta^2 & \in M(2^{4k+2}), \\
v_1^{4k}\eta^3 & \in M(2^{4k+3}), \\
v_1^{4k}8\sigma & \in M(2^{4k+4}).
\end{align*}
\end{ex}

\begin{ex}
Take $j_0 = 1$ and $v_{(0,-1)} = h$.  If inductively the self-maps $v^{j_m}_{(m,m-1)}$ exist for $m \le n$, and the appropriate form of \ref{item:MLhyp} holds, then the argument above shows that
$$ \alpha^{(n)}_{sj_n/j_{n-1},\ldots, j_1} \in M(\alpha^{(n-1)}_{sj_{n}/j_{n-1}, \ldots, j_1}) $$
where $\alpha^{(k)}$ refers to the $k$th Greek letter construction (see \cite{MRW}).
\end{ex}

Quigley has studied $C_2$-equivariant generalizations of Mahowald invariants and Mahowald redshift \cite{Quigley}.  

\begin{question}
Are Quigley's $C_2$-equivariant Mahowald invariant computations related to $C_2 \times C_2$-equivariant chromatic homotopy theory?
\end{question}

\bibliographystyle{amsalpha}
\bibliography{MUG}

\end{document}